\documentclass[11pt]{amsart}
\usepackage{fullpage}
\usepackage{amsfonts,amssymb,amsmath,amsthm}
\usepackage{hyperref}
\usepackage{color}

\theoremstyle{plain}
\newtheorem{theorem}{Theorem}[section]
 
  \newtheorem{example}[theorem]{Example}
 \newtheorem{lemma}[theorem]{Lemma}
 \newtheorem{proposition}[theorem]{Proposition}
 \theoremstyle{definition}
 \newtheorem{definition}[theorem]{Definition}
 
 \theoremstyle{remark}
 \newtheorem{remark}[theorem]{Remark}
 \numberwithin{equation}{section}

\def\eps{\varepsilon}
\def\Op{\text{Op}}
\def\Id{\text{\rm Id}}
\def\R{{\mathbb R}}
\def\N{{\mathbb N}}
\def\Z{{\mathbf Z}}
\def\T{{\mathbf T}}

\def\Tend#1#2{\mathop{\longrightarrow}\limits_{#1\rightarrow#2}}

\numberwithin{equation}{section}

\begin{document}
\title[]{Observability and controllability\\
for the Schr\"odinger equation\\
on quotients of groups of Heisenberg type}
\author[]{Clotilde Fermanian Kammerer}
\address[C. Fermanian Kammerer]{Universit\'e Paris Est Cr\'eteil, LAMA, 61, avenue du G\'en\'eral de Gaulle\\
94010 Cr\'eteil Cedex\\ France}
\email{clotilde.fermanian@u-pec.fr}
\author[]{Cyril Letrouit}
\address[C. Letrouit]{Sorbonne Universit\'e, Universit\'e Paris-Diderot, CNRS, Inria, Laboratoire Jacques-Louis Lions, 75005 Paris, France \newline DMA, \'Ecole normale sup\'erieure, CNRS, PSL Research University, 75005 Paris, France}
\email{letrouit@ljll.math.upmc.fr}

\medskip

\maketitle

\begin{abstract}
We give necessary and sufficient conditions for the controllability of a Schr\"odinger equation involving  the sub-Laplacian of a nilmanifold  obtained by taking the quotient of a  group of Heisenberg type by one of its discrete sub-groups.
This class of nilpotent Lie groups is a major example of stratified Lie groups of step 2. The sub-Laplacian involved in these Schr\"odinger equations is subelliptic, and, contrarily to what happens for the usual elliptic Schr\"odinger equation for example on flat tori or on negatively curved manifolds, there exists a minimal time of controllability. The main tools used in the proofs are (operator-valued) semi-classical measures constructed by use of representation theory and a notion of semi-classical wave packets that we introduce here in the context of groups of Heisenberg type.
\end{abstract}

 \tableofcontents

\section{Introduction} \label{s:intro}

In this paper, we consider a {\it nilmanifold}~$M$, that is a
 manifold $M=\widetilde{\Gamma}\backslash G$ which  is the left quotient of a nilpotent Lie group $G$  by a discrete cocompact subgroup $\widetilde{\Gamma}$ of $G$.  We assume here that the Lie group $G$, as a differential manifold, is  an {\it H-type group (also called ``group of Heisenberg type'')}. On the manifold~$M$, we consider the {\it sub-Laplacian} $-\Delta_M$ and  we are interested in the Schr\"odinger operators   $-\frac12\Delta_M-\mathbb V$ for analytic potentials $\mathbb V$. We  study the controllability and the observability of the associated Schr\"odinger equation on~$M$  thanks to  the Harmonic analysis properties of the group~$G$.
 We give in the next section precise definitions about these notions and develop concrete 
  examples  in Section~\ref{s:exampleinheis}.

 \subsection{The nilmanifold $M$ and the Schr\"odinger equation}
 An H-type group~$G$ is  a connected and simply connected nilpotent Lie group whose Lie algebra is an H-type algebra, denoted by $\mathfrak g$.
This means that:
\begin{itemize}
\item $\mathfrak g$ is a step~$2$ stratified Lie algebra: it is equipped with a vector space decomposition
$$
  \displaystyle
  {\mathfrak g}=  \mathfrak v \oplus \mathfrak z \, ,
  $$
  such that $[{\mathfrak v},{\mathfrak v}]= {\mathfrak z}\not=\{0\}$ and ${\mathfrak z}$ is the center of ${\mathfrak g}$. 
  \item 
  $\mathfrak g$ is endowed with a scalar product $\langle \cdot, \cdot\rangle$ such that, for all $\lambda\in\mathfrak z^*$,
  the skew-symmetric map
\begin{equation*}
J_\lambda:\mathfrak{v}\rightarrow \mathfrak{v}
\end{equation*}
defined by
\begin{equation} \label{e:Jlambda}
\langle J_\lambda(U),V\rangle= \lambda([U,V]) \ \ \ \forall U,V\in \mathfrak{v}
\end{equation}
satisfies $J_\lambda^2= -|\lambda|^2 {\rm Id}$. In other words, $J_\lambda$ is an orthogonal map as soon as $|\lambda|=1$. Here, to define $|\lambda|$, we first identify $\mathfrak z^*$ to $\mathfrak z$ thanks to $\langle\cdot,\cdot\rangle$, then we define $|\lambda|$ as the norm (deriving from $\langle\cdot,\cdot\rangle$) of the image of $\lambda$ through this identification.
\end{itemize}
The Lie group $G$, as a differential manifold, is diffeomorphic to $\R^{2d+p}$, where $p$ is the dimension of the center of the group. H-type groups were introduced in \cite{Kap80}, the main motivation being that the sub-Laplacians in these groups  admit explicit fundamental solutions of an elementary form. 
The Heisenberg groups~$\mathbf{H}^d\sim \R^{2d+1}$ are examples of H-type groups (with $p=1$), as will be recalled below. 

\medskip 

We consider $\widetilde{\Gamma}$, a discrete cocompact subgroup of $G$. A concrete example is given  in Example~\ref{ex:general1}. Then, we set $M=\widetilde{\Gamma}\backslash G$.

\medskip

  Via the exponential map
 $$
{\rm Exp} :  {\mathfrak g} \rightarrow G
 $$
 which is a diffeomorphism from ${\mathfrak g}$ to $G$,
 one identifies $G$ and ${\mathfrak g}$ as sets and manifolds. We may identify ${\mathfrak g}$ with the space of left-invariant vector fields via
 \begin{equation}\label{eq:Xf}
 Xf(x)= \left.{d\over dt} f(x{\rm Exp}(tX))\right|_{t=0},
 \end{equation}
 which acts on functions of $x\in G$ and on functions of $x\in M$ since it passes to the quotient.
  Choosing an orthonormal basis~$(V_j)_{1\leq j\leq 2d}$ of ${\mathfrak v}$ and identifying $\mathfrak g$ with the Lie algebra of left-invariant vector fields on $G$,
  one defines the sub-Laplacian
    $$
  {\Delta}_M= \sum_{j=1}^{2d} V_j^2,
  $$
on $M$, where ${\rm dim}\, {\mathfrak v}=2d$. Note that this makes sense since the $V_j$ are left-invariant, and thus pass to the quotient.

\medskip

We consider the
 hypoelliptic second order equation (see~\cite{Hor67})
  \begin{equation}  \label{e:Schrod}
i\partial_t \psi+\frac 12 \Delta_M \psi  +\mathbb{V} \psi =0
\end{equation}
on $M$, where $\mathbb{V}$ is an analytic function defined on~$M$ (the latter assumption could be relaxed as soon as a unique continuation principle holds for~$\frac12\Delta_M+\mathbb V$, see Remark~\ref{rem:analyticity} below).

\subsection{Examples of nilmanifolds}\label{s:exampleinheis}
Let us describe now an example of a quotient manifold~$M$ to which our result will apply. It is known (see \cite[Theorem 18.2.1]{BLU}, and also~\cite{BFG2}) that any H-type group is isomorphic to one of the ``prototype H-type groups'', which are defined as follows:
let $P^{(1)},\ldots,P^{(p)}$ be $p$  linearly independent $2d \times 2d$ orthogonal skew-symmetric matrices satisfying the property
$$P^{(r)}P^{(s)}+P^{(s)}P^{(r)}=0,\;\;
\forall r,s \in \left \{1,...,p \right \},\;\;r \neq s.$$
Let us denote by 
$  ({w},s)=({w_1,\cdots,w_{2d}},s_1,\cdots,s_p)$
the points of  $\R^{2d+p}$,  that is endowed with the group law 
\begin{equation*}
\quad \quad \quad ({w},s) \cdot ({w'},s'):=\begin{pmatrix}
w+w' \\s_j+s_j'+\frac12 \langle w, P^{(j)}w' \rangle,\,\,\,j=1,...,p
\end{pmatrix}
\end{equation*}
This defines a Lie group with a  Lie algebra of left invariant vector fields
 spanned by the following vector fields: for~$
j$ running  from~1 to~$2d$ and~$k$ from~$ 1$ to~$p$,
\begin{equation*}
 X_j\! :=\!\partial_{w_j} +  \frac 1 2 \sum^{p}_{k=1}\sum^{2d}_{l=1}w_l \, P_{l,j}^{(k)}\partial_{s_k},  \ \ \ \ \text{and} \ \ \ \partial_{s_k} .
\end{equation*}
For more explicit examples of H-type groups,  see \cite[Section 18.1]{BLU} (e.g., Example 18.1.3). It includes the Heisenberg group $\mathbf{H}^d$ (of dimension $2d+1$), but also groups with a center of dimension $p>1$.

\medskip 

In this representation, the  Heisenberg group~$\mathbf{H}^d$ corresponds to $p=1$ and the  choice of 
$$\displaystyle{P^{(1)}= \begin{pmatrix} 0 & {\bf 1}_{\R^d} \\ - {\bf 1}_{\R^d} & 0\end{pmatrix}}.$$  The group law then is
\begin{equation*}
\quad \quad \quad (x,y,s) \cdot (x',y',s'):=\begin{pmatrix}
x+x'\\ y+y'\\s+s'+\frac12 \sum_{j=1}^d(x_jy_j'-x_j'y_j)
\end{pmatrix}
\end{equation*}
where $x,y,x',y'\in\R^d$ and $s,s'\in\R$. We define the scalar product on $\mathfrak v$ by saying that the $2d$ vector fields
\begin{equation}\label{def:XY}
X_j=\partial_{x_j}-\frac{y_j}{2}\partial_{s}, \ \  \ Y_j=\partial_{y_j}+\frac{x_j}{2}\partial_{s}, \qquad j=1,\ldots,d
\end{equation}
form an orthonormal basis, and we define the scalar product on $\mathfrak z$ by saying that $\partial_s$ has norm $1$ (and $\mathfrak v$ and $\mathfrak z$ are orthogonal for the scalar product on $\mathfrak g$). Then we obtain
\begin{equation*}
J_\lambda\left(\sum_{j=1}^d (a_jX_j+b_jY_j)\right)=\lambda\sum_{j=1}^d (-b_jX_j+a_jY_j).
\end{equation*}
where $J_\lambda$ has been introduced in \eqref{e:Jlambda}.

\begin{example}\label{ex:general1}
An example of discrete cocompact subgroup of the Heisenberg group $\mathbf{H}^d$ is
\begin{equation} \label{e:cocompact}
\widetilde{\Gamma}_0=(\sqrt{2\pi}\mathbb{Z})^{2d}\times \pi\mathbb{Z},
\end{equation}
and the associated quotient manifold is the left quotient $M_0=\widetilde{\Gamma}_0\backslash \mathbf{H}^d$. The manifold $M_0$ is  a circle bundle over the $2d$-torus $\T^{2d}$,   its fundamental group is $\widetilde{\Gamma}_0$ which is non-commutative, implying that~$M_0$ is not homeomorphic to a torus.
For more general examples of discrete cocompact subgroups in H-type groups, see \cite[Chapter 5]{CG}.
\end{example}

\subsection{Controllability and observability, geometric conditions}
  
  One says that the Schr\"odinger equation \eqref{e:Schrod} is {\it controllable} in time $T$ on the measurable set $U\subset M$ if for any
 $u_0,u_1\in L^2(M)$, there exists $f\in L^2((0,T)\times M)$ such that the solution $\psi\in L^2((0,T)\times M)$ of $$i\partial_t \psi+\frac 12 \Delta_M \psi+\mathbb{V}\psi=f\bold{1}_U$$ (where $\bold{1}_U$ denotes the characteristic function of $U$) with initial condition $\psi(0,x)=u_0(x)$ satisfies $\psi(T,x)=u_1(x).$
By the Hilbert Uniqueness Method (see~\cite{JLL}), it is well-known that controllability is equivalent to an observability inequality.

\medskip

The Schr\"odinger equation \eqref{e:Schrod} is said to be {\it observable} in time $T$ on the measurable set $U$ if there exists a constant $C_{T,U}>0$ such that
\begin{equation}\label{obs}
\forall u_0\in L^2(M),\;\;\|u_0\|^2_{L^2(M)} \leq C_{T,U} \int_0^T \left\|  {\rm e}^{it(\frac12 \Delta_M+\mathbb{V})} u_0\right\|^2_{L^2(U)} dt.
\end{equation}

\medskip

For the usual (Riemannian) Schr\"odinger equation, it is known that if the so-called {\it  Geometric Control Condition} is satisfied in some time $T'$ (which means that any ray of geometric optics  enters~$U$ within time~$T'$), then observability, and thus controllability, hold in any time $T>0$ (see \cite{Leb}). Much less is known about the converse implication, due to curvature effects.

\medskip

Our main result
gives a similar condition, replacing the rays of geometric optics by the curves of the flow map on~$M\times \mathfrak z^*$:
$$\Phi^s_0: (x,\lambda)\mapsto ({\rm Exp} (sd\mathcal Z^{(\lambda)}/2)x,\lambda),$$
where, for $\lambda= \sum_{1\leq j\leq d } \lambda_j  Z_j^*\in\mathfrak z^*$ (where $(Z_j^*)_{1\leq j\leq p}$ is the dual basis  of the basis  $(Z_1,\cdots ,Z_p)$ of $\mathfrak z$),  $\mathcal Z^{(\lambda)}$ is the element of~$\mathfrak z$ defined by 
$\mathcal Z^{(\lambda)} = \sum_{1\leq j\leq p} \frac{ \lambda_j}{|\lambda| }Z_j$. Equivalently, $\mathcal Z^{(\lambda)}= \lambda/|\lambda|$ after identification of $\mathfrak z$ and $\mathfrak z^*$. 
%
 Note that the integral curves of this flow are transverse to the space spanned by the $V_j$'s.
We introduce the following H-type geometric control condition.

$ $

{\bf (H-GCC)} The measurable set $U$ satisfies  \textbf{H-type GCC}  in time $T$ if
$$\forall  (x,\lambda)\in M\times( \mathfrak z^*\setminus\{0\}),\;\; \exists s\in (0,T),\;\;\Phi^{s}_0 ((x,\lambda))\in U\times \mathfrak z^*.$$

\begin{definition}
We denote by $T_{\rm GCC}(U)$ the infimum of all $T>0$ such that H-type GCC holds in time~$T$ (and we set $T_{\rm GCC}(U)=+\infty$ if H-type GCC does not hold in any time).
\end{definition}

\medskip

In the sequel, we will also consider an additional  assumption (A). 
To give a rigorous statement, we write the coordinates $v=(v_1,\ldots,v_{2d})$ of a vector in the orthonormal basis $V=(V_1,\ldots, V_{2d})$ of~$\mathfrak v$:
$$
V=v_1V_1 +\ldots + v_{2d} V_{2d}\in \mathfrak v.
$$ 
Given $\omega\in\mathfrak v^*$, we write $\omega_j$ for the coordinates of $\omega$ in the dual basis of $V$, and we write $|\omega|=1$ when $\sum_{j=1}^{2d}\omega_j^2=1$.
$ $
\begin{itemize}
\item[{\bf (A)}]   For any $(x,\omega)\in M\times \mathfrak v^*$ such that $|\omega|=1$, there exists $s\in\R$ such that 
$${\rm Exp} \bigl(s\sum_{j=1}^{2d}\omega_jV_j \bigr) x \in U.$$
\end{itemize}
$ $
 Note that this condition is independent of the choice of the basis $V$.
 
\begin{example}\label{ex:general2}
Let us compute the flows involved in the above conditions in the context of Example~\ref{ex:general1}. Denoting by $(x,y,t)$ the  elements of $M_0$,
$$\Phi^s_0 (x,y,t,\lambda)=\left(x,y, t+s\frac d2 \,{\rm sgn} (\lambda), \lambda\right),\;\;s\in\R$$
and choosing the basis $V=(X_1,\cdots,X_d,Y_1,\cdots , Y_d) $ of $\mathfrak v$ (see~\eqref{def:XY}), 
$${\rm Exp} \bigl(s\sum_{j=1}^{d}(a_jX_j +b_jY_j)\bigr) (x,y,t)= \left(x+ sa,y+sb,t+\frac s2(x\cdot b-y\cdot a)  \right),\;\;s\in\R.$$
These trajectories are the lifts in $\mathbf{H}^d$ of the geodesics of $\mathbf{T}^{2d}$.
A typical open set $U\subset \widetilde{\Gamma}_0\backslash \mathbf{H}^d$ of control which one may consider is the periodization of the complementary of a closed ball in a fundamental domain:
$$A=M\setminus (\widetilde{\Gamma}_0 \cdot B)$$
where $B\subset  [0,\sqrt{2\pi})^{2d}\times [0,\pi)$ is a closed ball (for the Euclidean norm for example) whose radius is strictly less than $\pi$. Note that in the definition of $A$, the symbol $\setminus$ stands for the difference of two sets, and not for the quotient.
One can also 
 verify that both Assumption (A) and (H-GCC) (in sufficiently large time, which depends on $I$) are satisfied. 
\end{example}

\subsection{Main result}  \label{s:mainresfkl}

With these geometric definitions, we are able to state conditions for observability and thus controllability of the subelliptic Schr\"odinger equation with analytic potential on H-type nilmanifolds. 

\begin{theorem}\label{t:main} Assume that the potential $\mathbb V$ in \eqref{e:Schrod} is analytic. Let $U\subset M$ be open and denote by $\overline{U}$ its closure.
\begin{enumerate}
\item Assume that $U$ satisfies {\bf (A)} and that $T> T_{\rm GCC}(U)$, then the observability inequality \eqref{obs} holds, i.e. the Schr\"odinger equation~\eqref{e:Schrod} is observable  in time $T$ on~$U$  and thus~\eqref{e:Schrod} is controllable in time $T$ on $U$. 
\item Assume $T\leq T_{\rm GCC}(\overline{U})$, then the observability inequality \eqref{obs} fails, and thus the controllability in time $T$ also fails on $U$.
\end{enumerate}
\end{theorem}

Although this will be commented more thoroughly in Remark \ref{r:assumpA}, let us already say that the authors conjecture that  the observability inequality \eqref{obs} holds in $U$ at time $T$ under the only condition that $T>T_{\rm GCC}(U)$ (and thus one could avoid using Assumption (A)). We also point out Remark~\ref{rem:analyticity} about the  assumption that  the potential is analytic. Finally, we notice that in general $T_{\rm GCC}(U)\neq T_{\rm GCC}(\overline{U})$. This is due to the possible existence of ``grazing rays'', see Remark \ref{r:grazing} for more comments on this issue.

 \medskip

The existence of a minimal time of control in Theorem \ref{t:main} contrasts strongly with the observability in arbitrary small time, under Geometric Control Condition, of the usual elliptic Schr\"odinger equation (see \cite{Leb}), which is related to its ``infinite speed of propagation''. In the subelliptic setting which we consider here (meaning that $\Delta_M$ is subelliptic but not elliptic), in the directions defined by $\mathfrak z$, the Schr\"odinger operator has a very different behaviour, possessing for example a family of travelling waves moving at speeds proportional to $n\in\N$, as was first noticed in \cite[Section 1]{BGX} (see also \cite[Theorem 2.10]{FF}).  The existence of a minimal time of observability for hypoelliptic PDEs was first shown in the context of the heat equation: for instance the case of the heat equation with Heisenberg sub-Laplacian has been investigated in \cite{BC} and the case of the heat equation with ``Grushin'' sub-Laplacian has been studied in \cite{Koe}, \cite{DK} and \cite{BDE}.

\medskip

More recently, in \cite{BS}, it was shown that the Grushin Schr\"odinger equation $i\partial_{t}u-\partial_{x}^2u-x^2\partial_{y}^2u=0$ in $(-1,1)_x\times \mathbb{T}_y$ is observable on a set of horizontal strips if and only the time $T$ of observation is sufficiently large. With related ideas, it is shown in \cite{LS} that the observability of the Grushin-type Schr\"odinger equation $i\partial_{t}u+(-\partial_{x}^2-|x|^{2\gamma}\partial_{y}^2)^su=0$ in $(-1,1)_x\times \mathbb{T}_y$ (with observation on the same horizontal strips as in \cite{BS}) depends on the value of the ratio $(\gamma+1)/s$: observability may hold in arbitrarily small time, or only for sufficiently large times, or even never hold if $(\gamma+1)/s$ is large enough. These results share many similarities with ours, although their proofs use totally different techniques. Finally, in contrast with the usual ``finite time of observability'' of elliptic waves (under GCC), it was shown in \cite{Let} that subelliptic wave equations are never observable.
We can roughly summarize all these results by saying that \emph{the subellipticity of the sub-Laplacian slows down the propagation of evolution equations in the directions needing brackets to be generated.}

\medskip

The proof of Theorem \ref{t:main} is based on  adapting standard semi-classical approach to prove observability for a class of Schr\"odinger equations with \emph{subelliptic} Laplacian, through the use of the operator-valued semi-classical measures of~\cite{FF} which are adapted to this stratified setting. The proof also uses the introduction of wave packets playing  in this non-commutative setting a role similar to the ones introduced in~\cite{CR} and~\cite{hag} in the Euclidean case. To say it differently, we follow the usual scheme for proving or disproving observability inequalities, but with all the analytic tools (i.e., pseudodifferential operators, semiclassical measures and wave packets) adapted to our subelliptic setting: we do not use, for instance, classical pseudodifferential operators.

\subsection{Strategy of the proof} 
The theorem consists in two parts: firstly that the condition {\bf (A)} guarantees that the observability holds when $T>T_{\rm GCC}(U)$ and, secondly, that the observability fails when $T\leq T_{\rm GCC}(\overline{U})$. Beginning with the first part, 
 it is standard  (see \cite{Leb}) to start with  a {\it localized observability} result as stated in the next lemma. 

 \begin{lemma}[Localized observability]\label{lem:locobs}
 Assume the set $U$ satisfies assumption {\bf (A)} and that~{\bf (H-GCC)} holds in time $T$ for $U$. Let $h>0$ and $\chi\in C_c^\infty((1/2,2),[0,1])$. Using functional calculus, we set
 \begin{equation}\label{def:Pih}
 \mathcal{P}_h f= \chi\left(-h^2\left(\frac12 \Delta_M+\mathbb{V}\right)\right) f,\;\; f\in L^2(M).
 \end{equation}
Then, there exists a constant $C_0>0$ such that for any sufficiently small $h>0$ and any $u_0\in L^2(M)$,
\begin{equation}\label{obs_loc}
\| \mathcal{P}_h u_0\|^2_{L^2(M)} \leq C_0 \int_0^T \left\|   {\rm e}^{it(\frac12 \Delta_M+\mathbb{V})}  \mathcal{P}_h u_0\right\|^2_{L^2(U)} dt.
\end{equation}
 \end{lemma}

 \begin{remark} \label{r:masscons}
By conservation of mass in the LHS (and invariance of H-type GCC by translation in time), this inequality also holds when the integral in the RHS is taken over an arbitrary time interval $(T_1,T_2)$ such that $T_2-T_1\geq T$.
\end{remark}

The proof of the localized observablity  is done in Section~\ref{sec:proof} below.
The argument is by contradiction (as in~\cite{BZ} or~\cite[Section 7]{AM14}) and it
 uses the semi-classical setting based on representation theory and developed in~\cite{FF1,FF} that we extend to the setting of quotient manifolds in Section~\ref{sec:semiclas}.
  In particular, this argument relies in a strong way on the operator-valued semi-classical measures constructed in Sections \ref{s:sclmeas} and \ref{s:sclmeastime}.

\medskip 

The role of semiclassical measures in the context of observability estimates was first noticed by Gilles Lebeau~\cite{lebeau}
and has been
widely used since then~\cite{MaciaTorus,AM14,AFM15,MacRiv18}, with all the developments of semi-classical measures, especially two-scale (also called two-microlocal) semi-classical measures that allow to analyze more precisely the concentration of families on submanifolds. These two-scale measures introduced in the  end of the 90-s (see~\cite{Fermanian_note1,Fermanian_Note2,FG02, NierScat, MillerThesis}) have known since then a noticeable development in control theory (see the survey~\cite{MaciaLille}) and in a large range of problems  from conical intersections in quantum chemistry \cite{LT05, FL08} to effective mass equations~\cite{CFM19,CFM20}. The semi-classical measures that we consider here  have common features with the two-scales ones in the sense that they are operator-valued. This operator-valued feature  arises from the inhomogeneity of the nilmanifolds, in parallel with the homogeneity introduced by a second scale of concentration as in the references above. However, the operator-valued feature is more fundamental here since it is due to non-commutativity of nilmanifolds and is a  direct consequence of the original features of Fourier analysis on nilpotent groups: it is thus intrinsic to the structure of the problem.

  \medskip

The second step of the proof of the first part of Theorem \ref{t:main} consists in passing from the localized observability to observability itself.
Standard arguments (see ~\cite{BZ}) that we describe in Section~\ref{s:weakobs}  allow to derive from Lemma~\ref{lem:locobs},  a {\it weak observability} inequality in time $T$ on the domain $U$:   there exists a constant $C_1>0$ such that
\begin{equation}\label{obs_weak}
\forall u_0\in L^2(M),\;\;\|u_0\|^2_{L^2(M)} \leq C_1 \int_0^T \left\|   {\rm e}^{it(\frac12 \Delta_M+\mathbb{V})} u_0\right\|^2_{L^2(U)} dt
+ C_1 \| (\Id-\Delta_M)^{-1} u_0\|^2_{L^2(M)}.
\end{equation}
Note that compared to \eqref{obs}, the latter inequality has an added term in its RHS which controls the low frequencies. This weak observability inequality \eqref{obs_weak} implies~\eqref{obs} via a Unique continuation principle for $\frac12\Delta_M+\mathbb{V}$ (see~\cite{Bon69} and~\cite{Laur}), as we describe in Section~\ref{s:fromweaktostrongobs}. It is then not surprising that the result of Theorem~\ref{t:main} holds as soon as a Unique continuation principle is known for $\frac12\Delta_M+\mathbb{V}$, without further assumption of analyticity on~$\mathbb V$ (see Remark~\ref{rem:analyticity}).

\medskip

For proving the second part of Theorem~\ref{t:main} -- the necessity of the condition~{\bf (H-GCC)} --  we construct a family of initial data $(u^\eps_0)$ for which the solution~$(\psi^\eps(t))$ of the Schr\"odinger equation~\eqref{e:Schrod} concentrates on the curve $\Phi^t_0(x_0,\lambda_0)$, for any choice of $(x_0,\lambda_0)\in M\times \mathfrak z\setminus \{0\}$. As mentioned above, this set of initial data is  the non-commutative counterpart to the wave packets (also called coherent states)  in the Euclidean setting~\cite{CR,hag}. These aspects are the subject of Section~\ref{sec:WP}. Our proof relies on a statement of propagation of semiclassical measures which was proved in~\cite{FF} when $\mathbb V=0$ and that we adapt to our setting. A second proof consists in using the results of Appendix \ref{a:wpsolutions}, which are of independent interest: we prove that, if the initial datum is a wave packet, the solution of~\eqref{e:Schrod} is also (approximated by) a wave packet.

\medskip

Our approach 
could be developed in general graded Lie groups through the  generalization of the tools we use: for semi-classical measures in graded groups, see Remarks~3.3 and~4.4 in~\cite{FF1}, and for an extension  of non-commutative wave packets to a more general setting, see Sections~6.3 and~6.4 in~\cite{FF18} (based on~\cite{pedersen}).

\medskip

\noindent{\bf Acknowledgements}. We thank V\'eronique Fischer, Matthieu L\'eautaud, Fabricio Maci\`a and Chenmin Sun for interesting discussions. The authors are also grateful to the referees for their remarks and suggestions.  C.L. was partially supported by the grant ANR-15-CE40-0018 of the ANR (project SRGI).


\section{Semi-classical analysis on quotient manifolds}\label{sec:semiclas}

Semi-classical analysis is based on the analysis of the scales of oscillations of functions. It uses a microlocal approach, meaning that one understands functions in  the phase space, i.e. the space of  position/impulsion of quantum mechanics.
As the impulsion variable is the dual variable of the position variable via the Fourier transform, microlocal analysis crucially relies on the Fourier representation of functions, and on the underlying harmonic analysis.

\medskip

Recall that, in the usual Euclidean setting,  the algebra of pseudodifferential  operators contains those of multiplications by  functions together with Fourier multipliers. These operators are defined by their symbols via the Fourier inversion formula and are  used for analyzing families of functions in the phase space. Indeed, their boundedness in $L^2$ for adequate classes of symbols allows to build a  linear map on  the set of symbols, the weak limits of which are characterized by non-negative Radon measures. These measures give phase space  information on the obstruction to strong convergence of bounded families in $L^2(\R^d)$. In a context where no specific scale is specified, they are called microlocal defect measures, or~$H$-measures and were first introduced independently in~\cite{gerard_91,tartar}. When a specific scale of oscillations is prescribed, this scale is called the semi-classical parameter and they are called semi-classical (or Wigner) measures (see~\cite{HMR,gerard_X,gerardleichtnam,LP93,GMMP}).
 If these functions are moreover solutions of some equation, the semi-classical measures may have additional properties such as invariance by a flow. 

\medskip  

In the next sections, we follow the same steps, adapted to the context of quotients of H-type groups, which are non-commutative: following the theory of non-commutative harmonic analysis (see \cite{CG,Tay} and some elements given in Appendix~\ref{a:rep}), we define the (operator-valued) Fourier transform \eqref{e:fouriertransform}, based on the unitary irreducible representations of the group, recalled in \eqref{eq_widehatG}, which form an analog to the usual frequency space. Then, adapting the ideas of~\cite{FF1} to the context of nilmanifolds, we use the Fourier inversion formula \eqref{inversionformula} to define in~\eqref{e:defsigma} a class of symbols and the associated semi-classical pseudodifferential operators in~\eqref{def:pseudo}. From this, Proposition \ref{prop:semiclas} guarantees the existence of semi-classical measures, whose additional invariance properties for solutions of the Schr\"odinger equation are listed in Proposition~\ref{p:measure0}.

 \subsection{Harmonic analysis on quotient manifolds}

Let $G$ be a stratified nilpotent Lie group of $H$-type and $\widetilde{\Gamma}$ be a discrete cocompact subgroup of $G$. We consider the left quotient $M=\widetilde{\Gamma}\backslash G$ and we denote by
$\pi$ the canonical projection
$$\pi:G\rightarrow M$$
which associates to $x\in G$ its class modulo $\widetilde{\Gamma}$.

For each $\lambda \in \mathfrak z^*\setminus\{0\}$, one associates with $\lambda$ the canonical skew-symmetric form $B(\lambda)$ defined on~$\mathfrak v$ by
$$B(\lambda)(U,V)= \lambda([U,V]).$$

The map $J_\lambda:\mathfrak v\rightarrow \mathfrak v$ of Section \ref{s:intro} is the natural endomorphism associated with $B(\lambda)$ and the scalar product $\langle \cdot,\cdot\rangle$.
In $H$-type groups,
the symmetric form $-J_\lambda^2$ is the scalar map $|\lambda|^2 {\rm Id}$ (note that $-J_\lambda^2$ is  always a non-negative symmetric form).
Therefore,  one can find a $\lambda$-dependent orthonormal basis
$$
\left (P_1^{(\lambda)} , \dots ,P_d^{(\lambda)},  Q_1^{(\lambda)} , \dots ,Q_d^{(\lambda)}\right)
$$
of $\mathfrak v$  where $J_\lambda$ is represented by 
$$  \langle J_\lambda(U),V\rangle=B(\lambda)(U,V)= |\lambda| U^t JV\;\;\mbox{ with }\;\;
J=\begin{pmatrix}0 & {\rm Id} \\ -{\rm Id} & 0\end{pmatrix},$$ 
the vectors $U,V\in \mathfrak v$ being written in the $\left (P_1^{(\lambda)} , \dots ,P_d^{(\lambda)},  Q_1^{(\lambda)} , \dots ,Q_d^{(\lambda)}\right)$-basis.
We then  decompose~$ \mathfrak v$ in a $\lambda$-depending way as
$ \mathfrak v = \mathfrak p_\lambda+  \mathfrak q_\lambda$
 with
 $$
 \begin{aligned}
  \mathfrak p:=\mathfrak p_\lambda:= \mbox{Span} \, \big (P_1^{(\lambda)}, \dots ,P_d^{(\lambda)} \big) \, , & \quad \mathfrak q:=\mathfrak q_\lambda:= \mbox{Span} \, \big (Q_1^{(\lambda)}, \dots ,Q_d^{(\lambda)}\big).
   \end{aligned}
$$
Denoting by $z=(z_1,\cdots ,z_p)$ the coordinates of $Z$ in a fixed orthonormal basis $(Z_1,\cdots, Z_p)$ of $\mathfrak z$, and once given $\lambda\in\mathfrak z^*\setminus\{0\}$,
we will often use the writing of an element $x\in G$ or $X\in \mathfrak g$ as
\begin{equation}
\label{eqxpqz}
x={\rm Exp}(X) , \qquad X=p_1P_1^{(\lambda)} +\ldots + p_dP_d^{(\lambda)} \ + \ q_1Q_1^{(\lambda)} + \ldots + q_d Q_d ^{(\lambda)}\ + \ z_1 Z_1 +\ldots + z_p Z_p,
\end{equation}
where $X=P+Q+Z$, $p=(p_1,\cdots,p_d)$ are the $\lambda$-dependent coordinates of $P$ on the vector basis $(P_1^{(\lambda)}, \cdots,P_d^{(\lambda)})$,  $q=(q_1,\cdots,q_d)$ those of $Q$ on $(Q_1^{(\lambda)},\cdots,Q_d^{(\lambda)})$, and  $z=(z_1,\cdots , z_p)$ of $Z$ are independent of $\lambda$.

\begin{example}\label{ex:Heisenberg1}
In the Heisenberg group $\mathbf{H}^d$, there is a natural choice of coordinates, those we used in Section~\ref{s:exampleinheis} (see \cite[Chapter 1]{Tay}). However, it does not coincide with  the $(p,q,z)$ coordinates that we could define as above by associating with  $\lambda=\alpha dz$, $\alpha\in\R$, 
the vectors  $P_j^{(\lambda)}=X_j$, $Q_j^{(\lambda)}=Y_j$ for $\alpha>0$, 
and $P_j^{(\lambda)}=X_j$, $Q_j^{(\lambda)}=-Y_j$ for $\alpha<0$. One then finds coordinates 
$(p,q,z)$ that are not the usual coordinates $(x,y,s)$ of the Heisenberg groups: 
\begin{equation}\label{change:var}
(x,y,s)=(p,q,z)\;\; \text{if} \;\; \lambda>0\;\;\text{and} \;\; (x,y,s)= (p,-q,z)\;\; \text{if} \;\; \lambda<0.
\end{equation}
In general H-type groups, there is no canonical choice of coordinates, unlike for Heisenberg groups.   
\end{example}

As already mentioned in Section \ref{s:mainresfkl}, we also fix an orthonormal basis $(V_1,\ldots, V_{2d})$ of $\mathfrak v$ to write the coordinates $
v=(v_1,\ldots,v_{2d})$ of a vector 
$$
V=v_1V_1 +\ldots + v_{2d} V_{2d}\in \mathfrak v;
$$
 both this orthonormal basis and the coordinates are
independent of $\lambda$. 
With these coordinates, we  define  a  quasi-norm by setting
\begin{equation}\label{def:quasinorm}
|x|= \left( |v_1|^4 + \cdots +|v_{2d}|^4 + |z_1|^2 + \cdots + |z_p|^2\right)^{1/4},\;\; x= {\rm Exp} (V+Z)\in G.
\end{equation}
We recall that it satisfies a triangle inequality up to a constant.

\subsubsection{Functional spaces}
We shall say that a function $f$ on $G$ is  $\widetilde{\Gamma}$-leftperiodic if we have
$$\forall x\in G,\;\;\forall \gamma\in \widetilde{\Gamma} ,\;\; f(\gamma x)=f(x).$$
With a function $f$ defined on $M$, we associate the $\widetilde{\Gamma}$-leftperiodic function $f\circ \pi$ defined on $G$. Conversely, a
 $\widetilde{\Gamma}$-leftperiodic function $f$ naturally defines a function  on $M$. Thus the set of functions on $M$ is in one-to-one relation  with the set of $\widetilde{\Gamma}$-leftperiodic functions on $G$.

\medskip

The inner products on $\mathfrak v$ and ${\mathfrak z}$ allow us to consider the Lebesgue measure $dv\, dz$ on ${\mathfrak g}={\mathfrak v}\oplus{\mathfrak z}$. Via the identification of $G$ with ${\mathfrak g}$ by the exponential map, this induces a Haar measure $dx$ on $G$ and on $M$. This measure is invariant under left and right translations:
  $$
  \forall f  \in L^1(M) \, ,  \quad  \forall x  \in M \,, \quad \int_M f(y) dy  = \int_M f(x  y)dy= \int_M f(y  x)dy \, .
   $$
The convolution of two functions $f$ and $g$ on $M$ is given by
   $$f*g(x)=\int_M f(xy^{-1})g(y) dy=\int_Mf(y)  g(y^{-1} x) dy.$$
Using the bijection of the set of functions on $M$ with the set of $\widetilde{\Gamma}$-leftperiodic functions on $G$, we deduce that $f*g$ is well-defined as a function on $M$.
Finally,    we define   Lebesgue spaces by
$$
 \|f\|_{L^q (M)}  := \left( \int_M |f(y)|^q \: dy \right)^\frac1q \, ,
 $$
 for $q\in[1,\infty)$, with the standard modification when~$q=\infty$.

\subsubsection{Homogeneous dimension}
  Since $G$ is stratified,   there is a natural family of dilations on ${\mathfrak g}$ defined for $t>0$ as follows: if~$X$ belongs to~$ {\mathfrak g}$, we  decompose~$X$ as~$\displaystyle X=V+Z$ with~$V\in {\mathfrak v}$ and~$Z\in {\mathfrak z}$ and we set
   $$
   \delta_t X:=tV+t^2Z  \, .
   $$
The dilation is defined on $G$ via the identification by the exponential map  as the map ${\rm Exp}\, \circ \delta_t \, \circ {\rm Exp}^{-1}$ that we still denote by $\delta_t$.
  The dilations $\delta_t$, $t>0$, on $\mathfrak g$ and $G$ form a one-parameter group of automorphisms of the Lie algebra $\mathfrak g$ and of the group~$G$.
The Jacobian of the dilation $\delta_t$ is $t^Q$ where
 $$Q:={\rm dim}\, {\mathfrak v} +2{\rm dim}\, {\mathfrak z} = 2d+2p$$
  is called the homogeneous dimension of $G$.
 A differential operator $T$ on $G$
(and more generally any operator $T$ defined on $C^\infty_c(G)$ and valued in the distributions of $G\sim \R^{2d+p}$)
 is said to be homogeneous of degree $\nu$ (or $\nu$-homogeneous) when
 $
 T (f\circ \delta_t) = t^\nu (Tf)\circ \delta_t.
 $
We recall that the  quasi-norm introduced in~\eqref{def:quasinorm} satisfies $|\delta_r x|= r|x|$ for all $r>0$ and $x\in G$. It is a homogeneous quasi-norm and we recall that any homogeneous quasi-norm is equivalent to it.

\subsubsection{Irreducible representations and Fourier transform}

For the sake of completeness, many details about the results of this section, which are standard in non-commutative harmonic analysis, are given in Appendix \ref{a:rep}.

\medskip

The infinite dimensional irreducible representations of~$G$ are parametrized by $\mathfrak z^* \setminus \{0\}$:
for  $\lambda\in \mathfrak z^*\setminus\{0\}$, one defines $\pi_{\cdot}^\lambda:G\rightarrow L^2( \mathfrak p_\lambda) \sim L^2(\R^d)$ by
\begin{equation}\label{def:pilambda}
\pi^{\lambda}_{x} \Phi(\xi)=
 {\rm e}^{i\lambda(z)+ \frac i2 |\lambda|\,p \cdot q +i\sqrt{|\lambda|} \,\xi \cdot q} \,\Phi \left(\xi+\sqrt{|\lambda|}p\right),
 \end{equation}
 where $x$ has been written as in \eqref{eqxpqz}.
 The representations $\pi^\lambda$, $\lambda\in \mathfrak z^*\setminus\{0\}$, are infinite dimensional.
 The other unitary irreducible representations of $G$
  are given by the characters of the first stratum in the following way:
 for every  $\omega\in \mathfrak v ^*$,
  we set
\begin{equation}\label{eq:0omega}
\pi^{(0,\omega)}_x= {\rm e}^{i \omega(V)}, \quad
x={\rm Exp} (V+Z)\in G, \quad\mbox{with}\ V\in{\mathfrak v} \ \mbox{and} \  Z\in{\mathfrak z}.
\end{equation}
The set~$\widehat G$ of all unitary irreducible representations modulo unitary equivalence
is then parametrized by $({\mathfrak z}^*\setminus \{0\})\sqcup {\mathfrak v}^*$:
\begin{equation}
\label{eq_widehatG}	
\widehat G =
\{\mbox{class of} \ \pi^\lambda \ : \ \lambda \in \mathfrak z^* \setminus\{0\}\}
\sqcup \{\mbox{class of} \ \pi^{(0,\omega)} \ : \ \omega \in \mathfrak v^* \}.
\end{equation}
The subset $\mathfrak v^*$  of $\widehat G$ is often thought as a bundle over $\lambda=0$ (see the discussions about the Heisenberg fan in~\cite[Lemma~2.2]{FF1}). This explains the  $0$ in the notation $(0,\omega)$ that we use here to differentiate $\pi^{(0,\omega)}$ from $\pi^\lambda$. It is natural since we think of $\mathfrak v^*$ as ``horizontal'' and $\mathfrak z^*$ as ``vertical''.

\medskip
 
We will  identify each representation $\pi^\lambda$ with its equivalence class.
Note that the trivial representation $1_{\widehat G}$ corresponds to the class of $\pi^{(0,\omega)}$ with
$\omega=0$, i.e.
$
1_{\widehat G}
:=
\pi^{(0,0)}.
$
The dilation $\delta_\eps$ extends on~$\widehat G$ by $\eps \cdot\pi^ \lambda=\pi^{\eps^2 \lambda} $ for $\lambda\in\mathfrak z^*\setminus \{0\}$ and $\eps \cdot \pi^{(0,\omega) }= \pi^{(0,\eps \omega)} $ for $\omega\in\mathfrak v^*$.

\medskip

The set $G\times \widehat G$ will be interpreted in our analysis as  the phase space of $G$, and $M\times \widehat G$ as the phase space of $M$, in analogy with the fact that $\R^d\times \R^d$ and $\T^d\times \R^d$ are respectively the phase space of the Euclidean space $\R^d$ and of the torus $\T^d$.

\begin{example}
In the case of the Heisenberg group, the formula~\eqref{def:pilambda} differs from 
the usual one for the Heisenberg groups \cite[Equation (2.23) in Chapter 1]{Tay} because the coordinates $(p,q,z)$ are different from the canonical ones~$(x,y,s)$ (see Example~\ref{ex:Heisenberg1}). They are related by the relation~\eqref{change:var}.
\end{example}

\medskip

The Fourier transform  is defined on~$\widehat G$ and is valued  in   the space of  bounded operators
on~$L^2( \mathfrak p_\lambda)$:
for any~$\lambda\in{\mathfrak z}^*$, $\lambda\not=0$,
\begin{equation} \label{e:fouriertransform}
{\mathcal F}f(\lambda):=
\int_G f(x)\left( \pi^{\lambda}_{x }\right)^* \, dx \, ,
\end{equation}
Besides,  above finite dimensional representations, the Fourier transform is defined 
 for $\omega\in \mathfrak v^*$ by
 $$
\widehat f(0,\omega)=
 \mathcal F f (0,\omega)
 := \int_G f(x) (\pi^{(0,\omega)}_x)^* dx
=\int_{\mathfrak v \times \mathfrak z}
f({\rm Exp}(V+Z) ) e^{-i\omega(V)} dV dZ.	
$$
 Functions~$f$ of~$L^1(G)$  have a Fourier transform~$\left({\mathcal F}(f)(\lambda)\right)_{\lambda\in \mathfrak z^*}$ which is a bounded family of bounded operators on~$L^2( \mathfrak p_\lambda)$
 with uniform bound:
$$
\|\mathcal Ff (\lambda) \|_{{\mathcal L}(L^2(\mathfrak p_\lambda))}
\leq
\int_G |f(x)|\|(\pi^\lambda_x)^* \|_{{\mathcal L}(L^2(\mathfrak p_\lambda))} dx
=
\|f\|_{L^1(G)}.
$$
since the unitarity of $\pi^\lambda$ implies $\|(\pi^\lambda_x)^* \|_{{\mathcal L}(L^2(\mathfrak p_\lambda))}=1$.

\begin{example}\label{ex:Heisenberg2}
In  the Heisenberg group $\mathbf{H}^d$, using the link exhibited in Example~\ref{ex:Heisenberg1} between the coordinates  in the basis $(P_j^{(\lambda)},Q_j^{(\lambda)})_{1\leq j\leq d}$ and the variables $(x,y,s)$ of Section~\ref{s:exampleinheis}, we obtain that the Fourier transform  of $f\in\mathcal S(\mathbf{H}^d)$ writes 
$$\forall \Phi\in \mathcal S(\R^d),\;\; \mathcal Ff(\lambda)  \Phi(\xi)=
 \left\{
\begin{array} l
\int_{\R^{2d+1} }{\rm e}^{i\lambda s + \frac i 2 \lambda x\cdot y + i\sqrt\lambda\, \xi\cdot y } \Phi(\xi +\sqrt\lambda \,x) dx\, dy\, ds\;\;\mbox{if}\;\;\lambda>0,\\
\int_{\R^{2d+1} }{\rm e}^{i\lambda s + \frac i 2 \lambda x\cdot y - i\sqrt{|\lambda|}\, \xi\cdot y } \Phi(\xi +\sqrt{|\lambda|} \, x) dx\, dy\, ds\;\;\mbox{if}\;\;\lambda<0.
\end{array}
\right.
$$
\end{example}

The Fourier transform can be extended to an isometry from~$L^2(G)$ onto the Hilbert
space of measurable families~$ A  = \{ A (\lambda ) \}_{\lambda \in{\mathfrak z}^*\setminus \{0\}}$
 of operators on~$L^2( \mathfrak p_\lambda)$ which are
Hilbert-Schmidt for almost every~$\lambda\in{\mathfrak z}^*\setminus \{0\}$,  with norm
\[ \|A\| := \left( \int_{\mathfrak z^* \setminus\{0\}}
\|A (\lambda )\|_{HS (L^2( \mathfrak p_\lambda))}^2 |\lambda|^d \, d\lambda
\right)^{\frac{1}{2}}<\infty  \, .\]
  We have the  Fourier-Plancherel formula:
$$ \int_G  |f(x)|^2  \, dx
=  c_0 \, \int_{\mathfrak z^* \setminus\{0\}} \|{\mathcal F}f(\lambda)\|_{HS(L^2( \mathfrak p_\lambda))}^2 |\lambda|^d  \,  d\lambda   \,,
$$
where $c_0>0$ is a computable constant. 

\begin{remark}\label{rem:Ghat}
This relation shows that Plancherel measure   of $\widehat G$  is $d\mu:= c_0|\lambda|^d d\lambda$ and   is supported in the subset $\{\mbox{class of} \ \pi^\lambda \ : \ \lambda \in \mathfrak z^* \setminus\{0\}\}$ of~$\widehat G$, in particular the subset $\{\mbox{class of} \ \pi^{(0,\omega)} \ : \ \omega \in \mathfrak v^* \}$ of $\widehat G$ is of mass~$0$ for the Plancherel measure. Therefore, the integral on $\mathfrak z^*\setminus \{0\}$ of the Fourier-Plancherel formula can be thought as an integral on $\widehat G$, thinking $\mathfrak v^*$ above $\{\lambda=0\}$, as suggested by the notation. 
\end{remark}

Finally, an inversion formula for~$ f \in {\mathcal S}(G)$ and~$x\in G$  writes:
\begin{equation}
\label{inversionformula} f(x)
= c_0 \, \int_{\mathfrak z^* \setminus\{0\}} {\rm{Tr}} \, \Big(\pi^{\lambda}_{x} {\mathcal F}f(\lambda)  \Big)\, |\lambda|^d\,d\lambda \,,
\end{equation}
where ${\rm Tr}$ denotes the trace of operators of ${\mathcal L}(L^2({\mathfrak p}_\lambda))$ (see \cite[Theorem 2.7]{Tay}).
This   formula makes sense since for Schwartz functions $f \in {\mathcal S}(G)$, the operators ${\mathcal F}f(\lambda)$, $\lambda\in \mathfrak z^*\setminus\{0\}$, are trace-class, with enough regularity in $\lambda$ so that $\int_{\mathfrak z^* \setminus\{0\}} {\rm{Tr}} \, \Big| {\mathcal F}f(\lambda)  \Big|\, |\lambda|^d\,d\lambda$ is finite.

\medskip 

To conclude this section, it is important to notice that the differential operators have a  Fourier resolution  that allows to think them as Fourier multipliers. In particular, the resolution of the sub-Laplacian $-\Delta_G$ is well-understood
$$\forall f\in{\mathcal S}(G),\;\;\mathcal F(-\Delta_G f)(\lambda)= H(\lambda) {\mathcal F} (f) (\lambda).$$
At $\pi^{(0,\omega)}$, $\omega\in{\mathfrak v}^*$, it is the number
${\mathcal F} (-\Delta_G)(0,\omega) = |\omega|^2,$
and  at $\pi^\lambda$, $\lambda\in{\mathfrak z}^*\setminus\{0\}$, it is the unbounded operator
\begin{equation}\label{def:H}
 H(\lambda)=|\lambda| \sum_{ j=1}^{d} \left( -\partial_{\xi_j}^2+\xi_j^2\right),
\end{equation}
where we have used the identification $\mathfrak p_\lambda\sim \R^d$ and the observation that for $\lambda\in\mathfrak z^*\setminus\{0\}$, $ f\in L^2(\mathfrak p_\lambda)$ and $1\leq j\leq d$,
$$\mathcal F (P^{(\lambda)}_jf)=i\partial_{\xi_j}\mathcal F(f) \; \;\mbox{and}\; \;\mathcal F(Q^{(\lambda)}_jf)=\xi_j\mathcal F( f).$$
One writes 
\begin{equation}\label{calcul:pi(P)}
\pi^\lambda(P^{(\lambda)}_j)=i\partial_{\xi_j} \;  \;\mbox{and}\;\; \pi^\lambda(Q^{(\lambda)}_j)=\xi_j,\;\;\lambda\in\mathfrak z^*\setminus\{0\},\;\;1\leq j\leq d.
\end{equation}

\subsection{Semi-classical pseudodifferential operators on quotient manifolds} \label{sec:semiclassic}
As observables of quantum mechanics are functions on the phase space, the symbols of pseudodifferential operators on $M$ are functions defined on $M\times \widehat{G}$. In this non-commutative framework, they have the same properties as the Fourier transform and they are operator-valued symbols.

\medskip

Following~\cite{FF,FF1}, we consider the class of symbols ${{\mathcal A}_0}$ of fields of operators defined on $M\times \widehat G$  by
$$\sigma(x,\lambda)\in{\mathcal   L}(L^2(\mathfrak p_\lambda)),\;\;(x,\lambda)\in M\times \widehat G,$$
that are smooth in the variable $x$ and   Fourier transforms of  functions of the set~${\mathcal S}(G)$ of Schwartz functions on~$G$
 in the variable $\lambda$: for all $(x,\lambda)\in M\times \widehat G$,
\begin{equation}\label{e:defsigma}
\sigma(x,\lambda) = {\mathcal F} \kappa_x (\lambda),\;\;\kappa\in\mathcal C^\infty(M,\mathcal S(G)).
\end{equation}
A similar class of symbols in the Euclidean context was introduced in \cite[Section 3]{LP93}.
\noindent Note that we kept in \eqref{e:defsigma} the notation $\lambda$ also for the parameters $(0,\omega)$, $\omega\in\mathfrak v^*$.
In this case, the operator  ${\mathcal F} \kappa_x ((0,\omega))=\sigma(x,(0,\omega))$ reduces to a complex number since the associated Hilbert space is $\mathbb{C}$.

\smallskip

 If $\eps>0$, we associate with $\kappa_x$ (and thus with $\sigma(x,\lambda)$) the function $\kappa_x^\eps$ defined on $G$ by
\begin{equation} \label{e:kappaeps}
\kappa^\eps_x(z)= \eps^{-Q} \kappa_x(\delta_{\eps^{-1}}(z)),
\end{equation}
We then  define the semi-classical pseudodifferential operator ${\rm Op}_\eps(\sigma)$ via the identification of functions $f$ on $M$ with $\widetilde{\Gamma}$-leftperiodic functions on $G$:
\begin{equation}\label{def:pseudo}
{\rm Op}_\eps(\sigma)f(x)=  \int_{ G} \kappa_x^\eps(y^{-1} x) f(y)dy.
\end{equation}
When $\eps=1$, we omit the index $\eps$ and just write ${\rm Op}$ instead of ${\rm Op}_\eps$.
\begin{remark}
The formulas \eqref{def:pseudo}, \eqref{e:kappaeps} and \eqref{e:defsigma} may be compared to the formulas of the semiclassical (standard) quantization on the torus $\T^n= (\R/2\pi\Z)^{n}$, namely, for $\sigma(x,\xi), x\in\T^n, \xi\in\R^n$ and $f$ a  $(2\pi\Z)^n$-periodic function,
\begin{align*}
\Op_\eps^{\T^n}(\sigma)f(x)&=\int_{\R^n}K^\eps\left(x,x-y\right)f(y)dy\\
 \text{where} \qquad & K^\eps(x,z)=\eps^{-n}K(x,\eps^{-1}z), \\
  K(x,w)=\frac{1}{(2\pi)^n}&\int_{\R^n}e^{iw\cdot\xi} \sigma(x,\xi)d\xi\in C^\infty(\T^n,\mathcal S(\R^n)), \\
 \text{i.e.} \quad& \sigma(x,\xi)=(\mathcal{F}^{\R^n}_w K)(x,\xi).
\end{align*}
\end{remark}

\medskip 

We observe the following facts (the proofs of points (3) to (7) are discussed more in details in Appendix~\ref{sec:proofpseudo}).

\begin{enumerate}
\item
The operator ${\rm Op}_\eps(\sigma)$ is well-defined as an operator on $M$. Indeed,
\begin{align*}
{\rm Op}_\eps(\sigma) f(\gamma x)
=  \int_{ G} \kappa_{\gamma x}^\eps(y^{-1} \gamma x) f(y)dy
=  \int_{ G} \kappa_{x}^\eps(y^{-1} x) f(\gamma y)dy 
= {\rm Op}_\eps(\sigma) f( x).
\end{align*}
Here we have used  a change of variable and the relations $\kappa_{\gamma x}(\cdot)=\kappa_x(\cdot)$ and $f(\gamma  y)=f(y)$.

\medskip

\item Using \eqref{inversionformula} and \eqref{e:defsigma}, we have the useful identities 
\begin{align*}
{\rm Op}_\eps(\sigma)f(x)
=\eps^{-Q}  \int_{ G} \kappa_x(\delta_{\eps^{-1}}(y^{-1} x)) f(y)dy 
&= \int_{G\times \left(\mathfrak z^*\setminus \{0\}\right)} {\rm Tr} ( \pi^\lambda_{y^{-1} x} \sigma(x,\eps\cdot \lambda) ) f(y) |\lambda|^d d\lambda dy.
\end{align*}
In view of Remark~\ref{rem:Ghat}, using the notations of the dilations on $\widehat G$, we have the general formula  (as in~\cite{FF1}, Remark~3.3)
$${\rm Op}_\eps(\sigma)f(x)
= \int_{G\times \widehat G} {\rm Tr} ( \pi_{y^{-1} x} \sigma(x,\eps\cdot \pi) ) f(y)  d\mu(\pi) dy.$$

\item
The kernel of ${\rm Op}_\eps(\sigma)$ is given by
$$k_\eps(x,y) = \sum_{\gamma\in \widetilde{\Gamma}}  \kappa_{ x}^\eps(\gamma y^{-1} x)$$
\item The family of operators $\left({\rm Op}_\eps(\sigma)\right)_{\eps >0}$ is uniformly bounded in ${\mathcal L}(L^2(M))$:
\begin{equation}\label{eq:boundedness}
\| {\rm Op}_\eps (\sigma)\|_{\mathcal L(L^2(M))} \leq  \int_{G} \sup_{x\in M}  |\kappa_x(y)|dy.
\end{equation}
\item Semi-classical pseudodifferential operators act locally: let $\sigma\in{\mathcal A}_0$ be compactly supported in an open set $\Omega$ such that $\overline \Omega$ is strictly included in a fundamental domain $\mathcal B$ of $\widetilde{\Gamma}$ and $\chi\in \mathcal C^\infty_c(\mathcal B)$ such that $\chi \sigma=\sigma$. Then, by definition
$${\rm Op}_\eps(\sigma)={\rm Op}_\eps(\chi \sigma) =\chi {\rm Op}_\eps(\sigma)$$
and for all $N\in \N$, there exists a constant $c_N$ such that, for any $\eps>0$,
\begin{equation}\label{eq:localisation}
\left\| {\rm Op}_\eps(\sigma) -  \chi   {\rm Op}_\eps(\sigma) \chi \right\| _{ \mathcal L(L^2(M))}=
\left\| {\rm Op}_\eps(\sigma) -    {\rm Op}_\eps(\sigma) \chi \right\| _{ \mathcal L(L^2(M))}
\leq c_N \,\eps^N.
\end{equation}
 \end{enumerate}

\begin{remark}\label{correspGM}
The last property is crucial for our analysis since it allows to transfer results obtained in the nilpotent group $G$ for functions in $L^2_{\rm loc} (G)$  to the case of square-integrable functions of the homogeneous manifold~$M$.  Indeed, if $f\in L^2(M)$, then $f$ can be identified to a $\widetilde{\Gamma}$-leftperiodic function on $L^2_{loc}(G)$. In particular, we have  $\chi f\in L^2(G)$ and $ {\rm Op}_\eps( \sigma) \chi f= \chi {\rm Op}_\eps( \sigma) \chi f $ coincides with the standard definition of~\cite{FF1,FF}. Then, for $f,g\in L^2(M)$ and $\sigma,\;\chi$ as before, we have for all $N\in\N$ 
\begin{equation}\label{eq:correspGM}
 \left({\rm Op}_\eps(\sigma) f,g\right)_{L^2(M)} =  \left({\rm Op}_\eps(\sigma)\chi f,\chi g\right)_{L^2(G)} +O(\eps^N \| f\|_{L^2(M)}\|g\|_{L^2(M)}).
 \end{equation}
\end{remark}

This correspondance between computations in $M$ and in $G$ will be further developed at the beginning of Section \ref{s:WPnoncom}, notably through the periodization operator $\mathbb{P}$. It is also at the root of the  next two properties. For stating them, we introduce the difference operators, acting on ${\mathcal L}(L^2(\mathfrak p_\lambda))$:
$$
\Delta_{p_j}^{\lambda}= |\lambda|^{-1/2}[\xi_j, \cdot ] ,
\qquad
\Delta_{q_j}^{\lambda} = |\lambda|^{-1/2}[i\partial_{\xi_j}, \cdot],
\quad 1\leq j\leq d.
$$
We also use the operators $\pi^\lambda(P^{(\lambda)}_j)$ and $\pi^\lambda(Q^{(\lambda)}_j)$ calculated in~\eqref{calcul:pi(P)}.

\begin{enumerate}
\item[(6)] The following symbolic calculus result holds:
\end{enumerate}

\begin{proposition}\label{prop:symbcal}
Let $\sigma\in {\mathcal A}_0$. Then, in ${\mathcal L}(L^2(M))$,
\begin{equation}\label{adjoint}
{\rm Op}_\eps(\sigma)^*=  {\rm Op}_\eps (\sigma^*) -{\eps} \,{\rm Op}_\eps(P^{(\lambda)}\cdot \Delta^\lambda_p \sigma^*+Q^{(\lambda)}\cdot \Delta^\lambda_q \sigma^*)+O(\eps^2).
\end{equation}
Let $\sigma_1,\sigma_2 \in {\mathcal A}_0$. Then in ${\mathcal L}(L^2(M))$,
\begin{equation}\label{composition}
{\rm Op}_\eps(\sigma_1)\circ {\rm Op}_\eps(\sigma_2) = {\rm Op}_\eps(\sigma_1\,\sigma_2) -\eps\,  {\rm Op}_\eps\left( \Delta_p^\lambda \sigma_1 \cdot P^{(\lambda)}\,\sigma_2+\Delta^\lambda_q \sigma_1 \cdot Q^{(\lambda)}\, \sigma_2\right)+O(\eps^2).
\end{equation}
\end{proposition}

\medskip 

\begin{enumerate}
\item[(7)] The main contribution of the  function $(x,z)\mapsto \kappa_x (z)$ to the operator ${\rm Op}_\eps(\sigma)$, $\sigma(x,\lambda)=\mathcal F (\kappa_x) (\lambda) $ is due to its values close to $z={\bf 1}_G$.
\end{enumerate}

\begin{proposition} \label{prop:cutoffkernel}
 Let $\chi_0\in{\mathcal C}^\infty(G)$ be compactly supported close to $1_G$ and $\chi_\eps=\chi_0\circ \delta_\eps$. With $\sigma =\mathcal F (\kappa_x) (\lambda) $ we associate 
 $\sigma_\eps = \mathcal F (\kappa_x \chi_\eps).$
 Then, in $L^2(M)$, for all $N\in\N$,
 $${\rm Op}_\eps (\sigma)={\rm Op}_\eps (\sigma_\eps) +O(\eps^{N}).$$
 \end{proposition}

\subsection{Semi-classical measures}\label{s:sclmeas}
When given a bounded sequence $(f^\eps)_{\eps>0}$ in $L^2(M)$, one defines the quantities $\ell_\eps(\sigma)$  in analogy with quantum mechanics as the  action of observables on this family, i.e. the families
$$\ell_\eps(\sigma)= \left({\rm Op}_\eps(\sigma) f^\eps,f^\eps\right),\;\; \sigma\in \mathcal A_0.$$
Since these quantities are bounded sequences of real numbers, it is then natural to study the asymptotic $\eps\rightarrow 0$. The families $(\ell_\eps(\sigma))_{\eps>0}$  have weak limits that depend linearly on $\sigma$ and enjoy additional properties. We call semi-classical measure of $(f^\eps)_{\eps>0}$ any of these linear forms.

\medskip

For describing the properties of semi-classical measures,   we need  to introduce a few notations.
If~$Z$ is a locally compact Hausdorff set, we denote by ${\mathcal M}(Z)$ the set of finite  Radon measures on~$Z$ and by ${\mathcal M}^+(Z)$ the subset of its positive elements. Considering the metric space
$M\times \widehat G$ endowed with the field of complex Hilbert spaces $L^2(\mathfrak p_\lambda)$ defined above elements $(x,\lambda)\in M\times \widehat G$, we denote by
	$ \widetilde{\mathcal M}_{ov}(M\times \widehat G) $
	 the set of pairs $(\gamma,\Gamma)$ where $\gamma$ is a positive Radon measure on~$M\times \widehat G$
	and $\Gamma=\{\Gamma(x,\lambda)\in {\mathcal L}(L^2(\mathfrak p_\lambda)):\lambda \in \widehat G\}$ is a measurable field of trace-class operators
such that
\[
\|\Gamma d \gamma\|_{\mathcal M}:=
\int_{M\times \widehat G}{\rm Tr}(
 |\Gamma(x,\lambda)|)d\gamma(x,\lambda)
<\infty.
\]
Here, as usual, $|\Gamma|:=\sqrt{\Gamma\Gamma^*}$. Note that $\Gamma(x,\lambda)$ is defined as a linear operator on the space $L^2(\mathfrak p_\lambda)$ which does not depend on $x$ but which depends on $\lambda$. Considering that 	
 two pairs $(\gamma,\Gamma)$ and $(\gamma',\Gamma')$
in $\widetilde {\mathcal M}_{ov}(M\times \widehat G)$
are {equivalent} when there exists a measurable function $f:M\times \widehat G\to \mathbb C\setminus\{0\}$ such that
$$d\gamma'(x,\lambda) =f(x,\lambda)  d\gamma(x,\lambda)\;\;{\rm  and} \;\;\Gamma'(x,\lambda)=\frac 1 {f(x,\lambda)} \Gamma(x,\lambda)$$
for $\gamma$-almost every $(x,\lambda)\in M\times \widehat G $, we define the equivalence class of $(\gamma,\Gamma)$  by $\Gamma d \gamma$,
and the resulting quotient  by ${\mathcal M}_{ov}(M\times \widehat G)$. One checks readily that $\mathcal M_{ov} (M\times\widehat G)$ equipped with the norm $\| \cdot\|_{{\mathcal M}}$ is a Banach space.

Finally, we say that
a pair $(\gamma,\Gamma)$
in $ \widetilde {\mathcal M}_{ov}(M\times \widehat G)$
 is {positive} when
$\Gamma(x,\lambda)\geq 0$ for $\gamma$-almost all $(x,\lambda)\in M\times \widehat G$.
In  this case, we write  $(\gamma,\Gamma)\in  \widetilde {\mathcal M}_{ov}^+(M\times \widehat G)$,
and $\Gamma d\gamma \geq 0$ for $\Gamma d\gamma \in {\mathcal M}_{ov}^+(M\times \widehat G)$.

\medskip

With these notations in mind, one can mimic the proofs of~\cite{FF}, considering the $C^*$-algebra~$\mathcal A$  obtained as the closure of~$\mathcal A_0$ for the norm $\sup _{(x,\lambda)\in M\times \widehat G} \| \sigma(x,\lambda)\|_{\mathcal L(L^2(\mathfrak p_\lambda))}$. Indeed, the properties of this algebra depend on those of  $\widehat G$  and the analysis of the set and of~\cite{FF1,FF} also applies in this context.  Then, arguing as in~\cite{FF1,FF}, one can define semi-classical measures as  follows.

\begin{theorem}
Let $(f^\eps)_{\eps>0}$ be a bounded family in $L^2(M)$. There exist a sequence $(\eps_k)\in(\R_+^*)^{\N}$ with $\eps_k \underset{k\rightarrow +\infty}{\longrightarrow} 0$, and $ \Gamma d\gamma\in \mathcal{M}_{ov}^+(M\times \widehat{G})$ such that for all $\sigma\in\mathcal{A}$,
\begin{equation*}
({\rm Op}_{\eps_k}(\sigma)f^{\eps_k},f^{\eps_k})_{L^2(M)}  \underset{k\rightarrow +\infty}{\longrightarrow} \int_{M\times\widehat{G}} {\rm Tr}(\sigma(x,\lambda)\Gamma(x,\lambda))d\gamma(x,\lambda).
\end{equation*}
Given the sequence $(\eps_k)_{k\in\N}$, the measure $ \Gamma d\gamma$ is unique up to equivalence. Besides,
\begin{equation*}
\int_{M\times \widehat{G}}{\rm Tr}(\Gamma(x,\lambda))d\gamma(x,\lambda)\leq \limsup_{\eps\rightarrow 0}\|f^\eps\|^2_{L^2(M)}.
\end{equation*}
\end{theorem}

We emphasize on the operator-valued nature of $\Gamma (x,\lambda) {\bf 1}_{\lambda\in \mathfrak z^*}(\lambda)$ in opposition to the fact that
$\Gamma (x,\lambda) {\bf 1}_{\lambda\in \mathfrak v^*}(\lambda)\in \R^+$ (since  finite dimensional representations are scalar operators).

\medskip

The link of semi-classical measures with the limit of energy densities $|f^\eps(x)|^2 dx$ will be discussed below, it is solved thanks to the notion of $\eps$-oscillating families (see Section \ref{s:epsosc}).

\subsection{Time-averaged semi-classical measures}
 \label{s:sclmeastime}
The local observability inequality takes into account time-averaged quadratic quantities of the solution of Schr\"odinger equation. Physically, it corresponds to an observation, i.e. the measurement of an observable during a certain time. For example,  when $\mathbb V=0$, the right-hand side of inequality~\eqref{obs_loc}  can be expressed with the  set of observables introduced in the previous section using the symbol $\sigma(x,\lambda)= {\bf 1}_{x\in M} \chi (H(\lambda))$ (see \eqref{def:H} for a definition of $H(\lambda)$). Therefore, when considering time-dependent families, as solutions to the Schr\"odinger equation~\eqref{e:Schrod}, we are interested in the limits of time-averaged quantities: let  $(u^\eps)_{\eps>0}$ be a bounded family in $L^\infty(\R,L^2(M))$, $\theta\in L^1(\R)$ and $\sigma\in {{\mathcal A}_0}$, we define
$$\ell_\eps(\theta, \sigma) = \int_{\R} \theta(t) \left({\rm Op}_\eps(\sigma) u^\eps(t) ,u^\eps(t)\right) _{L^2(M)} dt$$ and we are interested in the limit as $\eps$ goes to~$0$ of these quantities.

\medskip

When introduced,  semi-classical measures were first used for systems with a semi-classical time scaling, i.e. involving $\eps \partial_t$ derivatives, which is not the case here  when multiplying the equation~\eqref{e:Schrod} by $\eps^2$. It is then difficult to derive results for the semi-classical measures at each time  $t$. However, one can deduce results for the time-averaged semi-classical measures that hold almost everywhere in time. Indeed, these measures satisfy important geometric properties that can lead to their identification (for example in Zoll manifolds).
This was first remarked by~\cite{MaciaTorus} and led to important results in control~\cite{AM14,AFM15, MacRiv19}, but also for example  in the analysis of dispersion effects of operators arising in solid state physics~\cite{CFM19,CFM20}. This approach has been extended to $H$-type groups in~\cite{FF} and, arguing in the same manner as for the proof of Theorem~2.8 therein, we obtain the next result on the nilmanifold~$M$.

\begin{proposition}\label{prop:semiclas}
Let $(u^\eps)_{\eps>0}$ be a bounded family in $L^\infty(\R,L^2(M))$. There exist a sequence $(\eps_k)\in(\R_+^*)^{\N}$ with $\eps_k \underset{k\rightarrow +\infty}{\longrightarrow} 0$ and a map $t\mapsto \Gamma_t d\gamma_t$ in $L^\infty(\R,\mathcal{M}_{ov}^+(M\times \widehat{G}))$ such that we have for all $\theta\in L^1(\R)$ and $\sigma\in\mathcal{A}$,
\begin{equation*}
\int_\R \theta(t)({\rm Op}_{\eps_k}(\sigma)u^{\eps_k}(t),u^{\eps_k}(t))_{L^2(M)} dt  \underset{k\rightarrow +\infty}{\longrightarrow} \int_{\R\times M\times\widehat{G}} \theta(t){\rm Tr}(\sigma(x,\lambda)\Gamma_t(x,\lambda))d\gamma_t(x,\lambda)dt.
\end{equation*}
Given the sequence $(\eps_k)_{k\in\N}$, the map $t\mapsto \Gamma_td\gamma_t$ is unique up to equivalence. Besides,
\begin{equation*}
\int_\R\int_{M\times \widehat{G}}{\rm Tr}(\Gamma_t(x,\lambda))d\gamma_t(x,\lambda)dt\leq \limsup_{\eps\rightarrow 0}\|u^\eps\|^2_{L^\infty(\R,L^2(M))}.
\end{equation*}
\end{proposition}

\subsubsection{$\eps$-oscillating families} \label{s:epsosc}
The link between semi-classical measures and the weak limits of  time-averaged energy densities is solved thanks to the notion of $\eps$-oscillation.
Let $(u^\eps)_{\eps>0}$ be a bounded family in $L^\infty(\R,L^2(M))$.
We say that the family $(u^\eps)_{\eps>0}$ is {\it uniformly $\eps$-oscillating} when we have for all $T>0$,
$$
\limsup_{\eps\rightarrow 0} \sup_{t\in[-T,T]} \left\|{\bf 1}_{-\eps^2{\Delta}_M>R} u^\eps(t)
 \right\|_{L^2(M)}\Tend{R}{+\infty}0.
 $$

\begin{proposition}
\label{prop:eops}[\cite{FF}Proposition~5.3]
Let $(u^\eps)\in L^\infty(\R, L^2(M))$ be a uniformly $\eps$-oscillating family admitting a time-averaged semi-classical measure $t\mapsto \Gamma_t d\gamma_t$ for the sequence $(\eps_k)_{k\in \mathbb N}$. Then
 for all
  $\phi\in{\mathcal C}^\infty (M)$ and $\theta\in L^1(\R)$,
$$
\lim_{k\rightarrow +\infty}  \int_{\R\times M} \theta(t) \phi(x) |u^{\eps_k}(t,x)|^2 dxdt
=\int_{\R}  \theta(t)
\int_{M\times \widehat G}
\phi(x)  {\rm Tr} \left(\Gamma_t(x,\lambda) \right) d\gamma_t(x,\lambda) \, dt,
$$
\end{proposition}

\subsubsection{Semi-classical measures for families of Schr\"odinger equations}

Families of solutions to the Schr\"odinger equation~\eqref{e:Schrod} have special features. We recall that  in the (non compact) group~$G$, the operator 
$$H(\lambda)=|\lambda| \sum_{j=1}^{d} \left( -\partial_{\xi_j}^2+\xi_j^2\right)$$ introduced in~\eqref{def:H} is the Fourier resolution of the sub-Laplacian $-\Delta_G$ above $\lambda\in \mathfrak z^*\setminus \{0\}$.
Up to a constant, this is a quantum harmonic oscillator with 
discrete spectrum $\{|\lambda|(2n+d), n\in \N\}$ 
and finite dimensional eigenspaces.
For each eigenvalue $|\lambda|(2n+d)$, we denote by
$
\Pi_n^{(\lambda)}$ and
$\mathcal V^{(\lambda)}_n
$
 the corresponding spectral orthogonal projection and eigenspace.
 Even though the spectral resolution of $-\Delta_G$ and $-\Delta_M$ are quite different, we shall use the operator $H(\lambda)$ as one uses the function  $\xi\mapsto |\xi|^2$ on the phase space of the torus $\T^d$, when studying the operator $-\Delta_{\T^d}$.

\begin{proposition} \label{p:measure0}
Assume $\Gamma_t d\gamma_t$ is associated with a family of solutions to~\eqref{e:Schrod}.
\begin{enumerate}
\item For $(x,\lambda)\in  M\times \mathfrak z^*$
\begin{equation}\label{eq:decomp}
\Gamma_t(x,\lambda)=\sum_{n\in\N } \Gamma_{n,t}(x,\lambda)\;\;{ with}\;\; \Gamma_{n,t}(x,\lambda):= \Pi_n^{(\lambda)}\Gamma_t(x,\lambda) \Pi_n^{(\lambda)}.
\end{equation}
Moreover,
the map $(t,x,\lambda)\mapsto \Gamma_{n,t}(x,\lambda) d\gamma_t(x,\lambda)$ defines
 a continuous function from  $\R$ into the set of distributions on $M\times (\mathfrak z^*\setminus\{0\})$ valued in the finite dimensional space ${\mathcal L}(\mathcal V_n^{(\lambda)})$ which
satisfies
\begin{equation}\label{transport}
\left(\partial_t -(n+\frac d 2) {\mathcal Z}^{(\lambda)}\right)\left(\Gamma_{n,t}(x,\lambda)d\gamma_t(x,\lambda)\right)=0
\end{equation}
\item For $(x,(0,\omega))\in M\times \mathfrak v^*$, the scalar measure $\Gamma_t d\gamma_t$ is invariant under the flow
$$\Xi ^s:(x,\omega) \mapsto (x{\rm Exp} (s\omega\cdot V),\omega).$$
Here, $\omega\cdot V=\sum_{j=1}^{2d} \omega_jV_j$ where $\omega_j$ denote the coordinates of $\omega$ in the dual basis of $V$.
\end{enumerate}
\end{proposition}

The proof of this proposition follows ideas from~\cite{FF} that we adapt to our situation. We give some elements on the proof of this Proposition in Appendix~\ref{app:semiclas}, in particular we explain the continuity of the map $t\mapsto \Gamma_t d\gamma_t$.

\medskip

We have now  all the tools that we shall use for proving Theorem~\ref{t:main} in the next two sections.


 \section{Proof of the sufficiency of the geometric conditions} \label{s:proofsufficiency}

We prove here the first part of Theorem~\ref{t:main}, that if $U$ satisfies condition {\bf (A)}, $T_{\rm GCC}(U) <+\infty$ and $T>T_{\rm GCC}(U)$, then the Schr\"odinger equation~\eqref{e:Schrod} is observable on $U$ in time $T$.

\subsection{Proof of localized observability.}\label{sec:proof}

We argue by contradiction. If~\eqref{obs_loc} is false, then there exist  $(u_0^k)_{k\in\N}$ and $(h_k)_{k\in\N}$ such that $u_0^k= \mathcal{P}_{h_k} u_0^k$,
\begin{equation}\label{hyp}
\| u^k_0\|_{L^2(M)}=1\;\;\mbox{and}\;\;  \int_0^T \left\|   {\rm e}^{it(\frac12 \Delta_M+\mathbb{V})} \mathcal{P}_{h_k} u_0^k\right\|_{L^2(U)}^2 dt\Tend{k}{ +\infty} 0.
\end{equation}
Because $u_0^k= \mathcal{P}_{h_k} u_0^k$ with $\chi$ compactly supported in an annulus (see~\eqref{def:Pih}) and $\mathbb{V}$ is bounded, the family $u_0^k$ is 
$h_k$-oscillating in the sense of Section \ref{s:epsosc}
and so it is for
$$\psi_k(t)=   {\rm e}^{it(\frac12 \Delta_M+\mathbb{V})} \mathcal{P}_{h_k} u_0^k.$$
We consider (after extraction of a subsequence if necessary), the semi-classical measure $\Gamma_t d\gamma_t$ of $\psi_k(t)$ given by Proposition~ \ref{prop:semiclas} and satisfying the properties listed in Proposition~\ref{p:measure0}.

\begin{proposition}\label{p:pmeasure}
We have the following facts:
\begin{enumerate}
\item
There holds
\begin{equation} \label{e:nullityoverU}
\int_0^T \int_{U\times \widehat G} {\rm Tr} (\Gamma_t(x,\lambda)) d\gamma_t (x,\lambda)dt=0.
\end{equation}
\item $\gamma_t$ is supported above $\mathfrak z^*\setminus \{0\}$ for almost every $t\in\R$. 
\end{enumerate}
\end{proposition}

\begin{proof}[Proof of Proposition \ref{p:pmeasure}]
To prove (1), let us recall that for $\theta\in L^1(\R)$ and $\sigma\in{\mathcal A}_0$,
\begin{equation} \label{e:convmeasure}
\int_{\R}\theta(t)(\text{Op}_{h_k}(\sigma)\psi_k(t),\psi_k(t))_{L^2(M)}dt\Tend{k}{ +\infty} \int_{\R\times M\times\widehat{G}}\theta(t)\text{Tr}(\sigma(x,\lambda)\Gamma_t(x,\lambda))d\gamma_t(x,\lambda)dt.
\end{equation}
We take $\varphi_j(x)$ a sequence of smooth non-negative functions converging to $\mathbf{1}_{U}(x)$, bounded above by $1$ and such that $\text{supp}(\varphi_j)\subset U$, and $\alpha\in C_c^\infty((-1,1))$ non-negative with $\alpha=1$ in a neighborhood of $0$.
Since $\psi_k(t)$ is uniformly $\eps$-oscillating for $\eps=h_k$, we have
\begin{align*}
\int_0^T \int _{\R\times M\times\widehat{G}}& \text{Tr}(\varphi_j(x)\Gamma_t(x,\lambda))d\gamma_t(x,\lambda)dt=\\
&\lim_{R\rightarrow +\infty} \lim_{k\rightarrow +\infty}
\int_0^T \left( \text{Op}_{h_k}( \varphi_j(x)  \alpha (R^{-1} H(\lambda)) ) \psi_k(t),\psi_k(t)\right)_{L^2(M)} dt.
\end{align*}
Besides, $\text{Op}_{h_k}( \varphi_j(x)  \alpha (R^{-1} H(\lambda)))= \varphi_j(x) \alpha(-h_k^2R^{-1} \Delta_M)$, thus
$$\| \text{Op}_{h_k}( \varphi_j(x)  \alpha (R^{-1} H(\lambda)))\|_{\mathcal L(L^2(M))}\leq 1$$
and
$$\left| \int_0^T  \left( \text{Op}_{h_k}( \varphi_j(x)  \alpha (R^{-1} H(\lambda)) ) \psi_k(t),\psi_k(t)\right)_{L^2(M)}\right| \leq
\int_0^T \|\psi_k(t)\| ^2_{L^2(U)} dt.$$
We deduce from \eqref{hyp} that
$$\int_0^T \int _{\R\times M\times\widehat{G}}\text{Tr}(\varphi_j(x)\Gamma_t(x,\lambda))d\gamma_t(x,\lambda)dt=0.$$
Taking the limit $j\rightarrow +\infty$ and using Lebesgue's dominated convergence theorem  (since $\Gamma_td\gamma_t\geq 0$), we get~\eqref{e:nullityoverU}.

\medskip

Point (2) 
follows from Point (1), the positivity of $\Gamma_td\gamma_t$,  Assumption (A) and Point (2) of Proposition \ref{p:measure0}.
\end{proof}

\medskip

Set
$$\gamma_{n,t} (x,\lambda)= {\rm Tr}\left(\Gamma_{n,t}(x,\lambda)\right) \gamma_t(x,\lambda).$$
We have obtained
\begin{align*}
0& =\sum_{n\in\N} \int_0^T \int_{U\times \widehat G} {\rm Tr} (\Gamma_{n,t}(x,\lambda)) d\gamma_t (x,\lambda)dt=\sum_{n\in\N} \int_0^T \int_{U\times \widehat G} d\gamma_{n,t} (x,\lambda)dt
\end{align*}
whence, the positivity of $\Gamma_t$ (and thus of $\gamma_{n,t}$) yields
$$ \int_{U\times \mathfrak z^*} d\gamma_{n,t} (x,\lambda)=0,\;\; \text{for almost every } t\in [0,T],\;\; \forall n\in\N,$$
where we have also used that the support of $d\gamma_{n,t}$ is above $\mathfrak z^*\setminus \{0\}$.

\medskip

We now use transport equation~\eqref{transport}.
For $n\in\N$ and $\lambda\in\mathfrak z^*\setminus \{0\}$, we set
$$Z_n(\lambda)= (n+\frac d 2)  {\mathcal Z}^{(\lambda)}$$
and we have
$$| Z_n(\lambda)|= n+\frac{d}{2}.$$
We
introduce the map $\Phi^s_n$ defined for $s\in\R$ and $n\in\N$ as an application from $M\times (\mathfrak z^*\setminus\{0\})$ to itself by
$$\Phi^s_n: (x,\lambda)\mapsto \left({\rm Exp}[s Z_n(\lambda)]x,\lambda\right).$$
The flows $\Phi_n^s$ and $\Phi_0^s$ are related by 
$$
\Phi_n^s(x,\lambda)=\Phi_0^{s'}(x,\lambda), \qquad s'=\left(\frac{2n}{d}+1\right)s.
$$

The transport equation \eqref{transport} implies  that for any interval $I$ and any $ \Lambda\subset M\times (\mathfrak z^*\setminus\{0\})$,
$$\frac d {ds} \left(\int _{(I+s)\times \Phi^{s}_n ( \Lambda)} d\gamma_{n,t} dt\right) =0,$$
which means
\begin{equation} \label{e:transportgamman}
\int _{(I+s)\times \Phi^{s}_n ( \Lambda)} d\gamma_{n,t} dt= \int _{I\times \Lambda} d\gamma_{n,t} dt.
\end{equation}

Since $T>T_{\rm GCC} (U)$, we may choose $T'$ such that $T_{\rm GCC} (U)<T'<T$ and (H-GCC) holds in time $T'$. Assume that there exists $\tau$ with $0<\tau<T-T'$ such that
\begin{equation} \label{e:seekcontrad}
\int_0^{\tau} \int_{M\times \mathfrak z^*} d\gamma_{t} dt >0.
\end{equation}
We seek for a contradiction.

\medskip 
Writing $\gamma_t=\sum_{n=0}^\infty \gamma_{n,t}$, with all $\gamma_{n,t}$ being non-negative Radon measures on $M\times (\mathfrak z^*\setminus \{0\})$ (since Point 2 of Proposition \ref{p:pmeasure} ensures that it has no mass on the trivial representation), we see that there exists $n_0\in\N$ and a bounded open subset $\Lambda\subset M\times (\mathfrak z^*\setminus\{0\})$ such that
$$\int_0^{\tau} \int_{\Lambda} d\gamma_{n_0,t} dt >0.$$
Fix $(x,\lambda)\in \Lambda$ and $s\in (0,T')$ such that  $\Phi^{s}_0 ((x,\lambda))\in U\times \mathfrak z^*$. Note that, making $\Lambda$ smaller if necessary, by continuity of the flow and using that $U$ is open,  $\Phi^{s}_0 ((x',\lambda'))\in U\times \mathfrak z^*$ for any $(x',\lambda')\in\Lambda$. Therefore $\Phi^{s(n_0)}_{n_0} ((x',\lambda'))\in U\times \mathfrak z^*$ for any $(x',\lambda')\in\Lambda$, where $s(n_0)=\frac{sd}{2n_0+d}$ (with a slight abuse of notation).

\medskip
From \eqref{e:nullityoverU}, we get
$$\gamma_{n_0,t} (\Phi^{s(n_0)}_{n_0} (\Lambda))=0,\;\;a.e.\; t\in(0,T),$$
and in particular
$$\int_{s(n_0)}^T \int_{ \Phi^{s(n_0)}_{n_0}(\Lambda)} d\gamma_{n_0,t} dt =0.$$
Therefore, by \eqref{e:transportgamman},
$$\int_0^{T-s(n_0)} \int_{\Lambda} d\gamma_{n_0,t} dt =0.$$
Since $\tau<T-T'<T-s(n_0)$, we get
$$\int_0^{\tau} \int_{\Lambda} d\gamma_{n_0,t} dt =0$$
which is a contradiction.
Therefore
\begin{equation*}
\int_0^{\tau} \int_{M\times \mathfrak z^*} d\gamma_{t} dt =0.
\end{equation*}
This implies $\gamma_t=0$ for almost every $t\in (0,\tau)$. In turn, this contradicts the fact that $\| \psi_k(t)\|_{L^2}=1$. Therefore \eqref{obs_loc} holds.

\begin{remark} \label{r:assumpA}
Assumption (A) corresponds to the usual Geometric Control Condition which is known to be a sufficient condition for the control/observation of the Riemannian Schr\"odinger equation (see \cite{Leb}). It is well known that, in the Riemannian setting, this condition is  not always necessary : it is not for  the Euclidean torus (see~\cite{J,AM14,BZ}) while it is for Zoll manifolds~\cite{Mac11} (these manifolds have geodesics that are all periodic); so, it depends on the manifold. As already mentioned in the introduction, the authors tend to think that in the particular case considered in this paper (quotients of H-type groups), Theorem \ref{t:main} still holds without this assumption. Assumption (A) has been used in the proof of Point (2) of Proposition \ref{p:pmeasure}, and it is the only place of the paper where we use it. By analogy with the results of~\cite{AM14,AFM15,BS}, it is likely that as in~\cite[Section 7]{BS}, a key argument should be a  reduction to a problem on the Euclidian torus, as those studied in~\cite{AFM15}  for example. Then, the semiclassical analysis of this reduced problem would show that the part of the measure $\gamma_t$ located above $M\times {\mathfrak v}^*$ vanishes. That would prove that H-type GCC alone is enough and would avoid the use of Assumption (A). 
\end{remark}

\subsection{Proof of weak observability}\label{s:weakobs}

We prove here {\bf \eqref{obs_loc} $ \Longrightarrow$ \eqref{obs_weak}}.

\medskip

Consider a partition of unity over the positive real half-line $\R^+$:
\begin{equation} \label{e:partition}
\forall x\in\R^+, \ \ 1=\chi_0(x)^2+\sum_{j=1}^\infty \chi_j(x)^2
\end{equation}
where, for $j\geq1$, $\chi_j(x)=\chi(2^{-j}x)$ with $\chi\in C_c^\infty((1/2,2),[0,1])$. To construct such a partition of unity, consider $\psi\in C_c^\infty((-2,2),[0,1])$ such that $\psi\equiv 1$ on a neighborhood of $[-1,1]$, and set $\chi(x)=\sqrt{\psi(x)-\psi(2x)}$ for $x\geq 0$, which is smooth for well-chosen $\psi$. Finally, define $\chi_0(x)$ for $x\geq 0$ by $\chi_0(x)^2=1-\sum_{j=1}^\infty \chi_j(x)^2$, so that $\chi_0(x)=0$ for $x\geq 2$. Then \eqref{e:partition} holds.

\medskip

We follow the proof of \cite[Proposition 4.1]{BZ}. Set $h_j=2^{\frac{-j}{2}}$ for $j\geq 1$, and note that $\mathcal{P}_{h_j}=\chi_j(-(\frac12\Delta_M+\mathbb{V}))$. We choose $K$ so that $h_K\leq h_0$, where $h_0$ is taken so that \eqref{obs_loc} holds for $0<h\leq h_0$.  We take $\varepsilon>0$ such that $T'+2\varepsilon<T$ and $\psi\in C_c^\infty((0,T),[0,1])$ with $\psi=1$ on a neighborhood of $[\varepsilon,T'+2\varepsilon]$. Then
\begin{align}
\|u_0\|_{L^2(M)}^2&=\sum_{j=0}^\infty \left\|\chi_j\left(-\frac12\Delta_M+\mathbb{V}\right)u_0\right\|^2_{L^2(M)} \nonumber\\
&= \sum_{j=0}^K \|\mathcal{P}_{h_j}u_0\|^2_{L^2(M)}+\sum_{j=K+1}^\infty \|\mathcal{P}_{h_j}u_0\|^2_{L^2(M)}\nonumber \\
&\leq C \left\|\left(\Id-(\frac12\Delta_M+\mathbb{V})\right)^{-1}u_0\right\|^2_{L^2(M)}+\sum_{j=K+1}^\infty \|\mathcal{P}_{h_j}u_0\|^2_{L^2(M)}\nonumber \\
&\leq C\|(\Id-\Delta_M)^{-1}u_0\|^2_{L^2(M)}+C\sum_{j=K+1}^\infty  \left\| \psi(t) {\rm e}^{it(\frac12 \Delta_M+\mathbb{V})}\mathcal{P}_{h_j} u_0\right\|^2_{L^2((0,T)\times U)} \nonumber
\end{align}
where in the third line we bounded above the low frequencies with a constant $C=C_K$, and in the last line we used \eqref{obs_loc} (with the term on $U$ being integrated for $t\in (\varepsilon,T'+2\varepsilon)$, which is of length $>T'$, see Remark \ref{r:masscons}). Note that we also used the fact that $\mathbb{V}$ is analytic and thus bounded, and therefore the resolvents of the operators $\frac12\Delta_M+\mathbb{V}$ and $\Delta_M$ are comparable in $L^2$ norm.
Using equation \eqref{e:Schrod}, we may change $\mathcal{P}_{h_j}=\chi_j(-(\frac12\Delta_M+\mathbb{V}))$ into $\chi_j(-D_t)$ where $D_t=\partial_t/i$. We get
\begin{align}
\|u_0\|_{L^2(M)}^2&\leq C\|(\Id-\Delta_M)^{-1}u_0\|^2_{L^2(M)}+C\sum_{j=K+1}^\infty  \left\|\psi(t) \chi_j(-D_t)  {\rm e}^{it(\frac12 \Delta_M+\mathbb{V})} u_0\right\|^2_{L^2((0,T)\times U)}
\end{align}
If $\widetilde{\psi}\in C_c^\infty((0,T),[0,1])$ satisfies $\widetilde{\psi}=1$ on $\text{supp}(\psi)$, we note that
\begin{align}
\psi(t)\chi_j(-D_t)&=\psi(t)\chi_j(-D(t))\widetilde{\psi}(t)+\psi(t)[\widetilde{\psi}(t),\chi_j(-D_t)] \nonumber \\
&=\psi(t)\chi_j(-D(t))\widetilde{\psi}(t)+E_j(t,D_t)
\end{align}
where $E_j$ is smoothing, i.e.,
\begin{equation*}
\partial^\alpha E_j=O(\langle t\rangle^{-N}\langle\tau\rangle^{-N}2^{-Nj})
\end{equation*}
for any $\alpha\in\N$, any $N\in\N$ and uniformly in $j$. This fact follows from symbolic calculus and the remark that, on the support of $\psi$, $\widetilde{\psi}$ is constant and 
all the derivatives of $\widetilde \psi$ are zero on the support of $\psi$. 

\medskip

Therefore, integrating by parts in the time variable in the second term of the right-hand side and absorbing the error terms $E_j(t,D_t)$ in $\|(\Id-\Delta_M)^{-1}u_0\|^2_{L^2}$, we get
\begin{align}
\|&u_0\|_{L^2(M)}^2\leq C\|(\Id-\Delta_M)^{-1}u_0\|^2_{L^2(M)}+C\sum_{j=K+1}^\infty  \|\psi(t)\chi_j(-D_t)\widetilde{\psi}(t)  {\rm e}^{it(\frac12 \Delta_M+\mathbb{V})} u_0\|_{L^2((0,T)\times U)}^2 \nonumber \\
&\leq C\|(\Id-\Delta_M)^{-1}u_0\|^2_{L^2(M)}+C\sum_{j=K+1}^\infty  \|\chi_j(-D_t)\widetilde{\psi}(t)  {\rm e}^{it(\frac12 \Delta_M+\mathbb{V})} u_0\|_{L^2((0,T)\times U)}^2 \nonumber \\
&= C\|(\Id-\Delta_M)^{-1}u_0\|^2_{L^2(M)}
+C\sum_{j=K+1}^\infty  \left(\chi_j(-D_t)^2\widetilde{\psi}(t)  {\rm e}^{it(\frac12 \Delta_M+\mathbb{V})} u_0\,,\,\widetilde{\psi}(t) {\rm e}^{it(\frac12 \Delta_M+\mathbb{V})} u_0\right)_{L^2((0,T)\times U)} \nonumber \\
&\leq C\|(\Id-\Delta_M)^{-1}u_0\|^2_{L^2(M)}
+C \left(\sum_{j=0}^\infty \chi_j(-D_t)^2 \widetilde{\psi}(t) {\rm e}^{it(\frac12 \Delta_M+\mathbb{V})} u_0\,,\, \widetilde{\psi}(t) {\rm e}^{it(\frac12 \Delta_M+\mathbb{V})} u_0\right)_{L^2((0,T)\times U)} \nonumber \\
&\leq  C\|(\Id-\Delta_M)^{-1}u_0\|^2_{L^2(M)}+C\| {\rm e}^{it(\frac12 \Delta_M+\mathbb{V})} u_0\|^2_{L^2((0,T)\times U)} \nonumber
\end{align}
where we used \eqref{e:partition} in the last line. This concludes the proof of $ \eqref{obs_weak}$.

\medskip

\subsection{Proof of observability} \label{s:fromweaktostrongobs}

We prove here \eqref{obs_weak} $ \Longrightarrow$ \eqref{obs}, which concludes the proof of the sufficiency of the geometric condition {\bf H-type GCC}. We follow the classical Bardos-Lebeau-Rauch argument, see for example \cite{BZ}.

\medskip

For $\delta\geq 0$, we set
\begin{equation*}
\mathcal{N}_\delta=\{u_0\in L^2(M) \ \mid \ e^{it(\frac{1}{2}\Delta_M+\mathbb{V})}u_0\equiv0 \text{  on } (0,T-\delta)\times U\}.
\end{equation*}
\begin{lemma} \label{l:blrlemma}
There holds $\mathcal{N}_0=\{0\}$.
\end{lemma}
\begin{proof}
Let $u_0\in \mathcal{N}_0$. We define
\begin{equation} \label{e:defveps}
v_{\epsilon,0}=\frac{1}{\epsilon}\left(e^{i\epsilon(\frac{1}{2}\Delta_M+\mathbb{V})}-\Id\right)u_0.
\end{equation}
If $\epsilon\leq \delta$, then $e^{it(\frac{1}{2}\Delta_M+\mathbb{V})}v_{\epsilon,0}=0$ on $(0,T-\delta)\times U$. We write $u_0$ in terms of orthonormal eigenvectors~$f_\lambda$ of $\frac12\Delta_M+\mathbb{V}$ (associated with $\lambda\in \text{Sp}$, the spectrum of $\frac12\Delta_M+\mathbb{V}$):
\begin{equation*}
u_0=\sum_{\lambda\in\text{Sp}} u_{0,\lambda}f_\lambda
\end{equation*}
For small enough $\alpha,\beta$, applying \eqref{obs_weak} with a slightly smaller $T$, we have
\begin{align}
\|v_{\alpha,0}-v_{\beta,0}\|^2_{L^2}&\leq   C \|(\Id-\Delta_M)^{-1}(v_{\alpha,0}-v_{\beta,0})\|^2_{L^2} \nonumber \\
&\leq C \|(\Id-(\frac12\Delta_M+\mathbb{V}))^{-1}(v_{\alpha,0}-v_{\beta,0})\|^2_{L^2} \nonumber \\
&\leq C\sum_{\lambda\in\text{Sp}}\left|\frac{e^{i\alpha\lambda}-1}{\alpha}-\frac{e^{i\beta\lambda}-1}{\beta}\right|^2(1+\lambda)^{-2}|u_{0,\lambda}|^2\nonumber \\
&\leq C\sum_{\lambda\in\text{Sp}}\lambda^2|\alpha-\beta|^2(1+\lambda)^{-2}|u_{0,\lambda}|^2\nonumber \\
&\leq C|\alpha-\beta|^2. \nonumber
\end{align}
Hence there exists $v_0\in L^2(M)$ such that $v_0=\lim_{\alpha\rightarrow 0}v_{\alpha,0}$ where the limit is taken in $L^2(M)$. This limit is necessarily in $\mathcal{N}_\delta$ for all $\delta>0$, hence in $\mathcal{N}_0$. Moreover, thanks to \eqref{e:defveps}, there holds in the sense of distributions
\begin{equation*}
e^{it(\frac{1}{2}\Delta_M+\mathbb{V})}v_0=\partial_te^{it(\frac{1}{2}\Delta_M+\mathbb{V})}u_0
\end{equation*}
and therefore
\begin{equation*}
v_0=i(\frac{1}{2}\Delta_M+\mathbb{V}) u_0.
\end{equation*}
Therefore $\frac{1}{2}\Delta_M+\mathbb{V}:\mathcal{N}_0\rightarrow \mathcal{N}_0$ is a well-defined operator. Moreover, according to \eqref{obs_weak}, on $\mathcal{N}_0$, we have
\begin{equation*}
\|(\Id-\Delta_M)\cdot\|_{L^2(M)}\leq C\|\cdot\|_{L^2(M)}
\end{equation*}
and, by compact embedding (see Lemma \ref{l:compact} below), the unit ball of $\mathcal{N}_0\subset L^2(M)$ is compact. Hence $\mathcal{N}_0$ is finite dimensional and there exists an eigenfunction $w\in\mathcal{N}_0$ of $\frac12\Delta_M+\mathbb{V}:\mathcal{N}_0\rightarrow \mathcal{N}_0$, i.e.,
\begin{equation*}
(\frac12\Delta_M+\mathbb{V})w=\mu w, \ \ \ w_{|U}=0.
\end{equation*}
By a standard unique continuation principle (see \cite{Bon69} and \cite[Theorem 1.12]{Laur}), since $\mathbb{V}$ and $\Delta_M$ are analytic (see \cite[Section 5.10]{BLU} for example), we conclude that $w=0$, hence $\mathcal{N}_0=\{0\}$.
\end{proof}
\begin{remark}\label{rem:analyticity}
To our knowledge, the unique continuation principle used in the above proof is only known when $\mathbb{V}$ is analytic. In $C^\infty$ regularity, counterexamples to the unique continuation principle exist, see \cite{Ba}. However, the result of Theorem~\ref{t:main} holds as soon as a unique continuation principle holds for~$\frac12\Delta_M+\mathbb V$.  
\end{remark}

\begin{lemma} \label{l:compact}
Set
\begin{equation*}
\mathcal{H}(M)=\{u\in L^2(M) \ \mid \ ({\rm Id}-\Delta_M)u \in L^2(M)\}.
\end{equation*}
Then $\mathcal{H}(M)\hookrightarrow L^2(M)$ with compact embedding.
\end{lemma}
\begin{proof}
By \cite[Corollary B.1]{Laur}, we have $\|u\|_{H^1(M)}\leq \|(\Id-\Delta_M)u\|_{L^2(M)}$ since $G$ is step $2$. Therefore, $\mathcal{H}(M)\hookrightarrow H^1(M)$ continuously. The result then follows by the Rellich-Kondrachov (compact embedding) theorem.
\end{proof}

Assume that ~\eqref{obs} does not hold. Then there exists a sequence $(u_0^k)_{k\in\N}$ such that
\begin{equation} \label{e:contradictionargument}
\|u_0^k\|_{L^2(M)}=1\;\;\mbox{and}\;\; \int_0^T \left\| {\rm e}^{it(\frac12 \Delta_M+\mathbb{V})} u_0^k\right\|^2_{L^2(U)} dt
\Tend{k}{+\infty} 0.
\end{equation}
Since $(u_0^k)_{k\in\N}$ is bounded in $L^2(M)$, we can extract from $(u_0^k)_{k\in\N}$ a subsequence which converges weakly to some $u^\infty$ in $L^2(M)$. By Lemma \ref{l:compact}, we then have $(\Id-\Delta_M)^{-1}u_0^k\rightarrow (\Id-\Delta_M)^{-1}u^\infty$ strongly in $L^2(M)$. Moreover, the second convergence in \eqref{e:contradictionargument} gives $u^\infty\in\mathcal{N}_0$. Thanks to \eqref{obs_weak}, we know that
\begin{equation*}
\|u_0^k\|^2_{L^2(M)} \leq C_1 \int_0^T \left\|  {\rm e}^{it(\frac12 \Delta_M+\mathbb{V})} u_0^k\right\|^2_{L^2(U)} dt
+ C_1 \left\| (\Id-\Delta_M)^{-1} u_0^k\right\|^2_{L^2(M)}.
\end{equation*}
Therefore, taking the limit $k\rightarrow +\infty$, we get
\begin{equation*}
1\leq C_1 \| (\Id-\Delta_M)^{-1} u^\infty\|^2_{L^2(M)}.
\end{equation*}
Therefore $u^\infty\neq 0$, which contradicts Lemma \ref{l:blrlemma} since  $u^\infty\in\mathcal N_0$. Hence, \eqref{obs} holds.



\section{Non-commutative wave packets and the necessity of the geometric control}\label{sec:WP}

In this section, we  conclude the proof of Theorem~\ref{t:main} and prove the necessity of the condition~{\bf (H-GCC)} (for $\overline{U}$). We use special data that we call non-commutative wave packets that we first introduce, together with their properties, on which we also elaborate in Appendix~\ref{a:wpsolutions}. Then, we   conclude to the necessity of the H-type GCC.

\subsection{Non-commutative wave packets} \label{s:WPnoncom}
Let us first briefly recall basic facts about classical (Euclidean) wave packets. Given $(x_0,\xi_0)\in\R^d\times\R^d$ and $a\in\mathcal S(\R^d)$, we consider the family (indexed by $\eps$) of functions
\begin{equation}\label{WPeuclid}
u_{\rm eucl}^\eps(x)= \eps^{-d/4} a\left(\frac{x-x_0}{\sqrt\eps}\right) {\rm e}^{\frac i\eps \xi_0\cdot(x-x_0)}, \;\; x\in\R^d.
\end{equation}
 Such a family is called a (Euclidean) wave packet.

The oscillation along $\xi_0$ is forced by the term ${\rm e}^{\frac i\eps \xi_0\cdot(x-x_0)}$ and the concentration on $x_0$ is performed at the scale $\sqrt\eps$ for symmetry reasons :  the $\eps$-Fourier transform of $u^\eps_{\rm eucl}$,
$\eps^{-d/2}\widehat u^\eps_{\rm eucl} (\xi/\eps)$ presents a concentration on $\xi_0$ at the scale $\sqrt\eps$. The regularity of the wave packets makes them   a flexible tool. Besides, taking $a$ compactly supported in the interior of a fundamental domain  for the torus, one can generalize their definition to the case of the torus by extending them by periodicity.  For example, let us consider the torus $\T^d=\R^d/(2\pi\Z)^d$, we choose $a\in{\mathcal C}_c^\infty((-\pi,\pi)^d)$ and we define $a_\eps(x) $ as
$$a_\eps(x) = a\left(  \frac{x-x_0}{\sqrt\eps}\right).$$
We consider the periodisation operator $\mathbb P$ which associates with a function $\varphi$ compactly supported inside a set of the form $x_0+(-\pi,\pi)^d$ the periodic function defined on the sets $k+x_0+(-\pi,\pi)^d$ for $k\in(2\pi\Z)^d$ by
$\mathbb P \varphi(x)= \varphi(x-k)$. Then, the definition of a wave packet extends to functions on the torus by setting
$$u_{\rm torus}^\eps(x)= \eps^{d/4} \mathbb P a_\eps(x){\rm e}^{\frac i\eps \xi_0\cdot(x-x_0)}.$$

\medskip

We introduce here a generalization of these wave packets to the non-commutative setting of  Lie groups and nilmanifolds, in the context of $H$-type groups, which is strongly inspired by~\cite{FF1}.
For $x\in G$, we write
$$x={\rm Exp}(V+Z)= x_{\mathfrak z} x_{\mathfrak v} = x_{\mathfrak v} x_{\mathfrak z} \;\;{\rm  with}\;\;V\in{\mathfrak v},\;\; Z\in{\mathfrak z},$$
where
$$x_{\mathfrak z} ={\rm e}^Z\in G_{\mathfrak z}:={\rm Exp} (\mathfrak z)\;\;{\rm  and }\;\;x_{\mathfrak v}={\rm e}^V\in G_{\mathfrak v}:=G/ G_{\mathfrak z} .$$
 The concentration  is performed by use of dilations: with $a\in\mathcal C_c^\infty(G)$, we associate
$$a_\eps(x)=a \left(\delta_{\eps^{-1/2 }}(x )\right).$$
The oscillations are forced by using coefficients of the representations, in the spirit of~\cite{pedersen}: with $\lambda_0\in\mathfrak z^*$,
$\Phi_1,\; \Phi_2$ smooth vectors in the space of representations, i.e. in ${\mathcal S}(\R^d)$, we associate the oscillating term
$$e_\eps(x)=\left(\pi^{\lambda_\eps}_x \Phi_1, \Phi_2\right), \;\;\lambda_\eps= \frac{\lambda_0}{\eps^2}.$$
We restrict to $\eps\in (0,1)$ and define the periodisation operator $\mathbb P$ in analogy with the case of the torus described above, using the multiplication on the left by elements of $\widetilde\Gamma$. We consider a subset~$\mathcal B$ of~$G$ which is a neighborhood of~$1_G$ and such that~$\cup_{\gamma\in \widetilde{\Gamma}} (\gamma {\mathcal B} )=G$ and we choose functions~$a$ that are in~${\mathcal C}^\infty_c(\mathcal B)$ (in other words, their support is a subset of the interior of $\mathcal B$).

\begin{proposition}\label{prop:WP}
Let $\Phi_1,\Phi_2\in \mathcal S(\R^d)$, $a\in C_c^\infty(\mathcal B)$, 
$x_0\in M$,  $\lambda_0\in\mathfrak z^*\setminus\{0\}$. Then, there exists $\eps_0>0$ such that  the family~$(v^\eps)_{\eps\in(0,\eps_0)}$  defined by
$$v^\eps (x)=|\lambda_\eps|^{d/2}\,\eps^{-p/2}
\, {\mathbb P} (e_\eps a_\eps)( x_0^{-1} x),$$
 is a  bounded $\eps$-oscillating family in $L^2(M)$ with bounded $\eps$-derivatives and momenta:
\begin{equation}\label{eq:WPsobolev}
\forall k\in\N,\;\;\exists C_k>0,\;\;\forall \eps\in(0,\eps_0),\;\;  \| (-\eps^2\Delta_M)^{k/2} v^\eps\|_{L^2(M)}\leq  C_k.
\end{equation}
Moreover, $(v^\eps)_{\eps\in(0,\eps_0)}$ has only one  semi-classical measure $\Gamma d\gamma$  where
\begin{equation}\label{eq:WPmeasures}
 \gamma=c_a\, \delta(x-x_0) \otimes \delta (\lambda-\lambda_0),\;\; c_a= \|\Phi_2\|^2\int_{G_{\mathfrak z}} |a( x_{\mathfrak z})|^2 d x_{\mathfrak z},
 \end{equation}
 and $\Gamma$ is the operator
defined by
$$\Gamma\Phi= \frac{(\Phi,\Phi_1)}{\|\Phi_1\|^2} \Phi_1,\;\;\forall \Phi\in L^2(\R^d).$$
\end{proposition}

\medskip

In the following, we shall say that the family $v^\eps$ is a wave packet on $M$ with cores $(x_0,\lambda_0)$, profile~$a$ and harmonics $(\Phi_1,\Phi_2)$, and write
$$v^\eps= WP^\eps_{x_0,\lambda_0}(a,\Phi_1,\Phi_2)=|\lambda_\eps|^{d/2}\,\eps^{-p/2}
\, {\mathbb P} (e_\eps a_\eps)(  x_0^{-1}x).$$

\begin{remark}\label{rem:toto}
\begin{enumerate}
\item Note that $\eps_0$ is chosen small enough so that  for $\eps\in(0,\eps_0)$, the function $G\ni x\mapsto a_\eps(x)$ has support  included in a fundamental domain of $G$ for $\widetilde{\Gamma}$ and thus $x\mapsto (e_\eps a_\eps)(x_0^{-1}x)$ can be extended by periodicity on $G$, which defines a function of $M$.
\item Omitting the periodisation operator $\mathbb P$, we construct wave packets on~$G$ that also satisfy estimates in momenta
$$\forall k\in\N,\;\;\exists C_k>0,\;\;\forall \eps>0,\;\; \sum_{1\leq p+q\leq k} \| |x|^p (-\eps^2\Delta_G)^{q/2} v^\eps\|_{L^2(G)}\leq  C_k.$$
\item  The coefficient~$|\lambda_\eps|^{d/2}\eps^{-p/2}$  guarantees the boundedness in $L^2(M)$ of the family~$(v^\eps)_{\eps>0}$.
\item {\it Characterization of wave packets.} Let
$x\in M$ be identified to a point of $G$ and let us fix $\Phi_1$, $\Phi_2$, $x_0$ and $\lambda_0$. Then, $v^\eps$ is a wave packet on $M$ if there exist  $x_0\in M$,  $\lambda_0\in\mathfrak z^*\setminus\{0\}$, $a\in C_c^\infty(\mathcal B)$ and $\Phi_1,\Phi_2\in \mathcal S(\R^d)$, such that
\begin{align}\label{e:carac}
\eps^{Q/4} v^\eps(x_0\delta_{\sqrt\eps} (x)) & =|\lambda_\eps| ^{d/2}  \eps^{Q/4-p/2}a(x)  ( \Phi_1, (\pi^{\lambda_0} _{\delta_{\eps^{-1/2} }(x)})^*\Phi_2)\\
\nonumber
&= |\lambda_0|^{d/2} \eps^{-d/2}a(x) ( \Phi_1, (\pi^{\lambda_0} _{\delta_{\eps^{-1/2} }(x)})^*\Phi_2).
\end{align}
\item {\it Generalization.} The construction we make here extends to more general Lie groups following ideas from Section~6.4 in~\cite{FF1} and~\cite{pedersen}.
\end{enumerate}
\end{remark}

\subsection{Proof of Proposition~\ref{prop:WP}} The proof of Proposition \ref{prop:WP} is relatively long, and we decompose it into several steps.

\medskip

 \subsubsection{The norm of wave packets} \label{s:norm}
 By the definition of the periodisation operator $\mathbb P$,
 $$\int_M | v^\eps(x)|^2 dx=  |\lambda_\eps|^{d} \eps^{-p}\int_{G} |a_\eps(x_0^{-1}x)|^2 |e_\eps(x_0^{-1}x) |^2dx.$$
 We then use~\eqref{e:carac} and we write
  \begin{align*}
 \| v^\eps\|^2_{L^2(G)}&= |\lambda_0|^d \eps^{-d}
  \int_G |a(x)|^2 |(\pi^{\lambda_0}_{\delta_{\eps^{-1/2}} x} \Phi_1,\Phi_2)|^2 dx\\
&= |\lambda_0|^d  \int_G |a(\delta_{\sqrt\eps} (x_{\mathfrak v}) x_{\mathfrak z})|^2 |(\pi^{\lambda_0}_{x_{\mathfrak v} }\Phi_1,\Phi_2)|^2 dx_{\mathfrak v} dx_{\mathfrak z}\\
 &\leq\left(\int_{G_{\mathfrak z} }  \sup_{y_{\mathfrak v} \in G_{\mathfrak v}} |a (y_{\mathfrak v} x_{\mathfrak z})|^2 dx_{\mathfrak z} \right)
 \left( |\lambda_0|^d  \int_{G_{\mathfrak v}}|(\pi^{\lambda_0}_{x_{\mathfrak v}} \Phi_1, \Phi_2)|^2 dx_{\mathfrak v}\right).
 \end{align*}

Let us note that the following relation holds for any $\Phi, \widetilde{\Phi},\Psi,\widetilde{\Psi}\in \mathcal S(\R^d)$:
\begin{equation} \label{e:intPhi}
|\lambda_0|^d \int_{G_{\mathfrak{v}}} (\pi_{x_{\mathfrak v}}^{\lambda_0}\Phi,\Psi)\overline{ (\pi_{x_{\mathfrak v}}^{\lambda_0}\widetilde{\Phi},\widetilde{\Psi})}dx_{\mathfrak v} = (\Phi,\widetilde{\Phi})\overline{(\Psi,\widetilde{\Psi})}.
\end{equation}
Therefore,
$$|\lambda_0|^d \int_{G_{\mathfrak v}}|(\pi^{\lambda_0}_{x_{\mathfrak v}} \Phi_1, \Phi_2)|^2 dx_{\mathfrak v}=\| \Phi_1\|^2 \| \Phi_2\|^2.$$
We deduce that $v^\eps$ is uniformly bounded in $L^2(G)$.

\subsubsection{The  $\eps$-oscillation and the  regularity of wave packets.}
  Straightforward computations give that if
  $\lambda\in {\mathfrak z}^*\setminus \{0\}$,  $\Phi_1,\Phi_2\in \mathcal S(\mathbb R^d)$,
$x_{\mathfrak v} = {\rm Exp} [P+Q]$,  $x=x_{\mathfrak v}x_{\mathfrak z}$ with
$$P=\sum_{j=1}^{d} p_j P_j^{(\lambda)}\;\;{\rm  and}  \;\;Q=\sum_{j=1}^{d} q_j Q_j^{(\lambda)} ,$$
 then, for $1\leq j\leq d$,
\begin{equation}\label{est:growth}
\sqrt{|\lambda|} \,q_j \left(\pi^{\lambda}_{x} \Phi_1,\Phi_2\right)=  \left([ \pi^{ \lambda}_{x}, i\partial_{\xi_j}]\Phi_1,\Phi_2\right) ,\;\;
\sqrt{|\lambda|} \, p_j \left(\pi^{ \lambda}_{x} \Phi_1,\Phi_2\right)=  \left([\pi^{ \lambda}_{x} ,\xi_j]\Phi_1,\Phi_2\right) .
\end{equation}
Besides,
\begin{equation}\label{eq:PQWP}
P_j^{(\lambda)} \left(\pi^{\lambda}_{x} \Phi_1,\Phi_2\right) =  \sqrt{|\lambda|} \left(  \partial_{\xi_j}  \pi^{\lambda}_{x}\Phi_1,\Phi_2\right)\;\;\mbox{and}\;\;
Q_j^{(\lambda)} \left(\pi^{\lambda}_{x} \Phi_1,\Phi_2\right) = i\sqrt{|\lambda|} \left(
\xi_j \pi^\lambda_{x} \Phi_1,\Phi_2\right).
\end{equation}
For proving this formula for $P_j^{(\lambda)}$, we use~\eqref{eq:Xf} and we observe
$${\rm Exp} (tP_j^{(\lambda)}) {\rm Exp} (P+Q+Z) = {\rm Exp} (tP_j^{(\lambda)}+P+Q+Z +\frac t2 [P_j^{(\lambda)},P+Q]).$$
Since $[P_j^{(\lambda)}, Q_j^{(\lambda)}]= \mathcal Z^{(\lambda)}$ and for $k\not=j$,
$[P_j^{(\lambda)}, P_k^{(\lambda)}]=[P_j^{(\lambda)},Q_k^{(\lambda)}]=0$,
we deduce
$${\rm Exp} (tP_j^{(\lambda)}) {\rm Exp} (P+Q+Z) = {\rm Exp} (tP_j^{(\lambda)}+P+Q+Z +\frac t2 q_j {\mathcal Z}^{(\lambda)}).$$
Therefore, using $\lambda(\mathcal Z^{(\lambda)})= |\lambda|$, we obtain for $\Phi\in\mathcal S(\R^d)$ and $\xi\in\R^d$,
$$\left.\frac d {dt} \left( \pi^\lambda_{ {\rm Exp} (tP_j^{(\lambda)}) x} \Phi(\xi)\right)\right|_{t=0} =
\sqrt{|\lambda|} \pi^\lambda_{  x} \partial_{\xi_j} \Phi(\xi)
+i |\lambda| q_j  \pi^\lambda_{ x} \Phi(\xi)=\sqrt{|\lambda|} \partial_{\xi_j} \pi^\lambda_{ x}\Phi(\xi) .$$
The proof for $Q_j^{(\lambda)}$ is similar.
We deduce
~\eqref{eq:WPsobolev} and that the family $(v^\eps)$ is uniformly $\eps$-oscillating by the Sobolev criteria of Proposition~4.6 in~\cite{FF1}.

\subsubsection{Action of pseudodifferential operators on wave packets. }

For studying their semi-classical measure, it is convenient to analyze first the action of pseudodifferential operators on wave packets.

  \begin{lemma}\label{lem:pseudo}
  Let $\Phi_1,\;\Phi_2\in\mathcal S(\R^d)$, $(x_0,\lambda_0)\in G\times \mathfrak z^*$, $a\in{\mathcal C}^\infty_c(\mathcal B)$. Let $\sigma\in{{\mathcal A}_0}$ compactly supported in an open set $\Omega$ such that $\overline \Omega$ is strictly included in a fundamental domain $\mathcal B$ of $\widetilde \Gamma$. Then there exist $\eps_1>0$ and $c_1>0$ such that for all $\eps\in(0,\eps_1)$,  
  $$\| {\rm Op}_\eps(\sigma) WP^\eps_{x_0,\lambda_0}(a,\Phi_1,\Phi_2)- WP^\eps_{x_0,\lambda_0}(a,\sigma(x_0,\lambda_0)\Phi_1,\Phi_2)\|_{L^2(M)}  \leq c_1\, \sqrt\eps.$$
\end{lemma}

\begin{remark} \label{rem:asympexpansion}
 The proof we perform below shows that there exist sequences of profiles $(a_j)_{j\in\N}$ and of harmonics $(\Phi_1^{(j)}, \Phi_2^{(j)})_{j\in\N}$ such that for all $N\in\N$,
 $$\| {\rm Op}_\eps(\sigma) WP^\eps_{x_0,\lambda_0}(a,\Phi_1,\Phi_2)- \sum_{j=0}^N \eps^{\frac j 2} WP^\eps_{x_0,\lambda_0}(a_j,\Phi_1^{(j)},\Phi_2^{(j)})\|_{L^2(M)}  \leq c_1\, (\sqrt\eps)^{N+1}.$$
 Moreover, by commuting the operator $(-\eps^2 \Delta_G)^{s/2}$ with the pseudodifferential operators, one can extend this result in Sobolev spaces.
 Note also that the same type of expansion holds in $G$, in refined functional spaces where momenta are controlled:
$$
   \| {\rm Op}_\eps(\sigma) WP^\eps_{x_0,\lambda_0}(a,\Phi_1,\Phi_2)- \sum_{j=0}^N \eps^{\frac j 2} WP^\eps_{x_0,\lambda_0}(a_j,\Phi_1^{(j)},\Phi_2^{(j)})\|_{\Sigma^k_\eps(G)}  \leq c_1\, \eps^{\frac{N+1}{2}}$$
 where $\Sigma^k_\eps$ is the vector space of functions $f\in L^2(G)$ for which the semi-norms
\begin{equation} \label{e:Sigmakeps}
\| f\|_{\Sigma^k_\eps} := \sum_{\ell=0}^k\left( \| |x|^\ell f\|_{L^2(G)} +  \| (-\eps^2 \Delta_G)^{\ell/2} f\|_{L^2(G)}\right)
\end{equation}
 are finite.
 \end{remark}

\begin{proof}
We first observe that, in view of Remark~\ref{correspGM}, it is enough to prove the result for wave packets in $G$. Indeed, consider $\chi\in\mathcal C^\infty_c(\overline{\mathcal B})$ with $\chi\sigma=\sigma$. Then for any function $f\in \mathcal C^\infty_c(\overline{\mathcal B})$ and   $x\in M$ identified to the point $x$ of $G\cap \mathcal B$, we have for all $N\in\N$, thanks to \eqref{eq:localisation},
\begin{align*}
{\rm Op}_\eps(\sigma) \mathbb P (f) (x) & = {\rm Op}_\eps (\sigma) \chi \mathbb P(f) (x) +O(\eps ^N)\\
&=  {\rm Op}_\eps (\sigma) \chi f (x) +O(\eps ^N)= {\rm Op}_\eps(\sigma) f(x) +O(\eps^N).
\end{align*}
Therefore, we are going to prove the result of Lemma~\ref{lem:pseudo} for wave packets and pseudodifferential operators in $G$.
Besides,
for simplicity, we assume that $\sigma(x,\cdot)$ is the Fourier transform of a compactly supported function. This technical assumption simplifies the proof which extends naturally to symbols that are Fourier transforms of Schwartz class functions.

\medskip

We write
\begin{align*}
{\rm Op}_\eps(\sigma) v^\eps(x) &=c_0|\lambda_\eps|^{d/2} \eps^{-p/2}
 \int_{G\times \widehat G}  {\rm Tr} (\pi^\lambda_{y^{-1} x}\sigma(x,\eps^2\lambda) )
a_\eps(x_0^{-1}y) (\pi_{x_0^{-1}y}^{\lambda_\eps} \Phi_1,\Phi_2) |\lambda|^d d\lambda dy\\
&= c_0|\lambda_\eps|^{d/2} \eps^{-p/2}
 \int_{G\times \widehat G}  {\rm Tr} (\pi^\lambda_{y^{-1} x_0^{-1} x}\sigma(x,\eps^2\lambda) )
a_\eps(y) (\pi_{y}^{\lambda_\eps} \Phi_1,\Phi_2) |\lambda|^d d\lambda dy.
\end{align*}
where we have performed the change of variable $y\mapsto x_0y$. We now focus on $\eps^{-Q/4}{\rm Op}_\eps(\sigma) v^\eps(x_0\delta_{\sqrt\eps} x)$ in order to simplify the computations. Note that this quantity is uniformly bounded in $L^2(G)$.
 \begin{align*}
{\rm Op}_\eps(\sigma) v^\eps(x_0\delta_{\sqrt\eps}x) &
&= c_0 |\lambda_\eps|^{d/2}\eps^{-p/2}
 \int_{G\times \widehat G}  {\rm Tr} (\pi^\lambda_{y^{-1}  \delta_{\sqrt\eps} x}\sigma(x_0 \delta_{\sqrt\eps} x ,\eps^2\lambda)
a_\eps(y) (\pi_{y}^{\lambda_\eps} \Phi_1,\Phi_2) |\lambda|^d d\lambda dy.
\end{align*}
We perform the change of variable $\widetilde y = \delta_{\eps^{-1/2}} y$ and $\widetilde\lambda= \eps^2 \lambda$.
We have
$$\pi^{\lambda} _{y^{-1} \delta_{\sqrt\eps} x}= \pi^{\widetilde \lambda/\eps^2}_{\delta_{\sqrt\eps}(y^{-1} x)}=
\pi^{\widetilde\lambda}_{\delta_{\eps^{-1/2}} (\widetilde y ^{-1}x )},
\;\;
\pi^{\lambda_\eps}_y= \pi^{\lambda_0/\eps^2}_{ \delta _{\sqrt\eps}\widetilde y} = \pi^{\lambda_0} _{ \delta_{\eps^{-1/2}} (y)}$$
and
$$|\widetilde\lambda|^d d\widetilde \lambda d\widetilde y = \eps^{2d}\eps^{2p} \eps^{-Q/2} |\lambda|^d d\lambda dy= \eps^{Q/2}  |\lambda|^dd\lambda dy.$$
We obtain
 \begin{align*}
{\rm Op}_\eps(\sigma) v^\eps(x_0\delta_{\sqrt\eps} x) =\, & c_0|\lambda_\eps|^{d/2}\eps^{-p/2}\eps^{-Q/2} \\
&\;\; \times \int_{G\times \widehat G}  {\rm Tr} (\pi^{\lambda}_{ \delta_{\eps^{-1/2}} (y^{-1} x)}\sigma(x_0\delta_{\sqrt\eps} x,\lambda))
a(y) (\pi_{\delta_{\eps^{-1/2}}(y)}^{\lambda_0} \Phi_1,\Phi_2) |\lambda|^d d\lambda dy.
\end{align*}
The change of variables $w=\delta_{\eps^{-1/2}} (y^{-1} x)$ (for which $dy= \eps^{Q/2} dw$ and $y=x(\delta_{\sqrt\eps }w)^{-1}$)) gives
\begin{align*}
{\rm Op}_\eps(\sigma) v^\eps(x_0\delta_{\sqrt\eps}x)
=\, & c_0 |\lambda_\eps|^{d/2}\eps^{-p/2} \\
&\;\;\times  \int_{G\times \widehat G}  {\rm Tr} (\pi^{\lambda}_{w  }\sigma(x_0\delta_{\sqrt\eps}  x,\lambda) )
a(x (\delta_{\sqrt\eps} w)^{-1}) (\pi_{ (\delta_{\eps^{-1/2}} (x))w^{-1}  }^{\lambda_0} \Phi_1, \Phi_2) |\lambda|^d d\lambda dw\\
=& \; c_0 |\lambda_\eps|^{d/2}\eps^{-p/2} \\
&\;\;\times  \int_{G\times \widehat G}  {\rm Tr} (\pi^{\lambda}_{w  }\sigma(x_0\delta_{\sqrt\eps}  x,\lambda) )
a(x (\delta_{\sqrt\eps} w)^{-1})
(\pi_{ w^{-1}  }^{\lambda_0} \Phi_1,
(\pi^{\lambda_0}_{\delta_{\eps^{-1/2} }(x)})^*\Phi_2) |\lambda|^d d\lambda dw.
\end{align*}
Computing the integral in $\lambda$ thanks to the inverse Fourier transform formula \eqref{inversionformula} and denoting by~$\kappa_x$ the Schwartz function such that $\sigma(x,\cdot)=
\mathcal F(\kappa_x)$ we have
$$
\eps^{Q/4} {\rm Op}_\eps(\sigma) v^\eps(x_0\delta_{\sqrt\eps}x)
= |\lambda_0|^{d/2}\eps^{-d/2}   \int_{G}  \kappa_{x_0\delta_{\sqrt\eps}  x}(w )
a(x (\delta_{\sqrt\eps} w)^{-1})
(\pi_{ w^{-1}  }^{\lambda_0} \Phi_1,
(\pi^{\lambda_0}_{\delta_{\eps^{-1/2} }(x)})^*\Phi_2)  dw$$
that we can rewrite
$$
\eps^{Q/4} {\rm Op}_\eps(\sigma) v^\eps(x_0\delta_{\sqrt\eps}x)
= |\lambda_0|^{d/2}\eps^{-d/2} \left( Q^\eps(x)  \Phi_1,
(\pi^{\lambda_0}_{\delta_{\eps^{-1/2} }(x)})^*\Phi_2\right)
$$
with
$$Q^\eps(x)=
 \int_{G}  \kappa_{x_0\delta_{\sqrt\eps}  x}(w )
a(x (\delta_{\sqrt\eps} w)^{-1})
\pi_{ w^{-1}  }^{\lambda_0} dw.$$
By performing a Taylor formula on the functions $x\mapsto \kappa_{x_0\delta_{\sqrt\eps}  x}(w )$ and $w \mapsto a(x (\delta_{\sqrt\eps} w)^{-1}) $, we see that the operator $Q^\eps(x)$ admits a formal asymptotic expansion of the form
\begin{equation}\label{expansion:Q}
Q^\eps(x)= Q_0(x)+\sqrt\eps Q_1(x) + \cdots +\eps^{\frac j 2} Q_j(x) +\cdots
\end{equation}
with
$$Q_0(x)= a(x) \int_{G}  \kappa_{x_0}(w )
\pi_{ w^{-1}  }^{\lambda_0} dw=a(x) \sigma(x_0,\lambda_0).$$
It remains to prove the convergence of this asymptotic expansion by examining the remainder term.

\medskip

We examine the one-term expansion.
We write
\begin{equation} \label{e:aA}
a(x (\delta_{\sqrt\eps} w)^{-1})  = a(x) +   A(x,\delta_{\sqrt\eps} w)
\end{equation}
with
\begin{equation}\label{est:Aeps}
|A (x,w)| \leq \sum_{j=1}^{2d}  \sup_{|z|\leq |w|}  |z_j|   | V_j a (xz)|\leq C_a |w| ,
\end{equation}
where for $z\in G$, $|z|$ denotes the homogeneous norm defined in~\eqref{def:quasinorm}.
We obtain
\begin{equation} \label{e:devQ}
\eps^{Q/4} {\rm Op}_\eps(\sigma) v^\eps(x_0\delta_{\sqrt\eps}x) =
|\lambda_0|^{d/2}\eps^{-d/2} \left( Q_0 \Phi_1,
(\pi^{\lambda_0}_{\delta_{\eps^{-1/2} }(x)})^*\Phi_2\right) a(x) +\sqrt\eps  r_1^\eps(x)+\sqrt\eps  r_2^\eps(x)\end{equation}
with
$$r_1^\eps(x)=
 |\lambda_0|^{d/2}\eps^{-d/2} \left( R_1^\eps(x)  \Phi_1,
(\pi^{\lambda_0}_{\delta_{\eps^{-1/2} }(x)})^*\Phi_2\right) ,\;\;
R_1^\eps(x)= \eps^{-1/2}  \int_{G}  (\kappa_{x_0\delta_{\sqrt\eps}  x}(w )-\kappa_{x_0}(w)) a(x)
\pi_{ w^{-1}  }^{\lambda_0} dw$$
and
$$r_2^\eps(x)=
 |\lambda_0|^{d/2}\eps^{-d/2} \left( R_2^\eps(x)  \Phi_1,
(\pi^{\lambda_0}_{\delta_{\eps^{-1/2} }(x)})^*\Phi_2\right) ,\;\;
R_2^\eps(x)= \eps^{-1/2}  \int_{G}  \kappa_{x_0\delta_{\sqrt\eps}  x}(w )
A(x, \delta_{\sqrt\eps} w)
\pi_{ w^{-1}  }^{\lambda_0} dw.$$

\begin{lemma} \label{l:unifbounded}
The families $(r_1^\eps)_{\eps>0}$ and $(r_2^\eps)_{\eps>0}$ are uniformly bounded in $L^2(G)$.
\end{lemma}
Applying \eqref{e:carac} to the first term in the right hand side of \eqref{e:devQ}, we see that Lemma \ref{l:unifbounded} implies Lemma \ref{lem:pseudo}.
\end{proof}

\begin{proof}[Proof of Lemma \ref{l:unifbounded}]
The idea is that, for $j=1,2$, there holds  $r_j^\eps(x)= \eps^{-d/2}  \widetilde r_j^\eps(\delta_{\eps^{-1/2}} (x_{\mathfrak v}), x_{\mathfrak z}, x)$ with
$$y\mapsto \widetilde r_j^\eps (y_{\mathfrak v}, y_{\mathfrak z}, x)$$
that is in $L^2(G)$, uniformly with respect to $\eps$, with continuity of the map $x\mapsto \widetilde r_j^\eps (\cdot,\cdot, x)$.

 \medskip

With this idea in mind, we write, for $j=1,2$,
\begin{align}
\nonumber
\| r_j^\eps\|_{L^2(G)}^2 &
  =   |\lambda_0|^{d}\eps ^{-d} \int_G   \left| \left( R_j^\eps(x)  \Phi_1,
(\pi^{\lambda_0}_{\delta_{\eps^{-1/2} }(x)})^*\Phi_2\right)\right|^2 dx\\
\label{eq:boundtoto}
&= |\lambda_0|^{d} \int_G\left|\left( R_j^\eps(\delta_{\eps^{1/2}} (x_{\mathfrak v})x_{\mathfrak z})  \Phi_1,
(\pi^{\lambda_0}_{x_{\mathfrak v}})^*\Phi_2\right)\right|^2 dx_{\mathfrak v} dx_{\mathfrak z}.
\end{align}

Let us first deal with $r_1^\eps$. Writing a Taylor formula, we notice that
\begin{align*}
R_1^\eps(\delta_{\eps^{1/2}} (x_{\mathfrak v})x_{\mathfrak z})&=\eps^{-1/2}  \int_{G}  (\kappa_{x_0\delta_{\eps}(x_{\mathfrak v}) \delta_{\sqrt\eps}(x_{\mathfrak z})}(w)-\kappa_{x_0}(w))a(x) \pi_{ w^{-1}  }^{\lambda_0} dw \nonumber \\
&=\sqrt\eps \int_G B(x,w)a(x)\pi_{ w^{-1}  }^{\lambda_0} dw \nonumber
\end{align*}
where $(x,w)\mapsto B(x,w)$ is continuous and compactly supported in $w$. Therefore $R_1^\eps(\delta_{\eps^{1/2}} (x_{\mathfrak v})x_{\mathfrak z})$ is a bounded operator for any $x\in G$. Since $a$ is compactly supported, it implies that $(r_1^\eps)_{\eps>0}$ is uniformly bounded in $L^2(G)$.

\medskip

Let us now deal with $r_2^\eps$. We are going to use that for all multi-indexes $\alpha\in\N^{2d}$,  the map
\begin{equation} \label{e:map}
x\mapsto x_{\mathfrak v}^\alpha  \left( R_2^\eps(\delta_{\eps^{1/2}} (x_{\mathfrak v})x_{\mathfrak z})  \Phi_1,
(\pi^{\lambda_0}_{x_{\mathfrak v}})^*\Phi_2\right)
\end{equation}
is uniformly bounded and has compact support in $x_{\mathfrak z}$. Let us first prove these properties.

By assumption on the support of $\kappa_x$, we know that the $w$'s contributing to the integral defining $R_2^\eps(x)$ are contained in a compact set (independent of $x$). Then, using \eqref{e:aA} and the fact that $a$ has compact support, we obtain that $R_2^\eps$ has compact support. It follows that the map \eqref{e:map} has compact support in $x_{\mathfrak z}$, i.e., there exists $R_0>0$ such that $|x_{\mathfrak z}|\leq R_0$ for all $x$ that are in the support of $R_2^\eps(\delta_{\eps^{1/2}} (x_{\mathfrak v})x_{\mathfrak z})$. Because of~\eqref{est:Aeps} and because the integral is compactly supported in~$w$, $R_2^\eps(x)$
is a bounded operator for all $x\in G$. Besides, the bound is uniform  since $x$ belongs to a compact set.
Therefore, there exists a constant $C_0>0$ such that
$$\left| \left( R_2^\eps(\delta_{\eps^{1/2}} (x_{\mathfrak v})x_{\mathfrak z})  \Phi_1,
(\pi^{\lambda_0}_{x_{\mathfrak v}})^*\Phi_2\right)\right|\leq C_0 {\bf 1}_{x_{\mathfrak z} \leq R_0} (x).$$

One now wants to prove also decay at infinity in $x_{\mathfrak v}$. For this, we
 use the relations~\eqref{est:growth} and the fact that $\Phi_1$ and $\Phi_2$ are in the Schwartz class to absorb the factor $|x_{\mathfrak v}|$ in the right part of the scalar product. Therefore, for all $\alpha\in \N$,  there exists $C_\alpha$ such that
 $$|x_{\mathfrak v}|^\alpha \left|  \left( R_2^\eps(\delta_{\eps^{1/2}} (x_{\mathfrak v})x_{\mathfrak z})  \Phi_1,
(\pi^{\lambda_0}_{x_{\mathfrak v}})^*\Phi_2\right)\right|\leq C_\alpha {\bf 1}_{x_{\mathfrak z} \leq R_0} (x).$$
As a conclusion,  there exists $C>0$ such that
$$\int_G\left|\left( R_2^\eps(\delta_{\eps^{1/2}} (x_{\mathfrak v})x_{\mathfrak z})  \Phi_1,
(\pi^{\lambda_0}_{x_{\mathfrak v}})^*\Phi_2\right)\right|^2 dx_{\mathfrak v} dx_{\mathfrak z} \leq C \int {\bf 1}_{|x_{\mathfrak z} |\leq R_0} (1+|x_{\mathfrak v}|^2)^{-N} dx_{\mathfrak v} dx_{\mathfrak z}<+\infty$$
by choosing $N$ large enough. This implies the uniform boundedness of the family $(r_2^\eps)$ in $L^2(G)$, which concludes the proof of Lemma \ref{l:unifbounded}.
\end{proof}

Let us now shortly discuss the generalization of this proof in order to obtain an asymptotic expansion at any order, as stated in Remark~\ref{rem:asympexpansion}. The idea is to use a  Taylor expansion at higher order (see Section~3.1.8 of~\cite{FR}). The terms of the expansion~\eqref{expansion:Q} are of the form
$$Q_j(x)= x^\alpha a(x) \int_{G} w^\beta \kappa_{x_0}(w )\pi_{ w^{-1}  }^{\lambda_0} dw$$
where $\alpha$ and $\beta$ are multi-indexes such that the sum of their homogeneous lengths is exactly~$j$.
Denoting by $\Delta_{w^\beta} \sigma(x,\lambda_0)$ the Fourier transform of $w\mapsto w^\beta \kappa_{x_0}(w )$, we obtain
$$Q_j(x)= x^\alpha  a(x) \Delta_{w^\beta} \sigma(x,\lambda_0).$$
Observe that the operator $\Delta_{w^\beta}$ is a difference operator as defined in~\cite{FR}.
It order to justify Remark~\ref{rem:asympexpansion}, one then needs to remark that the rest term produced by the Taylor expansion at order $N$ is of the form
$$r^\eps_N(x)=
 |\lambda_0|^{d/2}\eps^{-d/2} \left( R^\eps_N(x)  \Phi_1,
(\pi^{\lambda_0}_{\delta_{\eps^{-1/2} }(x)})^*\Phi_2\right)$$
and
$$
R^\eps_N(x)=\eps^{-\frac {N+1}2}  \int_{G}  \kappa_{x_0\delta_{\sqrt\eps}  x}(w )
A_{N+1} (x, \delta_{\sqrt\eps} w)
\pi_{ w^{-1}  }^{\lambda_0} dw$$
where $A_{N+1}$ satisfies convenient bounds so that an argument similar to the preceding one can be worked out. We do not develop the argument further because we do not need such a precise estimate for our purpose.

\subsubsection{Semi-classical measure}
We can now deduce~\eqref{eq:WPmeasures} from  Lemma~\ref{lem:pseudo} and the following lemma.

\begin{lemma}\label{lem:WP2}
Let $(x_0,\lambda_0)\in G\times ( \mathfrak  z^*\setminus\{0\})$ $a,b\in\mathcal C^\infty_c(\mathcal B)$ where $\mathcal B$ is a fundamental domain of $M$, and $\Phi_1,\Phi_2, \Psi_1,\Psi_2\in\mathcal S(\R^p)$.
Then
$$\left(WP^\eps_{x_0,\lambda_0} (a,\Phi_1,\Phi_2),WP^\eps_{x_0,\lambda_0} (b,\Psi_1,\Psi_2)\right)_{L^2(M)}=( \Phi_1,\Psi_1)\overline{(\Phi_2,\Psi_2)}
 \int_{G_{\mathfrak z}} a(x_{\mathfrak z}) \overline{b(x_{\mathfrak z}) }  dx_{\mathfrak z}
 +O(\sqrt\eps) $$
\end{lemma}

\begin{proof}
Define  $u^\eps= WP^\eps_{x_0,\lambda_0} (a,\Phi_1,\Phi_2)$ and $v^\eps= WP^\eps_{x_0,\lambda_0} (b,\Psi_1,\Psi_2)$  the wave packets in $G$. We first use that
$$\left(WP^\eps_{x_0,\lambda_0} (a,\Phi_1,\Phi_2),WP^\eps_{x_0,\lambda_0} (b,\Psi_1,\Psi_2)\right)_{L^2(M)}=(u^\eps,v^\eps)_{L^2(G)}.$$
Besides, 
\begin{align*}
(u^\eps,v^\eps)_{L^2(G)} & = |\lambda_\eps|^d \eps^{-p} \int _G a_\eps (x_0^{-1} x) \overline b(x_0^{-1} x) (\pi^{\lambda_\eps}_{x_0^{-1} x}\Phi_1,\Phi_2) \overline{  (\pi^{\lambda_\eps}_{x_0^{-1} x}\Psi_1,\Psi_2) } dx\\
&= |\lambda_0|^d \int_G a\left(\delta_{\sqrt\eps} (x_{\mathfrak v} ) x_{\mathfrak z}\right) 
\overline b  \left(\delta_{\sqrt\eps} (x_{\mathfrak v} ) x_{\mathfrak z}\right) 
(\pi^{\lambda_0}_{x_{\mathfrak v}} \Phi_1,\Phi_2) 
\overline{(\pi^{\lambda_0}_{x_{\mathfrak v}} \Psi_1,\Psi_2) } dx_{\mathfrak v} dx_{\mathfrak z}.
\end{align*}
A Taylor expansion of the map $x\mapsto a(\delta_{\sqrt\eps}(x_{\mathfrak v})x_{\mathfrak z}) \overline{b(\delta_{\sqrt\eps}(x_{\mathfrak v})x_{\mathfrak z}) } $ gives
$$a(\delta_{\sqrt\eps}(x_{\mathfrak v})x_{\mathfrak z}) \overline{b(\delta_{\sqrt\eps}(x_{\mathfrak v})x_{\mathfrak z}) } =
a(x_{\mathfrak z}) \overline{b(x_{\mathfrak z}) }  +\sqrt\eps  \sum_{1\leq j\leq 2d} v_j r_j(x_{\mathfrak z},\delta_{\sqrt\eps}(x_{\mathfrak v}))$$
where $x_{\mathfrak v}={\rm Exp} (\sum_{1\leq j\leq 2d} v_jV_j)$ and with $| r_j(x,w) | \leq C_j$ for some constants $C_j$, $1\leq j\leq 2d$.
We deduce (using~\eqref{est:growth})
\begin{align*}
(u^\eps,v^\eps)_{L^2(G)}&=  |\lambda_0|^d \int_{G_{\mathfrak z}} a(x_{\mathfrak z}) \overline{b(x_{\mathfrak z}) }  dx_{\mathfrak z}
\int_{G_{\mathfrak v}} (\pi^{\lambda_0}_{x_{\mathfrak v}} \Phi_1,\Phi_2) \overline{ (\pi^{\lambda_0}_{ x_{\mathfrak v}} \Psi_1,\Psi_2)} dx_{\mathfrak v} +O(\sqrt\eps)\nonumber \\
&= (\Phi_1,\Psi_1)\overline{(\Phi_2,\Psi_2)}\int_{G_{\mathfrak z}} a(x_{\mathfrak z}) \overline{b(x_{\mathfrak z}) }  dx_{\mathfrak z} +O(\sqrt \eps), \nonumber
\end{align*}
where the second line follows from \eqref{e:intPhi}.
\end{proof}
Here again, the reader will observe that the expansion can be pushed at any order.

\medskip

It follows from Lemma \ref{lem:pseudo} and Lemma \ref{lem:WP2} that
\begin{align*}
&({\rm Op}_\eps(\sigma) WP^\eps_{x_0,\lambda_0}(a,\Phi_1,\Phi_2), WP^\eps_{x_0,\lambda_0}(a,\Phi_1,\Phi_2))\nonumber \\
&\qquad \qquad = (WP^\eps_{x_0,\lambda_0}(a,\sigma(x_0,\lambda_0)\Phi_1,\Phi_2), WP^\eps_{x_0,\lambda_0}(a,\Phi_1,\Phi_2))+O(\sqrt\eps)\nonumber \\
&\qquad \qquad =( \sigma(x_0,\lambda_0)\Phi_1,\Phi_1) \|\Phi_2\|^2\int_{G_{\mathfrak z}}|a(x_{\mathfrak z})|^2dx_{\mathfrak z}  +O(\sqrt\eps)\nonumber
\end{align*}
which concludes the proof of Proposition \ref{prop:WP}.

\subsection{End of the proof of Theorem~\ref{t:main}} By the results of Section \ref{s:proofsufficiency}, we only need to prove that if $T\leq T_{\rm GCC}(\overline{U})$, the observability inequality \eqref{obs} does not hold.

We first note that if the observability inequality \eqref{obs} is satisfied for some $T>0$, then there exists $\delta>0$ such that \eqref{obs} also holds in time $T-\delta$. Indeed, if it were not the case, there would exist $u_0^n\in L^2(M)$ such that $\|u_0^n\|_{L^2(M)}=1$ and
\begin{align*}
1=\|u_0^n\|_{L^2(M)}^2&\geq n\int_0^{T-2^{-n}} \left\|  {\rm e}^{it(\frac12 \Delta_M+\mathbb{V})} u_0^n\right\|^2_{L^2(U)} dt\\
&\geq n\int_0^{T} \left\|  {\rm e}^{it(\frac12 \Delta_M+\mathbb{V})} u_0^n\right\|^2_{L^2(U)} dt-\frac{n}{2^n}.
\end{align*}
due to conservation of energy, and \eqref{obs} would not hold in time $T$.

Therefore, we shall assume in the sequel that $T<T_{\rm GCC}(\overline{U})$.

Let $T<T_{\rm GCC}(\overline{U})$ and $(x_0,\lambda_0)\in G\times (\mathfrak z^*\setminus\{0\})$ such that 
\begin{equation}\label{e:hypfermee}
\text{for all $s\in [0,T]$, \  $\Phi_0^s(x_0,\lambda_0)\notin \overline{U}\times \mathfrak z^*$.}
\end{equation}
Let us chose initial data $u^\eps_0$ in~\eqref{e:Schrod} which is a wave packet in~$M$ with harmonics given by the first Hermite function $h_0$:
$$u^\eps_0=WP^\eps_{x_0,\lambda_0}(a,h_0,h_0).$$
As a consequence, the semi-classical measure of $(u^\eps_0)$ is $\Gamma_0(x,\lambda) d\gamma_0$ with $\Gamma_0$ the orthogonal projector on $h_0$ (this is where we use the fact that $h_0$ is the first Hermite function) and
$$ \gamma_0(x,\lambda)=c \,  \delta(x-x_0)\otimes \delta(\lambda-\lambda_0)$$
where $c=\limsup \| u^\eps_0\|_{L^2(M)}>0 $.
  Let us denote by $u^\eps(t)$ the associated solution, $u^\eps(t)=  {\rm e}^{it(\frac12 \Delta_M+\mathbb{V})} u^\eps_0$. By Proposition~\ref{p:measure0}, any of its semi-classical measures $\Gamma_t d\gamma_t$ decomposes above $G\times \mathfrak z^*$ according to the eigenspaces of $H(\lambda)$ following~\eqref{eq:decomp}.
Moreover, by Proposition~\ref{p:measure0},
the maps $(t,x,\lambda)\mapsto \Gamma_{n,t}(x,\lambda) d\gamma_t(x,\lambda)$ are continuous and satisfy the transport equation~\eqref{transport}.
We deduce  that for $n\not=0$, $\Gamma_{n,t}(x,\lambda)=0$,
\begin{equation}\label{eq:gammat}
\gamma_t (x,\lambda)= c \, \delta\left( x-{\rm Exp} \left(t\frac d2\mathcal Z^{(\lambda)} \right)x_0\right)\otimes \delta(\lambda-\lambda_0)
\end{equation}
and $\Gamma_{0}$ is the orthogonal projector on $h_0$.

As a consequence of the conservation of the $L^2$-norm by the Schr\"odinger equation,
$\| u^\eps(t) \|_{L^2(M)}=\| u^\eps_0\|_{L^2(M)} $. Besides, the $\eps$-oscillation (see Proposition~\ref{prop:eops}) gives that, for the subsequence defining~$\Gamma_t d\gamma_t$,
$$\lim_{\eps \rightarrow 0}\| u^\eps(t) \|^2_{L^2(M)}
= \int_{M\times \widehat G} {\rm Tr} (\Gamma_t(x,\lambda)) d\gamma_t(x,\lambda),\;\;\forall t\in\R.$$
We deduce that we have, for any $t\in\R$,
$$\int_{M\times \widehat G} {\rm Tr} (\Gamma_t(x,\lambda)) d\gamma_t(x,\lambda)=
\int_{M\times \widehat G} {\rm Tr} (\Gamma_0(x,\lambda)) d\gamma_0(x,\lambda).$$
On the other hand, the positivity of the measure $ {\rm Tr} (\Gamma_t(x,\lambda)) d\gamma_t(x,\lambda)$ combined with~\eqref{eq:gammat} gives
\begin{align*}\int_{M\times \widehat G} {\rm Tr} (\Gamma_t(x,\lambda)) d\gamma_t(x,\lambda)\geq
\int_{M\times \mathfrak z^*} {\rm Tr} (\Gamma_t(x,\lambda) )d\gamma_t(x,\lambda)&=
\int_{M\times \mathfrak z^*} {\rm Tr} (\Gamma_0(x,\lambda) )d\gamma_0(x,\lambda)\\
&=\int_{M\times \widehat G} {\rm Tr} (\Gamma_0(x,\lambda)) d\gamma_0(x,\lambda).
\end{align*}
We deduce that $\gamma_t {\bf 1}_{\mathfrak v^*}=0$. Now, using \eqref{e:hypfermee}, there exists a continuous function $\phi:M\rightarrow [0,1]$ such that $\phi( \Phi_0^s(x_0,\lambda_0))=0$ for any $s\in [0,T]$ and $\phi=1$ on $\overline{U}\times \mathfrak z^*$. Using Proposition \ref{prop:eops} for the subsequence defining the semi-classical measure $\Gamma_t d\gamma_t$, we get
\begin{align*}
0\leq \int_0^{T} \int_U | u^\eps(t,x)|^2 dx dt \leq \int_0^{T} \int_M \phi(x)| u^\eps(t,x)|^2 dx dt \Tend{\eps}{0} & \int_0^{T} \int_{M\times \mathfrak z^*} \phi(x)d\gamma_t (x,\lambda) dt=0.
\end{align*}
Therefore, the observability inequality \eqref{obs} cannot hold.
\begin{remark} \label{r:grazing}
As already noticed in the introduction, it can happen that $T_{\rm GCC}(\overline{U})<T_{\rm GCC}(U)$, and in this case, Theorem \ref{t:main} does not say anything about observability for times $T$ such that $T_{\rm GCC}(\overline{U})<T\leq T_{\rm GCC}(U)$. This is due to the possible existence of grazing rays, which are rays which touch the boundary $\partial U$ without entering the interior of $U$. This phenomenon already occurs in the context of the observability of Riemannian waves, as was shown for example in \cite[Section VI.B]{Leb2}. The example given in this paper is the observation of the wave equation in the unit sphere~$\mathbb{S}^2$ from its (open) northern hemisphere: although the GCC condition is violated by the geodesic following the equator, observability holds in time $T>\pi$. Intuitively, even wave packets following this geodesic have half of their energy located on the northern hemisphere.
\end{remark}


\appendix

\section{Representations of $H$-type groups} \label{a:rep}
In this Appendix, we provide a proof of the description \eqref{eq_widehatG} of $\widehat{G}$. This material is standard in non-commutative Fourier analysis, see for example \cite{CG}.

\subsection{The orbits of $\mathfrak g$}
As any group, a nilpotent connected, simply connected Lie group acts on itself by the inner automorphism $i_x: y\mapsto xyx^{-1}$. With this action, one derives the  action of
$G$ on its  Lie algebra $\mathfrak g$ called  the {\it adjoint map}
$$\begin{array} {rccc}
{\rm Ad} :& G &\rightarrow  &{\rm Aut} (\mathfrak g)\\
 & x & \mapsto &{\rm Ad}_x= d(i_x)_{| 1_G},
 \end{array}$$
and its action on $\mathfrak g^*$, the {\it co-adjoint map}
 $$
 \begin{array} {rccc}
{\rm Ad}^* :& G &\rightarrow  &{\rm Aut} (\mathfrak g^*)\\
 & x & \mapsto &{\rm Ad}^*_x
 \end{array}$$
 defined by
$$\forall x\in G, \;\;\forall \ell \in \mathfrak g^*,\;\; \forall Y\in \mathfrak g,\;\;  ({\rm Ad}^*_x \ell )(Y)=\ell({\rm Ad}_x^{-1} Y).$$
 It turns out that the orbits of this action play an important role in the representation theory of the group.
 Let us recall that the orbit of an element $\ell \in \mathfrak g^*$ is the set $\mathcal O_\ell$ defined by
 $$\mathcal O_\ell = \{ {\rm Ad}^*_x(\ell),\;\; x\in G\}.$$
 The next proposition describes the orbits of $H$-type groups.

\begin{proposition}
Let $G$ be a H-type group, then there are only two types of orbits.
\begin{enumerate}
\item[(i)] $0$-th. dimensional orbits. If $\ell \in \mathfrak v^*$, then $\mathcal O_\ell = \{\ell\}$.
\item [(ii)] $2d$-th. dimensional orbits. If $\ell = \omega +\lambda$ with $\omega\in \mathfrak v^*$ and $\lambda\in\mathfrak z^*\setminus\{0\}$, then $\mathcal O_\ell =\mathcal O_\lambda$ and
$$\mathcal O_\lambda = \{ \omega' + \lambda,\;\;\omega'\in \mathfrak v^*\}.$$
\end{enumerate}
\end{proposition}

\begin{proof}
Let $x={\rm Exp} (  V_x+ Z_x)\in G$ and $y={\rm Exp} ( V_y+ Z_y)\in G$. Then
\begin{align*}
i_x (y)& = x y x^{-1} = {\rm Exp} (  V_x+Z_x){\rm Exp} (V_y+Z_y) {\rm Exp} (-  V_x - Z_x)\\
&= {\rm Exp} ( V _y+ Z_y+[V_x,V_y] ).
\end{align*}
 We deduce that if  $Y= V_Y+Z_Y\in\mathfrak g$,
$${\rm Ad}_x^{-1} (Y)= V_Y + Z_Y+[V_x,V_Y].$$
Therefore, if  $\ell=  \omega+ \lambda$ with $\lambda \in \mathfrak z^*$ and $\omega \in \mathfrak v^*$,
$${\rm Ad}_x^*\ell (Y)= \langle \ell,  {\rm Ad}_x^{-1} (Y)\rangle =\langle \omega, V_Y\rangle +\langle \lambda, Z_Y+[V_x,V_Y]\rangle =\langle \omega +J_\lambda( V_x),V_Y\rangle  +\langle \lambda, Z_Y\rangle $$
As a consequence, if $\lambda=0$, ${\rm Ad}_x^*\ell (Y)=\ell(Y)$ for all $Y\in\mathfrak g$. We deduce ${\rm Ad}_x^*\ell=\ell $ for all $x\in G$, which gives the first type of orbits. \\
If now $\lambda\not=0$ and if $\omega'\in \mathfrak v^* $, one can find $V_x\in \mathfrak v$ such that
$$\langle \omega', V\rangle =\langle \omega+J_\lambda( V_x),V\rangle,\;\;\forall V\in \mathfrak v.$$
One deduces that for all $Y\in \mathfrak g$, ${\rm Ad}_x^*\ell (Y)= \ell'(Y)$ with $\ell'= \omega'+\lambda$.
We deduce that any of these~$\ell'$ is in the orbit of $\ell$, which concludes the proof.
 \end{proof}

 Let $\lambda\in \mathfrak z^*\setminus\{0\}$, the sets $\mathfrak p_\lambda\oplus \mathfrak z $ and $\mathfrak q_\lambda \oplus \mathfrak z $ are maximal isotropic  sub-algebras of $\mathfrak g$ for the bilinear map $B(\lambda)$ (with associated endomorphism $J_\lambda$). Such an algebra is said to be a {\it polarizing algebra} of $\mathfrak g$. We shall use these algebras in the next section.

 \subsection{Unitary irreducible representations of $G$} The {\it unitary representations} of a locally compact group are homomorphisms~$\pi$ of $G$ into the group of unitary operators on a Hilbert space that are continuous for the strong topology. The representations for which there is no proper closed $\pi(G)$-invariant subspaces in ${\mathcal H}_\pi$ are called {\it irreducible}. Arbitrary representations can be uniquely decomposed as sums of irreducible representations.

 \medskip

Kirillov theory establishes a one to one relation between the orbits $(\mathcal O_\ell) _{\ell \in \mathfrak g^*}$ and the irreducible unitary representations of $G$ for any nilpotent Lie group which is  connected and locally connected. We shall first explain how one associates to an orbit $\mathcal O_\ell$ a  representation~$\pi_\ell$ (which only depends on the class of the orbit $\mathcal O_\ell$). Then, in the next subsection, we shall explain how the Stone-Von Neumann Theorem implies that  any representation can be associated with an orbit.

\medskip

$\bullet$ Let $\omega \in \mathfrak v ^*$,  the map $\chi_\omega$ defined below is a $1$-dimensional representation of $G$.
$$\begin{array} {rccc}
\chi_\omega :&G & \rightarrow & {\bf S}^1\\
&{\rm Exp} (X) & \mapsto  &{\rm e}^{i\omega(X)}.
\end{array}$$
Note that $\chi_\omega=\pi^{(0,\omega)}$ as defined in~\eqref{eq:0omega}.

$\bullet$ Let $\lambda\in \mathfrak z^*\setminus\{0\}$.  We consider  the polarizing sub-algebra associated with $\lambda$
$$\mathfrak m_\lambda= \mathfrak q_\lambda \oplus \mathfrak z$$
and the subgroup of $G$ defined by $M:={\rm Exp} (\mathfrak m_\lambda)$. Then, if $\ell\in\mathcal O_\lambda$, $\ell([\mathfrak m_\lambda, \mathfrak m_\lambda])=0$, and  the map
$$\begin{array} {rccc}
\chi_{\lambda,M} :&M & \rightarrow & {\bf S}^1\\
&{\rm Exp} (Y) & \mapsto  &{\rm e}^{i\lambda(Y)}.
\end{array}$$
is a one-dimensional representation of $M$. This allows to construct an induced representation $\pi_\lambda$ on $G$ with Hilbert space $\mathfrak p_\lambda\sim L^2(\R^p)$ via the identification of ${\rm Exp}\left(\sum_{j=1}^d \xi_j P^{(\lambda)}_j\right)\in {\rm Exp}(\mathfrak p_\lambda)$ with $\xi=(\xi_1,\cdots,\xi_d)\in \R^d$.
Indeed, let us take $\xi\in \mathfrak p_\lambda$ and $x={\rm Exp} (X)$, with $X=P+Q+Z$ and $P\in\mathfrak p_\lambda$, $Q\in \mathfrak q_\lambda$ and $Z\in \mathfrak z$. We have, by the Baker-Campbell-Hausdorff formula,
$${\rm Exp}(\xi) {\rm Exp} (X)= {\rm  Exp}(Q+Z+[\xi, Q] +\frac 12[P,Q]) {\rm Exp} (\xi+P),$$
with
$$Q+Z+[\xi, Q] +\frac 12[P,Q]\in \mathfrak m_\lambda\;\;\mbox{ and}\;\;\xi+P\in\mathfrak p_\lambda.$$
 Let us denote by $p,q\in\R^d$
the coordinates of $P$ and $Q$ in the bases $(P^{(\lambda)}_j)_{1\leq j\leq d}$ and $(Q^{(\lambda)}_j)_{1\leq j\leq d}$ respectively. Following~\cite{CG}, we define the induced representation by
 $$\pi_\lambda(x) f(\xi)= \chi_\lambda\Bigl({\rm Exp}(Q+Z+[\xi, Q] +\frac 12[P,Q])\Bigr) f(\xi +p).$$
 Using $\lambda([P_j^{(\lambda)},Q_j^{(\lambda)}])= B(\lambda)(P_j^{(\lambda)},Q_j^{(\lambda)})= |\lambda| $, we obtain
$$\pi_\lambda(x) f(\xi)  = {\rm e}^{i \lambda(Z) +\frac i 2 |\lambda| p\cdot q+i|\lambda| \xi\cdot q} f(\xi+p).$$
We can then use the scaling operator $T_\lambda$ defined by
$$T_\lambda f(\xi)= |\lambda|^{d/4} f(|\lambda|^{1/2}\xi)$$
to get
the equivalent representation $\pi^\lambda_x := T_\lambda  ^* \pi_\lambda(x) T_\lambda$ written in \eqref{def:pilambda}.

 \medskip

This inductive process can be generalized to the case of groups presenting more than two strata. For our purpose,
it remains to prove that any irreducible representation is equivalent to one of those, which is a consequence of the Stone-Von Neumann Theorem.

\subsubsection{Stone-Von Neumann Theorem}
 Let us recall the celebrated Stone-Von Neumann theorem (see \cite[Section 2.2.9]{CG} for a proof).
\begin{theorem}
Let $\rho_1$, $\rho_2$ be two unitary representations of $G=\R^d$ in the same Hilbert space $\mathcal{H}$ satisfying, for some $\alpha\neq 0$, the covariance relation
\begin{equation*}
\rho_1(x)\rho_2(y)\rho_1(x)^{-1}=e^{i\alpha x\cdot y}\rho_2(y), \ \ \text{ for all $x,y\in\R^d$}.
\end{equation*}
Then $\mathcal{H}$ is a direct sum $\mathcal{H}=\mathcal{H}_1\oplus \mathcal{H}_2\oplus\ldots$ of subspaces that are invariant and irreducible under the joint action of $\rho_1$ and $\rho_2$. For any $k$, there is an isometry $J_k:\mathcal{H}_k\rightarrow L^2(\R^d)$ which transforms $\rho_1$ and $\rho_2$ to the canonical actions on $L^2(\R^d)$:
\begin{equation*}
[\widetilde{\rho}_1(x)f](\xi)=f(\xi+x), \ \ [\widetilde{\rho}_2(y)f](\xi)=e^{i\alpha y\cdot \xi}f(\xi).
\end{equation*}
For each $\alpha\neq 0$, the canonical pair $\widetilde{\rho}_1, \widetilde{\rho}_2$ acts irreducibly on $L^2(\R^d)$, so $\rho_1, \rho_2$ act irreducibly on each $\mathcal{H}_k$.
\end{theorem}

Let $\pi$ be an irreducible representation of $G$  on $\mathcal{H}_\pi$. Our goal is to prove that it is equivalent either to a $\chi_\omega$ or to a $\pi_\lambda$ of the preceding section. For $Z\in\mathfrak z$, the operators $\pi({\rm Exp}(Z))$ commute will all elements of  $\{\pi_g : g\in G\}$. By Schur's Lemma (see \cite[Lemma 2.1.1]{CG}), they are thus scalar: $\pi_{{\rm Exp}(Z)}=\chi({\rm Exp}(Z))\Id_{\mathcal{H}_\pi}$ where $\chi$ is a one-dimensional representation of the center $Z(G)={\rm Exp}(\mathfrak z)$ of $G$. Then, two cases appear:

$\bullet$ If $\chi\equiv 1$, then $\pi$ is indeed a representation of the Abelian quotient group $G/Z(G)={\rm Exp}(\mathfrak v)$, thus it is one-dimensional and of the form $\chi_\omega$ for some $\omega\in \mathfrak v^*$.

$\bullet$ If $\chi\not\equiv 1$, there is $\lambda\in \mathfrak z^*\setminus \{0\}$ such that $\chi({\rm Exp}(Z))=e^{i\lambda(Z)}$. We keep the notations of \eqref{eqxpqz}, the notations $P=p_1P_1^{(\lambda)}+\ldots+p_dP_d^{(\lambda)}$, $Q=q_1Q_1^{(\lambda)}+\ldots+q_dQ_d^{(\lambda)}$ and $Z=z_1Z_1+\ldots+z_pZ_p$ of the previous section, and we set $p=(p_1,\ldots,p_d)$, $q=(q_1,\ldots,q_d)$ and $z=(z_1,\ldots,z_p)$.
The actions of the $d$-parameter subgroups $\rho_1(p)=\pi_{{\rm Exp}(P)}$ and $\rho_2(q)=\pi_{{\rm Exp}(Q)}$ satisfy the covariance relation
\begin{align}
\rho_1(p)\rho_2(q)\rho_1^{-1}(p)\rho_2^{-1}(q)&=\pi_{{\rm Exp}(\frac12(p_1q_1[P_1^{(\lambda)},Q_1^{(\lambda)}]+\ldots+p_dq_d[P_d^{(\lambda)},Q_d^{(\lambda)}]))} \nonumber \\
&=e^{\frac{i}{2}|\lambda|p\cdot q}\Id_{\mathcal{H}_\pi} \nonumber
\end{align}
where we have used $[P_j^{(\lambda)},Q_j^{(\lambda)}]= \mathcal Z^{(\lambda)}$ with $\lambda(\mathcal Z^{(\lambda)})=|\lambda|$.
The joint action of $\rho_1$ and $\rho_2$ is irreducible since the $d$-parameter subgroups generate $G$ and $\pi$ is irreducible. Thus, we may apply the Stone-Von Neumann theorem, which gives that there exists an isometry identifying $\mathcal{H}_\pi$ with $L^2(\R^d)$ such that the actions take the form
\begin{align}
[\rho_1(p)f](t)=[\pi_{{\rm Exp}(P)}f](\xi)=f(\xi+p), \nonumber \\
[\rho_2(q)f](t)=[\pi_{{\rm Exp}(Q)}f](\xi)=e^{i|\lambda| q\cdot \xi}f(\xi)\nonumber
\end{align}
for all $f\in L^2(\R^d)$ and $p,q\in\R^d$.
Hence, in this model, the action of an arbitrary element of $G$ is
\begin{equation*}
[\pi_{{\rm Exp}(P+Q+Z)}f](\xi)=e^{i\lambda(z)+\frac{i}{2}|\lambda|p\cdot q+i|\lambda|q\cdot \xi}f(\xi+p)
\end{equation*}
since ${\rm Exp}(P+Q+Z)={\rm Exp}(Z+\frac{1}{2}[P,Q])\cdot{\rm Exp}(Q)\cdot{\rm Exp}(P)$ by the Baker-Campbell-Hausdorff formula. This is just the action of $\pi_\lambda$ modeled in $L^2(\R^d)$. Thus, an infinite-dimensional irreducible representation $\pi$ is isomorphic to $\pi_\lambda$ for some $\lambda$.


\section{Pseudodifferential operators and semi-classical measures}\label{sec:proofpseudo}

In this Appendix we focus on different aspects of the pseudodifferential calculus on quotient manifolds.

\subsection{Properties of pseudodifferential operators on quotient manifolds}

We prove here properties (3) to  (7) of Section~\ref{sec:semiclas}.

\medskip

\noindent $\bullet$ {\it Proof of Property (3)}. We write
$G=\cup_{\gamma \in \widetilde{\Gamma}} M \gamma^{-1}$
and, using the periodicity of $f$, we obtain
$$
\int_{  G} \kappa_{ x}^\eps(y^{-1} x) f( y)dy=\sum_{\gamma\in \widetilde{\Gamma}} \int _{y\in M \gamma^{-1}} \kappa_{ x}^\eps(y^{-1} x) f( y)dy= \sum_{\gamma\in \widetilde{\Gamma}} \int _{y\in M }  \kappa_{ x}^\eps(\gamma y^{-1} x) f( y)dy.$$
As a consequence, the action of the operator~${\rm Op}_\eps(\sigma)$ writes  as a sum of convolution
$${\rm Op}_\eps(\sigma)f(x) = \sum_{\gamma\in \widetilde{\Gamma}} f*\kappa^\eps_{x}(\gamma\cdot)  (x). $$

 \medskip

\noindent $\bullet$ {\it Proof of Property (4)}. By Young's convolution inequality
$$\| f*  \kappa^\eps_{x}(\gamma\cdot) \|_{L^2(M)}\leq \| \sup_{x\in M}| \kappa^\eps_{x}(\gamma\cdot) | \|_{L^1(M)} \|  f\|_{L^2(M)} .$$
We have
$$ \| \sup_{x\in M} | \kappa^\eps_{x}(\gamma\cdot)|  \|_{L^1(M)}= \eps^{-Q}\int_M \sup_{x\in M} |  \kappa_{x}(\eps\cdot \gamma y) |dy =\int_{\gamma^{-1}M}\sup_{x\in M} |\kappa_x(y)|dy.$$
Therefore
$$\| {\rm Op}_\eps(\sigma)f\|_{L^2(M)} \leq \|  f\|_{L^2(M)} \sum_{\gamma\in \widetilde{\Gamma}} \int_{\gamma^{-1}M}\sup_{x\in M} |\kappa_x(y)|dy= \|f\|_{L^2(M)}\int_{G} \sup_{x\in M} |\kappa_x(y)|dy,$$
which gives~\eqref{eq:boundedness}

 \medskip

\noindent $\bullet$ {\it Proof of Property (5)}. We argue as for the $L^2$ boundedness and observe that the kernel of ${\rm Op}_\eps(\sigma) - {\rm Op}_\eps(\sigma) \chi$
is the function
$$(x,y)\mapsto \kappa^\eps_x(y^{-1} x)(1-\chi)(y).$$
Writing
$$\kappa^\eps_x(y^{-1} x)(1-\chi(y))= \kappa^\eps_x(y^{-1} x)(1-\chi )(x(y^{-1} x)^{-1})$$
we deduce that we can write the operator ${\rm Op}_\eps(\sigma) -  {\rm Op}_\eps(\sigma) \chi$ as the convolution with an $x$-dependent function:
$$({\rm Op}_\eps(\sigma) -    {\rm Op}_\eps(\sigma) \chi )f(x)= \sum_{\gamma\in \widetilde{\Gamma}}f* \theta^\eps (x,\gamma \cdot) $$
with
$\theta^\eps(x,z)= \eps^{-Q} \kappa_x (\eps\cdot z)(1- \chi)( xz^{-1}).$
Therefore, if $K={\rm supp} \, \sigma$ (where $\chi\equiv 1$), we have
$$\| \sup_{x\in K} \theta^\eps(x,\gamma \cdot) \|_{L^1(M)}\leq \int_M\sup_{x\in K} | \kappa_x(\gamma z) |
|(1-\chi)( x(\eps\cdot (\gamma z))^{-1})| dz.$$
A Taylor formula gives that there exists a constant $c>0$ such that for all $x\in K$,
$$|(1-\chi)( x(\eps\cdot (\gamma z))^{-1})| \leq c\eps^N  |\gamma z|^N .$$
Therefore,
$$\| \sup_{x\in K} \theta^\eps(x,\gamma \cdot) \|_{L^1(M)}\leq  c \eps^N \int_M\sup_{x\in K} | \kappa_x(\gamma z) | |\gamma z|^N dz.$$
We deduce thanks to Young's convolution inequality
\begin{align*}
\| ({\rm Op}_\eps(\sigma)(1-   \chi )f\| _{L^2(M)}
\leq \eps^N c  \| f\|_{L^2(M)}  & \sum_{\gamma\in \widetilde{\Gamma}}  \int_M\sup_{x\in K} | \kappa_x(\gamma z) | |\gamma z|^N dz\\
&= \eps^N c  \| f\|_{L^2(M)}   \int_G\sup_{x\in K} | \kappa_x(z) | |z|^N dz.
\end{align*}

\medskip 

\noindent $\bullet$ {\it Proof of Property (6)}.

\begin{proof}[Proof of Proposition~\ref{prop:symbcal}]
We take $f,g\in L^2(M)$. We use a partition of unity 
$\sum_{1\leq j\leq J} \chi_j= 1_{\mathcal B}$ with $\chi_j\in \mathcal C^\infty_0(G)$, compactly supported in a fundamental domain of $M$ (which depends on $j$). We decompose $\sigma$ as 
$$\sigma(x,\lambda)= \sum_{1\leq j\leq J}\sigma_j(x,\lambda),\;\;\sigma_j(x,\lambda)= \chi_j(x) \sigma(x,\lambda),\;\;(x,\lambda)\in G\times \widetilde G;$$
and we consider $\widetilde \chi_j\in\mathcal C^\infty_0(G)$, real-valued, compactly supported in the same  fundamental domain of $M$ as $\chi_j$ with $\widetilde \chi_j=1$ on the support of $\chi_j$. For proving~\eqref{adjoint}, it is enough to prove it for each of the $\sigma_j$.  Besides,  the symbol $\sigma_j$ and the smooth function $\widetilde \chi_j$ satisfy Point (5) and we have 
$$ {\rm Op}_\eps(\sigma_j)= \widetilde \chi_j{\rm Op}_\eps (\sigma_j)\widetilde \chi_j +O(\eps^N)$$
for $N\in\N$ 
in $\mathcal L (L^2(M))$. We will use this property to transform the relations in $L^2(M)$ into relations in $L^2(G)$: 
\begin{align*}
\left({\rm Op}_\eps (\sigma_j)^* f,g\right)_{L^2(M)}&=  \left(  f, {\rm Op}_\eps (\sigma_j) g\right)_{L^2(M)}= \left( \widetilde \chi _j f, {\rm Op}_\eps (\sigma_j) \widetilde \chi_j g\right)_{L^2(G)}+ O(\eps^N \|\widetilde \chi_j f\|_{L^2(G)}  \|\widetilde\chi_j g\|_{L^2(G)})\\
&= \left({\rm Op}_\eps (\sigma_j)^* \widetilde \chi_j f,\widetilde \chi_j g\right)_{L^2(G)} + O(\eps^N \|\widetilde\chi_j f\|_{L^2(G)}  \|\widetilde\chi_j g\|_{L^2(G)})
\end{align*}
We can now use symbolic calculus in $L^2(G)$ and we obtain by
 Proposition~3.6 of~\cite{FF1}, 
\begin{align*}
\left({\rm Op}_\eps (\sigma_j)^* f,g\right)_{L^2(M)}&=  \left({\rm Op}_\eps (\sigma_j^*) \widetilde \chi_j f,\widetilde \chi_j g\right)_{L^2(G)}
-{\eps} ({\rm Op}_\eps(P^{(\lambda)}\cdot \Delta^\lambda_p \sigma_j^*+Q^{(\lambda)}\cdot \Delta^\lambda_q \sigma_j^*)\widetilde \chi_j  f,\widetilde \chi_j  g)_{L^2(G)} \nonumber \\
&\qquad \qquad+O(\eps^2 \|\widetilde\chi_j f\|_{L^2(G)} \|\widetilde \chi_j g\|_{L^2(G)})\nonumber\\
&=  \left({\rm Op}_\eps (\sigma_j^*) f, g\right)_{L^2(M)} -{\eps} ({\rm Op}_\eps( \widetilde \chi_j (P^{(\lambda)}\cdot \Delta^\lambda_p \sigma_j^*+Q^{(\lambda))}\cdot \Delta^\lambda_q \sigma_j^*)) f,g)_{L^2(M)}\nonumber \\
&\qquad \qquad+O(\eps^2 \| f\|_{L^2(M)} \| g\|_{L^2(M)}) \nonumber
\end{align*}
by~\eqref{eq:localisation}.
We now use that  $\widetilde \chi_j \sigma_j=\sigma_j$, whence   $\widetilde \chi_j \sigma_j^*=\sigma_j^*$  and also 
$\widetilde \chi_j \Delta_p^\lambda \sigma_j = \Delta_p^\lambda \sigma_j$, $\widetilde \chi_j  \Delta_q^\lambda \sigma_j = \Delta_q^\lambda \sigma_j
$. Besides since $\widetilde \chi_j=1$ on the support of $\sigma_j$, we deduce 
$$\widetilde \chi_j (P^{(\lambda)}\cdot \Delta^\lambda_p \sigma_j^*+Q^{(\lambda)}\cdot \Delta^\lambda_q \sigma_j^*)
=P^{(\lambda)}\cdot \Delta^\lambda_p \sigma_j^*+Q^{(\lambda)}\cdot \Delta^\lambda_q \sigma_j^*,$$
whence~\eqref{adjoint}.

\medskip 

Let us now prove~\eqref{composition}.  We argue similarly and  write in $\mathcal L (L^2(M))$
$${\rm Op}_\eps(\sigma_1) = \sum_{1\leq j\leq  J} {\rm Op}_\eps( \chi_j\sigma_1) = \sum_{1\leq j\leq  J}  \widetilde \chi_j{\rm Op}_\eps( \chi_j\sigma_1)\widetilde \chi_j +O(\eps^N)$$
for $N\in\N$. Considering $\underline \chi_j$ smooth, real-valued, compactly supported in a fundamental domain and equal to $1$ on the support of  $\widetilde \chi_j$, 
we have 
$$\widetilde \chi_j {\rm Op}_\eps (\sigma_2)= {\rm Op}_\eps (\widetilde \chi_j\sigma_2)= {\rm Op}_\eps (\widetilde \chi_j\sigma_2)\underline\chi_j +O(\eps^N)$$
in $\mathcal L (L^2(G))$ and 
we deduce  that for $1\leq j\leq J$ 
\begin{align*}
\left({\rm Op}_\eps (\chi_j\sigma_1)\circ {\rm Op}_\eps (\sigma_2) f,g\right)_{L^2(M)}
&= \left({\rm Op}_\eps ( \chi_j\sigma_1)\circ { \rm Op}_\eps (\widetilde \chi_j\sigma_2) \underline \chi_jf, \widetilde \chi_jg\right)_{L^2(G)}
+O(\eps^N\| f\|_{L^2(M)}\| g\|_{L^2(M)}) .
\end{align*} 
By symbolic calculus in $G$ 
\begin{align*}
\left({\rm Op}_\eps (\chi_j\sigma_1)\circ {\rm Op}_\eps (\sigma_2) f,g\right)_{L^2(M)}&
 =  \left({\rm Op}_\eps (  \chi_j \sigma_1 \sigma_2 -\eps r) \underline \chi_jf, \widetilde \chi_j g\right)_{L^2(G)} +O(\eps^2\| f\|_{L^2(M)}\| g\|_{L^2(M)}) 
\end{align*} 
with 
$$\displaystyle{r(x,\lambda)= \Delta_p^\lambda (\chi_j \sigma_1) \cdot P^{(\lambda)}\,(\widetilde \chi_j\sigma_2)+\Delta^\lambda_q (\chi_j\sigma_1) \cdot Q^{(\lambda)}\, (\widetilde \chi_j \sigma_2)= 
\chi_j(\Delta_p^\lambda  \sigma_1 \cdot P^{(\lambda)}\,\sigma_2+\Delta^\lambda_q \sigma_1 \cdot Q^{(\lambda)}\, \sigma_2)}$$
where we have used that $\widetilde \chi_j=1$ on the support of $\chi_j$. 
Summing the contributions in $j$, we obtain 
\begin{align*}
&\left({\rm Op}_\eps (\sigma_1)\circ {\rm Op}_\eps (\sigma_2) f,g\right)_{L^2(M)}\\
&\qquad = \sum_{1\leq j\leq J}  \left({\rm Op}_\eps (  \chi_j (\sigma_1 \sigma_2 -\eps (\Delta_p^\lambda  \sigma_1 \cdot P^{(\lambda)}\,\sigma_2+\Delta^\lambda_q \sigma_1 \cdot Q^{(\lambda)}\, \sigma_2)
)) \underline \chi_jf, \widetilde \chi_j g\right)_{L^2(G)}
 \\&\qquad \qquad
+O(\eps^2\| f\|_{L^2(M)}\| g\|_{L^2(M)}) \\
&\qquad  =  \sum_{1\leq j\leq J}  \left({\rm Op}_\eps (  \chi_j (\sigma_1 \sigma_2 -\eps (\Delta_p^\lambda  \sigma_1 \cdot P^{(\lambda)}\,\sigma_2+\Delta^\lambda_q \sigma_1 \cdot Q^{(\lambda)}\, \sigma_2)
) )f,  g\right)_{L^2(M)} \\
&\qquad \qquad+O(\eps^2\| f\|_{L^2(M)}\| g\|_{L^2(M)})
\end{align*}
because both $\underline \chi_j$ and $\widetilde \chi_j$ are equal to $1$ on the support of $\chi_j$. Finally, using $\sum_{1\leq j\leq J} \chi_j=1$, we obtain 
\begin{align*}
\left({\rm Op}_\eps (\sigma_1)\circ {\rm Op}_\eps (\sigma_2) f,g\right)_{L^2(M)}
& =    \left({\rm Op}_\eps (  \sigma_1 \sigma_2 -\eps (\Delta_p^\lambda  \sigma_1 \cdot P^{(\lambda)}\,\sigma_2+\Delta^\lambda_q \sigma_1 \cdot Q^{(\lambda)}\, \sigma_2))
 f,  g\right)_{L^2(M)} \\
 &\qquad +O(\eps^2\| f\|_{L^2(M)}\| g\|_{L^2(M)}),
\end{align*}
whence the result. 
\end{proof}

\medskip 

\noindent $\bullet$ {\it Proof of Property (7)}.

 \begin{proof}[Proof of Proposition~\ref{prop:cutoffkernel}]
 Here again, we reduce by using a partition of unity to the case of $\sigma$ as in (5) above, with a fundamental domain $\mathcal B$ containing ${\bf 1}_G$.  We introduce the associated function $\chi\in{\mathcal C}_c^\infty(\mathcal B)$ such that $\chi \sigma =\sigma$.  We observe that $\chi\sigma_\eps=\sigma_\eps$ and we use Proposition~3.4 of~\cite{FF1} to write for $f,g\in L^2(M)$,
\begin{align*}
\left({\rm Op}_\eps (\sigma) f,g\right)_{L^2(M)}&= \left({\rm Op}_\eps (\sigma) \chi f,\chi g\right)_{L^2(G)}\\
&= \left({\rm Op}_\eps (\sigma_\eps)\chi f,\chi g\right)_{L^2(G)} +O(\eps ^{N} \|\chi f\|_{L^2(G)} \| \chi g\|_{L^2(G)})\\
&=  \left({\rm Op}_\eps (\sigma_\eps) f,g\right)_{L^2(M)} +O(\eps ^{N} \| f\|_{L^2(M)} \| g\|_{L^2(G)})
\end{align*}
which concludes the proof.
 \end{proof}

\subsection{Time-averaged semi-classical measures}
\label{app:semiclas}

We give here comments about  the proof of Proposition~\ref{p:measure0}. Note that when $\mathbb V=0$, Theorem 2.10(ii)(2) in \cite{FF} implies the statement, except for the continuity of the map $t\mapsto \Gamma_t d\gamma_t$. The key observation is that for any symbol $\sigma\in\mathcal A_0$, 
\begin{equation}\label{eq:sansV}
\frac 1{i\eps} \left[ -\frac{\eps^2}{2}\Delta_M -\eps^2 \mathbb V , {\rm Op}_\eps(\sigma)\right] =  \frac 1{i\eps} \left[ -\frac{\eps^2}{2}\Delta_M , {\rm Op}_\eps(\sigma)\right] + O(\eps)
\end{equation}
in ${\mathcal L}(L^2(G))$ by the boundedness of $\mathbb V$. As a consequence, the results of Theorem 2.10(ii)(2) in \cite{FF} without potential passes to the case with a bounded potential. Note in particular that we do not need any analyticity on the potential. 
The two points of Proposition~\ref{p:measure0} derive from  relation~\eqref{eq:sansV}.

\medskip 
 For (1), using Proposition~\ref{prop:symbcal} and multiplying~\eqref{eq:sansV} by $\eps$, one gets  that for any symbol $\sigma\in\mathcal A_0$ and $\theta\in L^1(G)$,
$$\int_{\R\times G \times \hat G}  \theta(t) {\rm Tr} ([\sigma(x,\lambda) ,H(\lambda)] \Gamma_t(x,\lambda)) d\gamma_t(x,\lambda) dt=0,$$
which implies the commutation of $\Gamma_t(x,\lambda)$ with $H(\lambda)$ and thus the relation~\ref{eq:decomp}.

\medskip 
Let us now prove the transport equation and the continuity property; 
Let $\Pi_n^{(\lambda)}$ be the projector on the $n$-th eigenspace of $H(\lambda)$.
We prove here the continuity of the map $t\mapsto (\Pi_n^{(\lambda)}\Gamma_t {\bf 1}_{\mathfrak z^*}\Pi_n^{(\lambda)},\gamma_t {\bf 1}_{\mathfrak z^*})$. Since $\Pi_n^{(\lambda)}\notin{\mathcal A}_0$, it is necessary to regularize the operator $\Pi_n^{(\lambda)} \sigma(x,\lambda)\Pi_n^{(\lambda)}$ for $\sigma\in\mathcal{A}_0$.
In that purpose, we fix $\chi\in{\mathcal C}^\infty(\R)$ such that $0\leq \chi\leq 1$, $\chi(u)=1$ for on $|u|>1$ and $\chi(u)=0$ for $|u|\leq 1/2$. We  consider   $\sigma\in {{\mathcal A}_0}$ a symbol strictly supported inside a fundamental domain of $M$ and associate with it the symbol
$$\sigma^{(u,n)}(x,\lambda)= \chi(uH(\lambda)) \Pi_n^{(\lambda)} \sigma(x,\lambda)\Pi_n^{(\lambda)},\;\; n\in\N,\;\; u\in (0,1].$$
In view of Corollary~3.9 in~\cite{FF}, this symbol belongs to the class $S^{-\infty}$ of regularizing symbols. Besides, it is also supported inside a fundamental domain of $M$.
Fix $n\in\N$ and  consider the map
$$t\mapsto \left({\rm Op}_\eps( \sigma^{(u,n)}) \psi^\eps(t),\psi^\eps(t)\right):=\ell_{u,\eps}(t)$$
where $\psi^\eps(t)$ is a family of solutions to~\eqref{e:Schrod} for some family of initial data $(\psi^\eps_0)_{\eps>0}$.

\begin{lemma}\label{lem:equicontinuity}
The map $t\mapsto \left({\rm Op}_\eps( \sigma^{(u,n)}) \psi^\eps(t),\psi^\eps(t)\right)$ is equicontinuous with respect to the parameter $\eps\in(0,1)$.
\end{lemma}

We recall that from Theorem~2.5 (i) of~\cite{FF} we have for all $\sigma\in{{\mathcal A}_0}$, $\chi$ and $u$ as above, $\theta\in L^1(\R)$, and $p,p'\in\N$ with $p\not=p'$,
\begin{align}\label{eq:antidiag}
\int_\R \theta(t)  \left({\rm Op}_\eps(\Pi_p \chi(uH(\lambda))\sigma\Pi_{p'}) \psi^\eps(t),\psi^\eps(t)\right) dt = O(\eps)
\end{align}

\begin{proof}
For any symbol $\sigma\in{{\mathcal A}_0}$, we have
\begin{align}\label{eq:d/dt}
\frac d{dt}& \left({\rm Op}_\eps( \sigma) \psi^\eps(t),\psi^\eps(t)\right)  = \frac 1{i\eps^2} \left( \left[{\rm Op}_\eps( \sigma), -\frac {\eps^2} 2 \Delta_M-\eps^2\mathbb{V}\right] \psi^\eps(t),\psi^\eps(t)\right) \nonumber\\
&=  \frac 1{i\eps^2} \left( {\rm Op}_\eps( [\sigma, H(\lambda)] \psi^\eps(t),\psi^\eps(t)\right)
-\frac 1{i\eps} \left( {\rm Op}_\eps( V\cdot \pi^\lambda(V)\sigma) \psi^\eps(t),\psi^\eps(t)\right) \\
&\quad -\frac {1} {2i} \left( {\rm Op}_\eps(\Delta_M \sigma) \psi^\eps(t),\psi^\eps(t)\right) -\frac {1} {i} \left( \left[{\rm Op}_\eps(\sigma),\mathbb{V} \right] \psi^\eps(t),\psi^\eps(t)\right). \nonumber
\end{align}
For $\sigma^{(u,n)}$ (which commutes with $H(\lambda)$) we have
\begin{align*}
\frac d{dt} \ell_{u,\eps}(t) & = \frac 1{i\eps^2} \left( [{\rm Op}_\eps( \sigma^{(u,n)}), -\frac {\eps^2} 2 \Delta_M -\eps ^2 \mathbb V] \psi^\eps(t),\psi^\eps(t)\right)\\
&=
-\frac 1{i\eps} \left( {\rm Op}_\eps( V\cdot \pi^\lambda(V)\sigma^{(u,n)}) \psi^\eps(t),\psi^\eps(t)\right) -\frac {1} {2i} \left( {\rm Op}_\eps(\Delta_M \sigma^{(u,n)}) \psi^\eps(t),\psi^\eps(t)\right) +O(\eps)
\end{align*}
where we used $ [{\rm Op}_\eps( \sigma^{(u,n)}),  \mathbb V] =O(\eps) $ in $\mathcal L(L^2(M))$ by Proposition~\ref{prop:symbcal}.
By Lemma~4.1 in~\cite{FF}, there exists $\sigma_1(x,\lambda)$ such that
\begin{gather}
    V\cdot \pi^\lambda(V)\sigma^{(u,n)}(x,\lambda)= [\sigma_1(x,\lambda), H(\lambda)] \label{e:offdiag5}\\
    (V\cdot \pi^\lambda(V) \sigma_1(x,\lambda)) = \left((n+\frac d2)	i   {\mathcal Z}^{(\lambda)} -   \frac 12 \Delta_M\right) \sigma^{(u,n)}(x,\lambda)	\nonumber
\end{gather}
The proof of these relations is discussed at the end of the proof of Proposition~\ref{p:wpschrod} where we use quite similar properties.
We then write for $t,t'\in\R$,
\begin{align*}
\ell_{u,\eps}(t)-\ell_{u,\eps}(t')  &= -\frac 1 {i\eps} \int_{t'}^t
\left( {\rm Op}_\eps([ \sigma_1,H(\lambda)] )\psi^\eps(s),\psi^\eps(s)\right)ds\\
&\;\;-\frac 1{2i}
\int_{t'}^t
\left( {\rm Op}_\eps(\Delta_M \sigma^{(u,n)}) \psi^\eps(s),\psi^\eps(s)\right)ds +O(\eps |t-t'|).
\end{align*}
Besides, using~\eqref{eq:d/dt} for the symbol $\sigma_1$, we deduce
\begin{align*}
&-\frac 1{i\eps}  \left( {\rm Op}_\eps([ \sigma_1,H(\lambda)] )\psi^\eps(t),\psi^\eps(t)\right)  =-\frac \eps {i} \left([{\rm Op}_\eps(\sigma_1),\mathbb{V}] \psi^\eps(t),\psi^\eps(t)\right)
-\eps  \frac d{dt} \left({\rm Op}_\eps( \sigma_1) \psi^\eps(t),\psi^\eps(t)\right) \\
&\qquad \qquad \qquad \qquad \qquad-\frac 1 i \left({\rm Op}_\eps( V\cdot \pi^\lambda(V) \sigma_1) \psi^\eps(t),\psi^\eps(t)\right)-\frac \eps {2i} \left({\rm Op}_\eps( \Delta_M \sigma_1) \psi^\eps(t),\psi^\eps(t)\right).
 \end{align*}
 This implies
 \begin{align}\nonumber
\ell_{u,\eps}(t)-\ell_{u,\eps}(t')&  = -\frac 1 i \int_{t'}^t \left({\rm Op}_\eps( V\cdot \pi^\lambda(V) \sigma_1) \psi^\eps(s),\psi^\eps(s)\right) ds
-\frac 1{2i}
\int_{t'}^t
\left( {\rm Op}_\eps(\Delta_M \sigma_1) \psi^\eps(s),\psi^\eps(s)\right)ds \\
\nonumber 
&\qquad +O(\eps |t-t'|)\\
\label{derivation}
& = (n+\frac d 2) \int_{t'}^t \left({\rm Op}_\eps(\mathcal Z^{(\lambda)}   \sigma) \psi^\eps(s),\psi^\eps(s)\right) ds
+O(\eps |t-t'|)
\end{align}
which concludes the proof.
\end{proof}
The continuity of the map $t\mapsto (\Pi_n^{(\lambda)}\Gamma_t {\bf 1}_{\mathfrak z^*}\Pi_n^{(\lambda)},\gamma_t {\bf 1}_{\mathfrak z^*})$ follows from Lemma \ref{lem:equicontinuity} and the Arzel\`a-Ascoli theorem. Note that, equation~\eqref{derivation} of the proof of Lemma~\ref{lem:equicontinuity}  also implies the transport equation~\eqref{transport}.

\medskip 
Finally, let us prove Point (2) of Proposition~\ref{p:measure0}. We use the relation
\begin{equation*}
\frac{1}{\eps}[-\eps^2\Delta_M,{\rm Op}_\eps(\sigma)]=\frac{1}{\eps}{\rm Op}_\eps([H(\lambda),\sigma])-2{\rm Op}_\eps(V\cdot \pi^\lambda(V)\sigma)-\eps{\rm Op}_\eps(\Delta_M\sigma).
\end{equation*}
together with~\eqref{eq:sansV}.  We denote by~$\varsigma_t$ the scalar measure $\Gamma_t d\gamma_t{\bf 1}_{\mathfrak v^*}$ and we use that for the finite dimensional representations $\pi^{(0,\omega)}$, we have $\pi^{(0,\omega)}(V_j)=i\omega_j$. In the limit $\eps\rightarrow 0$, we obtain  that for any $\theta\in L^1(\R)$ and $\sigma\in \mathcal A_0$ commuting with $H(\lambda)$, 
$$\int_{\R\times M\times \mathfrak z^*} \theta(t)
{\rm Tr}(V\cdot \pi(V) \sigma(x,\lambda)\Gamma_t (x,\lambda)) d\gamma_t (x,\lambda)dt+ \int_{\R\times M\times \mathfrak v^*}\theta(t)i \omega\cdot V\sigma(x,\omega) d\varsigma_t(x,\omega) dt=0.$$
Since $\Gamma_t$ commutes with $H(\lambda)$ and $V\cdot \pi(V)\sigma$ is off-diagonal when $\sigma$ is diagonal (see~\eqref{e:offdiag5}), we deduce that the first term of the left-hand side of the preceding relation is $0$. Therefore,
$$ \int_{\R\times M\times \mathfrak v^*} \theta(t)\omega\cdot V\sigma(x,\omega) d\varsigma_t(x,\omega) dt=0,$$
which implies the invariance of $\varsigma_t(x,\omega)$ by the map $(x,\omega)\mapsto ({\rm Exp}(s\omega\cdot V) x,\omega)$, $s\in\R$.


\section{Wave packet solutions to the Schr\"odinger equation} \label{a:wpsolutions}

We assume here $\mathbb V=0$. We prove  that the solution of~\eqref{e:Schrod} with an initial datum which is a wave packet can be approximated by a wave packet. We focus on the case where the harmonics verify $\Phi_1=\Phi_2=h_0$, see the discussion preceding Remark \ref{r:profiles} for more details. We work in $G$, keeping in mind that by Remark~\ref{correspGM}, the result extends to~$M$. Note that the results of this section give in particular a second proof of the necessary part of Theorem \ref{t:main} in case $\mathbb{V}=0$.

\begin{proposition} \label{p:wpschrod}
Let $u^\eps(t)$ be the solution of equation~\eqref{e:Schrod} with $\mathbb V=0$ and initial data of the form
$$u^\eps_0= WP^\eps_{x_0,\lambda_0} (a,h_0,h_0), $$
where $(x_0,\lambda_0)\in M\times (\mathfrak z^*\setminus\{0\}) $, $a\in{\mathcal S}(G)$ and $h_0$ is the first Hermite function.
Then, there exists a map $(t,x)\mapsto a(t,x)$ in $\mathcal C^1(\R, {\mathcal S}(G))$ such that for all $k\in \N$,
$$u^\eps(t,x) = WP^\eps_{x(t),\lambda_0} (a(t,\cdot),h_0,h_0)+O(\sqrt\eps)$$
in $\Sigma^k_\eps$ (see \eqref{e:Sigmakeps} for definition), with
$$x(t)= {\rm Exp}\left( \frac{d}{2} t{\mathcal Z}^{(\lambda_0)} \right) x_0.$$
\end{proposition}
In particular, this proposition means that, contrarily to what happens in Riemannian manifolds, there are wave packet solutions of the Schr\"odinger equation which remain localized even in very long time (of order $\sim 1$ independently of $\eps$). For example, this is not the case for the torus (see~\cite{AM14,BZ}) or semi-classical completely integrable systems (see~\cite{AFM15}).

\medskip

In what follows, we use the notation $\pi^\lambda(X)$ for denoting the operator such that
$$\mathcal F ( Xf )(\lambda) =\pi^\lambda(X)\mathcal F(f) ,\;\;\forall f\in \mathcal H_\lambda$$
where $X\in \mathfrak g$ (recall that $Xf$ is defined in~\eqref{eq:Xf}).
Using an integration by part in the definition of $\mathcal F ( Xf )(\lambda)$ and the fact that $(\pi_x^\lambda)^*=\pi_{-x}^\lambda$, we obtain in particular
\begin{equation} \label{e:derpilambda}
X ( \pi^\lambda_x \Phi_1,\Phi_2)= ( \pi^\lambda(X)\pi^\lambda_x \Phi_1,\Phi_2)
\end{equation}
and, in view of~\eqref{eq:PQWP}, we have
\begin{equation} \label{e:comppilambda}
\pi^\lambda( P_j^{(\lambda)})= \sqrt{|\lambda|}  \partial_{\xi_j}\;\;\mbox{and}\;\; \pi^\lambda( Q_j^{(\lambda)})=i \sqrt{|\lambda|}  {\xi_j}.
\end{equation}
We recall that extending the definition to $-\Delta_G$, we have $\pi^\lambda(-\Delta_G)=  H(\lambda)$ where $H(\lambda)$ is the Harmonic oscillator
\begin{equation} \label{e:compHlambda}
H(\lambda)= |\lambda| \sum_{j=1}^{d} (-\partial_{\xi_j} ^2 +\xi_j^2).
\end{equation}
Of course, we also have the relations
\begin{equation} \label{e:Hlambda}
H(\lambda)= -\sum_{j=1}^{d} \pi^\lambda (V_j)^2 =- \sum_{j=1}^{d} \left(\pi^\lambda (P^{(\lambda)}_j)^2 +\pi^\lambda (Q^{(\lambda)}_j)^2\right).
\end{equation}
In the sequel, in order to simplify notations, since $\lambda=\lambda_0$ is fixed, we write $P_j$ and $Q_j$ instead of $P_j^{(\lambda_0)}$ and $Q_j^{(\lambda_0)}$. We also use the notation $\Pi_n$ instead of $\Pi_n^{(\lambda_0)}$.

\begin{proof}[Proof of Proposition \ref{p:wpschrod}]
We construct a function  $v^\eps(t,x)$  
of the form
\begin{equation} \label{e:veps1}
v^\eps(t,x)=  |\lambda_\eps|^{d/2} \eps^{-p/2} \left( \sigma^\eps(t,\delta_{\eps^{-1/2}}(x_0^{-1} x)) \pi^{\lambda_\eps}_{x_0^{-1} x} h_0,h_0\right),\;\;\lambda_\eps=\frac{\lambda_0}{\eps^2}
\end{equation}
which 
solves for all $t\in\R$,
\begin{equation} \label{e:veps2}
i\partial_t v^\eps +\frac 12 \Delta_g v^\eps= O(\sqrt\eps)
\end{equation}
 in all the spaces $\Sigma^\eps_k$, $k\in\N$.
 More precisely, we look for  $\sigma^\eps(t,x) =\sum_{j=1}^N\eps^{\frac{j}{2}}\sigma_j(t,x)$, for some $N\in\N$ to be fixed later and 
some  maps $(t,x)\mapsto \sigma_j(t,x) $ that are smooth maps from $\R\times G$ to $\mathcal L(L^2(\R^d))$, and we shall require that $\sigma_0(t,x)=a(t,x){\rm Id}$ for some smooth function $a$ satisfying $a(0,x)=a(x)$
(note that, more rigorously,   these operator-valued maps are the values at $\lambda=\lambda_0$ of fields of operators $\sigma_j(t,x,\lambda)$ over the spaces $\mathcal H_\lambda = L^2(\R^d)$ of representations, as the symbols of the pseudodifferential calculus). 
  Then, an energy estimate shows that $u^\eps(t)-v^\eps(t) =O(\sqrt\eps)$ in $L^2(G)$ for all $t\in \R$.
  
  \medskip

In view of~\eqref{e:carac}, it is equivalent to construct a family
$\widetilde v^\eps(t,x)= \eps^{Q/4}  v^\eps(t,x(t) \delta_{\sqrt\eps}( x))$
which satisfies
$$i\eps \partial_t \widetilde  v^\eps -i\frac{d}{2} \mathcal Z^{(\lambda_0)} \widetilde v^\eps  +\frac 12 \Delta_G \widetilde v^\eps= O(\eps \sqrt\eps)$$
and 
\begin{equation} \label{e:expveps}
\widetilde v^\eps (t,x)=  \sum_{j=0}^N \eps^{\frac{j}{2}}
(\sigma_j(t,x)\pi^{\lambda_0}_{\delta_{\eps^{-1/2}}(x)}h_0,h_0), \;\; N\in\N.
\end{equation}
We emphasize that if we look for operators $\sigma_j(t,x)$ which are of finite rank, then, decomposing $\sigma_j(t,x) h_0$ on the Hermite basis, the function $(\sigma_j(t,x)\pi^{\lambda_0}_{\delta_{\eps^{-1/2}}(x)}h_0,h_0)$ is a sum of terms of the form 
$$(a_{j,\beta}(t,x)\pi^{\lambda_0}_{\delta_{\eps^{-1/2}}(x)}h_0,h_\beta),$$
which means that $v^\eps(t)$ satisfying~\eqref{e:veps1} is indeed a sum of wave packets.

\medskip 
Let us now construct the operators $\sigma_j(t,x)$. 
In order to simplify the notations, we set $S_0=
|\lambda_0|^\frac{d}{ 2}$ and
$$\mathcal L =  i \frac{d}{2}  \mathcal Z ^{(\lambda_0)} -\frac 12 \Delta_G.$$
Note that
$$i\frac{d}{2}  \mathcal Z ^{(\lambda_0)}\pi^{\lambda_0}_x= -S_0 \pi^{\lambda_0}_x $$
 and that $S_0$ is such that $H(\lambda_0) h_0= 2S_0 h_0$. We denote by $\Pi_0$ the orthogonal projector on the eigenspace of $H(\lambda_0)$ for the eigenvalue $2S_0$.
For any operator-valued $\sigma(t,x)$, we have the following result:
\begin{align*}
(i\eps\partial_t-\mathcal{L})&(\sigma(t,x)\pi^{\lambda_0}_{\delta_{\eps^{-1/2}}(x)}h_0,h_0)=\frac{S_0}{\eps}(\sigma (t,x)\pi^{\lambda_0}_{\delta_{\eps^{-1/2}}(x)}h_0,h_0)-\frac{1}{2\eps}(\sigma (t,x)H(\lambda_0)\pi^{\lambda_0}_{\delta_{\eps^{-1/2}}(x)}h_0,h_0)\\
&\qquad +\frac{1}{\sqrt\eps}(V\sigma(t,x)\cdot \pi^{\lambda_0}(V)\pi^{\lambda_0}_{\delta_{\eps^{-1/2}}(x)}h_0,h_0)+((i\eps\partial_t-\mathcal{L})\sigma (t,x)\pi^{\lambda_0}_{\delta_{\eps^{-1/2}}(x)}h_0,h_0)
\end{align*}
where  $V\sigma\cdot \Pi^{\lambda_0}(V)= \sum_{j=1}^{2d} V_j\sigma \Pi^{\lambda_0}(V_j)$.
Equivalently, we can write the latter relation under the more convenient form:
\begin{align}
\nonumber
(i\eps\partial_t-\mathcal{L})&(\sigma (t,x)\pi^{\lambda_0}_{\delta_{\eps^{-1/2}}(x)}h_0,h_0)=\frac{1}{2\eps}([H(\lambda_0),\sigma (t,x)]\pi^{\lambda_0}_{\delta_{\eps^{-1/2}}(x)}h_0,h_0)\\
\label{computationcle}
&\qquad +\frac{1}{\sqrt\eps}(V\sigma(t,x)\cdot \pi^{\lambda_0}(V)\pi^{\lambda_0}_{\delta_{\eps^{-1/2}}(x)}h_0,h_0)+((i\eps\partial_t-\mathcal{L})\sigma(t,x)\pi^{\lambda_0}_{\delta_{\eps^{-1/2}}(x)}h_0,h_0).
\end{align}
Therefore, for $\sigma_0=a\in\mathcal{C}^1(\R,\mathcal{S}(G))$ a scalar map, we have
\begin{equation*}
 (i\eps \partial_t - \mathcal{L})(\sigma_0(t,x)\pi^{\lambda_0}_{\delta_{\eps^{-1/2}}(x)}h_0,h_0)=
 (r_0^\eps(t,x)\pi^{\lambda_0}_{\delta_{\eps^{-1/2}}(x)}h_0,h_0)
 \end{equation*}
with
\begin{align}\label{eq:r0}
r_0^\eps(t,x)= &\frac{1}{\sqrt\eps} (V\sigma_0(t,x)\cdot \pi^{\lambda_0}(V)\pi^{\lambda_0}_{\delta_{\eps^{-1/2}}(x)}h_0,h_0)+
((i\eps\partial_t -\mathcal{L})\sigma_0(t,x)\pi^{\lambda_0}_{\delta_{\eps^{-1/2}}(x)}h_0,h_0)
\end{align}
In other words, for any $\sigma_0(t,x)$ which is scalar, the rest term is of order $\eps^{-1/2}$. At the end of the proof, we will specify our choice of $\sigma_0$ in \eqref{e:sigma0}.

\medskip 

We now focus on constructing correction terms in order to compensate the rest term $r_0^\eps(x)$. Note that since $\Pi_0 h_0=h_0$, we also have
\begin{align*}
r_0^\eps(t,x)= &\frac{1}{\sqrt\eps} (\Pi_0 V\sigma_0(t,x)\cdot \pi^{\lambda_0}(V)\pi^{\lambda_0}_{\delta_{\eps^{-1/2}}(x)}h_0,h_0)+
((i\eps\partial_t -\mathcal{L})\sigma_0(t,x)\pi^{\lambda_0}_{\delta_{\eps^{-1/2}}(x)}h_0,h_0)
\end{align*}
The second  term  involves the scalar operator  $(i\eps\partial_t -\mathcal{L})\sigma_0(t,x)$ which commutes with $\Pi_0$  while the first one depends on $\Pi_0 V\sigma_0(t,x)\cdot \pi^{\lambda_0}(V)$ which does not. For constructing $\sigma_1(t,x)$, we use  the computation~\eqref{computationcle} and
 the fact that for symbols $\sigma(t,x)$ that anti-commute with $H(\lambda_0)$, one can find $\theta(t,x)$ such that
$\sigma(t,x)= [H(\lambda_0),\theta(t,x)]$.

\medskip

$\bullet$ {\it Construction of the approximate solution up to $\sqrt\eps$.} We  have already  noticed in Section~\ref{app:semiclas} that  if
$$\theta_0(t,x)= -\frac 1 {2i|\lambda_0|} \sum_{j=1}^{d}\left(P_j \sigma_0(t,x) \pi^{\lambda_0}(Q_j)   - Q_j\sigma_0(t,x) \pi^{\lambda_0}(P_j)\right),$$
we have the following relations that we prove below
\begin{align}\label{relation1}
 V \sigma_0(t,x)\cdot \pi^{\lambda_0}(V  )  & =-  [H(\lambda_0) , \theta_0(t,x)],\\
 \label{relation2}
 \Pi_0 (V \theta_0(t,x)\cdot\pi^{\lambda_0}(V))  \Pi_0 &= \frac12 \Pi_0 \left( i \frac{d}{2} \mathcal Z ^{\lambda_0} \sigma_0(t,x) -\frac 12\Delta_G \sigma_0(t,x) \right)\Pi_0= \frac 12 \Pi_0 \mathcal{L}\sigma_0(t,x).
\end{align}
Therefore, setting
$$\sigma_1(t,x)= 2 \Pi_0 \theta_0(t,x),$$
and using \eqref{computationcle}, we obtain that
\begin{align*}
(i\eps\partial_t-\mathcal{L})&(\sigma_1(t,x)\pi^{\lambda_0}_{\delta_{\eps^{-1/2}}(x)}h_0,h_0)=
-\frac{1}{\eps}( V \sigma_0(t,x)\cdot \pi^{\lambda_0}(V  )  \pi^{\lambda_0}_{\delta_{\eps^{-1/2}}(x)}h_0,h_0)\\
&\qquad +\frac{1}{\sqrt\eps}(\mathcal{L} \sigma_0(t,x) \pi^{\lambda_0}_{\delta_{\eps^{-1/2}}(x)}h_0,h_0)
+((i\eps\partial_t -\mathcal{L})\sigma_1(t,x)\pi^{\lambda_0}_{\delta_{\eps^{-1/2}}(x)}h_0,h_0)
\end{align*}
Therefore, the function
$\widetilde v^\eps_1(t,x)= ((\sigma_0(t,x) +\sqrt\eps \sigma_1(t,x) )\pi^{\lambda_0}_{\delta_{\eps^{-1/2}}(x)}h_0,h_0)$
satisfies in $\Sigma^k_\eps$ the equation
$$
(i\eps\partial_t-\mathcal{L})\widetilde v^\eps_1(t,x)= r^\eps_1(t,x)+O(\eps\sqrt\eps)
$$
with
$$r^\eps_1(t,x) = -\sqrt\eps
(\mathcal{L}\sigma_1(t,x)\pi^{\lambda_0}_{\delta_{\eps^{-1/2}}(x)}h_0,h_0)
+ i\eps (\partial_t \sigma_0 (t,x)\pi^{\lambda_0}_{\delta_{\eps^{-1/2}}(x)}h_0,h_0) .
$$

$\bullet$ {\it Construction of the approximate solution up to $\eps$.} We observe that by construction $\theta_0(t,x)$ and $\sigma_1(t,x)$ anticommute with $H(\lambda_0)$. Therefore, there exists $\sigma_2(t,x)$ such that
\begin{equation}\label{def:sigma2}
\mathcal{L}\sigma_1(t,x) =  \frac  12  [H(\lambda_0), \sigma_2(t,x)],
\end{equation}
and the function $\widetilde v^\eps_2(t,x)= ((\sigma_0(t,x) +\sqrt\eps \sigma_1(t,x) +\eps\sqrt\eps\sigma_2(t,x)  )\pi^{\lambda_0}_{\delta_{\eps^{-1/2}}(x)}h_0,h_0)$
satisfies the equation
$$
(i\eps\partial_t-\mathcal{L})\widetilde v^\eps_2(t,x)= r^\eps_2(t,x)+O(\eps\sqrt\eps)
$$
with
\begin{align*}
r^\eps_2(t,x) &= \eps(V\sigma_2(t,x)\cdot \pi^{\lambda_0}(V)\pi^{\lambda_0}_{\delta_{\eps^{-1/2}}(x)}h_0,h_0)
+ i\eps (\partial_t \sigma_0 (t,x)\pi^{\lambda_0}_{\delta_{\eps^{-1/2}}(x)}h_0,h_0) .
\end{align*}
At this stage of the proof, we observe that by choosing an adequate term $\sigma_3$,  the off-diagonal part of $V\sigma_2\cdot \pi^{\lambda_0}(V)$ can be treated in the same manner than the off-diagonal term $\mathcal{L}\sigma_1$. Finally we are left with
$$\widetilde v^\eps_3(t,x)= ((\sigma(t,x) +\sqrt\eps \sigma_1(t,x) +\eps\sqrt\eps \sigma_2(t,x) +\eps^2 \sigma_3(t,x)  )\pi^{\lambda_0}_{\delta_{\eps^{-1/2}}(x)}h_0,h_0)$$
and the equation
$$
(i\eps\partial_t-\mathcal{L})\widetilde v^\eps_3(t,x)= r^\eps_3(t,x) +O(\eps^{3/2})
$$
with
$$r^\eps_3(t,x)= \eps( (i\partial_t \sigma_0  + \Pi_0 V\sigma_2(t,x)\cdot \pi^{\lambda_0}(V) \Pi_0) \pi^{\lambda_0}_{\delta_{\eps^{-1/2}}(x)}h_0,h_0).$$

\medskip

$\bullet$ {\it Construction of the approximate solution up to $\eps^{3/2}$.}
For concluding the proof, we use the specific form of the term $\Pi_0 V\sigma_2(t,x)\cdot \pi^{\lambda_0}(V) \Pi_0$. We claim, and we prove below, that there exists a selfadjoint differential operator $\widetilde{\mathcal{L}}$ such that
\begin{equation} \label{relation3}
  \Pi_0 V\sigma_2(t,x)\cdot \pi^{\lambda_0}(V) \Pi_0=\widetilde{\mathcal{L}}\sigma_0(t,x)\Pi_0.
\end{equation}
Therefore, it is enough to choose the function $\sigma_0(t,x)$ as the solution of the equation
\begin{equation} \label{e:sigma0}
i\partial_t \sigma_0(t,x) + \widetilde{\mathcal{L}}\sigma_0(t,x)=0  \quad \sigma_0(0,x)=a(x).
\end{equation}

\medskip

$\bullet$ {\it Proof of relations~\eqref{relation1}, \eqref{relation2} and~\eqref{relation3}.}
Let us begin with~\eqref{relation1}. Using \eqref{e:comppilambda} and \eqref{e:compHlambda}, we get that for $1\leq j\leq d$ there holds
$$[H(\lambda_0), \pi^{\lambda_0}(Q_j) ]= 2i|\lambda|\pi^{\lambda_0}(P_j)\;\;\mbox{and}\;\;
[H(\lambda_0), \pi^{\lambda_0}(P_j) ]= -2i|\lambda_0|\pi^{\lambda_0}(Q_j).$$
Therefore
\begin{align*}
[H(\lambda_0), \theta_0] &= - \frac 1{2i|\lambda|} \sum_{j=1}^d (P_j \sigma_0 [H, \pi^{\lambda_0}(Q_j)]-Q_j\sigma_0 [H, \pi^{(\lambda_0)}(P_j)])\\
&= - \sum_{j=1}^{d} (P_j \sigma_0\pi^{\lambda_0}(P_j)+Q_j \sigma_0 \pi^{(\lambda_0)}(Q_j))\\
&= - V\sigma_0\cdot \pi^{\lambda_0}(V)
\end{align*}
which gives~\eqref{relation1}.

\medskip

The relation~\eqref{relation2} is a direct application of  Lemma~B.2 in~\cite{FF} which states that if
$$
T:=
\left(\sum_{j_1=1}^{2d}
V_{j_1} \pi^{\lambda_0} (V_{j_1})\right)\circ
\left(\sum_{j_2=1}^d
\left(P_{j_2} \pi^{\lambda_0} (Q_{j_2}) - Q_{j_2} \pi^{\lambda_0}(P_{j_2})\right)\right),
$$
then
\begin{equation*}\label{eq:PinTPin}
\Pi_n  T \Pi_n  =
|\lambda_0|
\left(   (n+\frac d2) {\mathcal Z}^{(\lambda_0)}
 +\frac i 2 \Delta_G\right) \Pi_n
 \end{equation*}
where $\Pi_n$ denotes the orthogonal projector on ${\rm Vect}(h_\alpha,\;\;|\alpha|=n)$ (recall that $\Pi_n$ depends on $\lambda_0$ since it is defined from $H(\lambda_0)$ but we omit this fact in the notation).
 Note that these relations are nothing but consequences of the elementary properties of the creation-annihilation operators $\partial_{\xi_j}$ and $i\xi_j$.

 \medskip

Let us now prove the claim~\eqref{relation3}.We use the notations of~\cite{FF} and introduce the operators 
$$
R_j := \frac12 (P_j - iQ_j), 
\quad\mbox{and}\quad 
\bar R_j := \frac12 (P_j + iQ_j).
$$
By \eqref{eq:PQWP}, the operators $\pi^{\lambda_0} (R_j)=\frac{\sqrt{|\lambda_0|}}2 (\partial_{\xi_j} +\xi_j)$ and 
$\pi^{\lambda_0} (\bar R_j)=\frac{\sqrt{|\lambda_0|}}2 (\partial_{\xi_j} -\xi_j) $
are the creation-annihilation operators  associated with the harmonic oscillator $H(\lambda_0)$.
The well-known recursive relations of the Hermite functions 
	give for
$\alpha\in \N^d$ and $j=1,\ldots,d$,
 $$
\pi^{\lambda_0} (R_j) h_\alpha
= \frac{\sqrt{|\lambda_0|}}2 
\sqrt{2\alpha_j} h_{\alpha-{\bf 1}_j}
\qquad
\pi^{\lambda_0}(\bar R_j) h_\alpha
= -\frac{\sqrt{|\lambda_0|}}2 
\sqrt{2(\alpha_j+1)} h_{\alpha+{\bf 1}_j}.
$$
In the preceding formula, we use the convention $h_{\alpha-{\bf 1}_j}=0$ as soon as $\alpha_j=0$. Actually, one has $\pi(R_j) h_0=0$.  We will also use the expression of $\Pi_0 \pi(\bar R_j)$ that derives from these formula.

\medskip 

Let us now compute $\sigma_2$. Starting from
$$\sum_{j=1}^d(P_j\pi^{\lambda_0}(Q_j)-Q_j \pi^{\lambda_0}(P_j) ) 
= -2 i \sum_{j=1}^d (R_j\pi^{\lambda_0}(\bar R_j)-\bar R_j \pi^{\lambda_0}(R_j)),
$$
and using $\Pi_0 \pi^{\lambda_0}(\bar R_j)=0$, we obtain
$$
\sigma_1(t,x)=-\frac{2\Pi_0}{|\lambda_0|} \sum_{j=1}^d \bar R_j a(t,x)\pi^{\lambda_0}(R_j).
$$
Therefore $\sigma_1=\Pi_0 \sigma_1 \Pi_1$ can be written
$$\Pi_0 \sigma_1 \Pi_{1} = -\frac 2{|\lambda_0|} \sum_{j=1}^d \bar R_j a(t,x)\Pi_0\pi^{\lambda_0}(R_j) .$$
We deduce from~\eqref{def:sigma2} that 
$$\Pi_0 \sigma_2 \Pi_{1} =- \frac{1}{|\lambda_0|}  \Pi_0\mathcal{L} \sigma_1\Pi_{1}.$$
Therefore  
$$\sigma_2(t,x)= \frac {2} {|\lambda_0|^2 } \sum_{j=1}^d  \mathcal{L} \bar R_j a(t,x)\Pi_0\pi^{\lambda_0}(R_j).
$$
We now use that for any operator-valued $\sigma(t,x)$,
$$V\sigma\cdot \Pi^{\lambda_0}(V)= 2\sum_{k=1}^d (R_k\sigma\pi^{\lambda_0}(\bar R_k)+\bar R_k \sigma\pi^{\lambda_0}(R_k))$$
and we obtain 
\begin{align*}
V\sigma_2 \cdot \Pi^{\lambda_0}(V)&=
\frac {4}{|\lambda_0|^2 } \sum_{j,k=1}^d  (R_k \mathcal{L} \bar R_j a(t,x)\Pi_0\pi^{\lambda_0}(R_j)\pi^{\lambda_0}(\bar R_k)
+ \bar R_k \mathcal{L} \bar R_j a(t,x)\Pi_0\pi^{\lambda_0}(R_j)\pi^{\lambda_0}( R_k)).
\end{align*}
When computing the diagonal part of the operator above or, more precisely $\Pi_0 V\sigma_2 \cdot \Pi^{\lambda_0}(V)\Pi_0$, we use  $\Pi_0\pi(R_j) \pi(\bar R_k) =\Pi_0\pi(\bar R_k) \pi(R_j)=0$ when $j\not=k$ and we find 
$$
\Pi_0 V\sigma_2 \cdot \Pi^{\lambda_0}(V)\Pi_0=
\frac {4}{|\lambda_0|^2 } \sum_{j=1}^d  R_j \mathcal{L} \bar R_j a(t,x)\Pi_0\pi^{\lambda_0}(R_j)\pi^{\lambda_0}(\bar R_j).$$
Using 
 $$R_j\bar R_j=\frac 14 (P_j^2+Q_j^2) +\frac i 4 \mathcal Z^{(\lambda_0)}
\;\; \quad \text{and}\quad \;\; [R_j,\bar R_j]=\frac{i}{2}\mathcal Z^{(\lambda_0)}, $$
we obtain 
$$
R_j\mathcal{L}\bar R_j= (\mathcal L -i\mathcal Z^{(\lambda_0)}) R_j\bar R_j
\quad \text{and} \quad  \Pi_0\pi^{\lambda_0}(R_j)\pi^{\lambda_0}(\bar R_j)=
-\frac{|\lambda_0|} 2 \Pi_0
$$
and therefore
\begin{align*}
\Pi_0 V\sigma_2 \cdot \Pi^{\lambda_0}(V)\Pi_0=& -\frac 2 {|\lambda_0|} \sum_{j=1}^d (\mathcal L -i\mathcal Z^{(\lambda_0)})R_j \bar R_j a\Pi_0\\
=&  -\frac 2 {|\lambda_0|}  (\mathcal L -i\mathcal Z^{(\lambda_0)})(\frac 14 \Delta_G + \frac {id}4 \mathcal Z^{(\lambda_0)}) a \Pi_0\\
=& -\frac 1 {2|\lambda_0|}\left( i\left( \frac d2 -1\right) \mathcal Z^{(\lambda_0)}-\frac 12 \Delta_G\right) (\Delta_G   +id  \mathcal Z^{(\lambda_0)})a \Pi_0
\end{align*}
which concludes the proof of \eqref{relation3} with 
$$\widetilde{\mathcal L}=  -\frac 1 {2|\lambda_0|}\left( i\left( \frac d2 -1\right) \mathcal Z^{(\lambda_0)}-\frac 12 \Delta_G\right) (\Delta_G   +id  \mathcal Z^{(\lambda_0)})$$
that is clearly self-adjoint.
\end{proof}

In case the harmonics of the initial wave packet are no more equal to $h_0$, e.g. 
$$u^\eps_0= WP^\eps_{x_0,\lambda_0} (a,h_\alpha,h_\alpha)$$
 with $\alpha\in\N^d$ of length $n$, the operator $\Pi_n V\sigma_2 \pi(V) \Pi_n$ is not scalar: it is matricial since one must add terms of the form $(b_\beta(t,x) \pi^{\lambda_0}_x h_\alpha, h_\beta)$ for all $\beta\in\N^d$ of length $n$. Equation \eqref{e:sigma0} is then replaced by an equation with values in finite-rank operators. Setting $F(\sigma_0)=\Pi_n V\sigma_2 \pi(V) \Pi_n$, $F$ is a linear map on the set ${\mathcal S}(G, \mathcal L(V_n))$ where $V_n={\rm Vect} (h_\alpha,\;|\alpha|=n)$.  We endow this set of matrix-valued functions  with the scalar product $\langle a, b\rangle =\int_G {\rm Tr}_{  \mathcal L(V_n)} (a(x)\overline b(x)) dx$. Then, one can define two linear maps $\mathbb A$ and $\mathbb S$ such that $F=\mathbb S+\mathbb A$ with $\mathbb S$ self-adjoint, $\mathbb A$ skew symmetric and $\mathbb A\circ \mathbb S=\mathbb S \circ \mathbb A$. Observing that $\sigma_0(0)= a(x) {\rm Id}_{V_n}\in {\rm Ker} \mathbb A $, one then solves $i\partial_t \sigma_0= F(\sigma_0)$ in ${\rm Ker} \mathbb A $, which induces the solution $\sigma_0(t) = {\rm e}^{-it\mathbb S} \sigma_0(0)$. As a conclusion, noticing that the argument would be the same for 
 $$u^\eps_0= WP^\eps_{x_0,\lambda_0} (a,h_\gamma,h_\alpha)$$
 for $\alpha\not=\gamma$, we deduce   the following remark from the linearity of the equation and the  fact that the set of Hermite functions generates $L^2(\R^d)$. 
 \begin{remark} \label{r:profiles}
 The solution to~\eqref{e:Schrod} with $\mathbb V=0$ and initial data which is a wave packet is asymptotic to a wave packet in finite time. 
\end{remark}

\end{document}